\documentclass[12pt,twoside, reqno]{amsart}
\usepackage{hyperref}
\hypersetup{
    colorlinks=true,
    linkcolor=blue,
    filecolor=magenta,      
    urlcolor=cyan,
    pdftitle={Initial Data Identification for Conservation Laws with Spatially Discontinuous Flux},
    pdfpagemode=FullScreen,
    }
\usepackage{enumitem}
\usepackage{tikz}
\usepackage{comment}
\usepackage{amsmath}
\usepackage{bbm}
\usepackage{mathrsfs}
\usepackage{stmaryrd}
\usepackage{subcaption}
\usepackage{amssymb}
\usepackage{mathrsfs}
\usepackage{appendix}
\usepackage{soul}
\usepackage{dirtytalk}
\newtheorem{thm}{Theorem}[section]
\newtheorem{prop}[thm]{Proposition}

\newtheorem{lemma}[thm]{Lemma}
\theoremstyle{definition}
\newtheorem{exmp}[thm]{Example}
\newtheorem{defi}[thm]{Definition}
\theoremstyle{remark}
\newtheorem{remark}[thm]{Remark}
\usepackage[margin=1in]{geometry}
\numberwithin{equation}{section}

\newcommand*\dif{\mathop{}\!\mathrm{d}}

\newcommand{\R}{\mathbb R}
\newcommand{\mc}[1]{\mathcal{#1}}

\newcommand{\mr}[1]{\mathrm{#1}}
\newcommand{\bs}[1]{\boldsymbol{#1}}
\newcommand{\ms}[1]{\mathsf{#1}}
\newcommand{\msc}[1]{\mathscr{#1}}
\newcommand{\ol}[1]{\overline{#1}}

\newcommand{\sabp}{\mc S^{\, [A B]+}}
\newcommand{\sabpt}{\mc S_t^{\, [A B]+}}

\newcommand{\sabpT}{{\mc S}_{\it T}^{\, [{\it A B}\,]+}}
\newcommand{\sabmt}{\mc S_t^{\, [A B]-}}

\newcommand{\sobapt}{\overline{\mc S}_t^{\, [\overline B\, \overline A\,]+}}

\title[Conservation laws with discontinuous flux]{Initial data identification for conservation laws\\ with spatially discontinuous flux}

\author{
Fabio Ancona \and Luca Talamini
}
\thanks{Dipartimento di Matematica ``Tullio Levi-Civita", Universit\`a di Padova, Italy, ~~E-mail: ancona@math.unipd.it, \\ luca.talamini@math.unipd.it}

\begin{document}

\maketitle

\begin{abstract}
We consider a scalar conservation law  
with a spatially discontinuous flux at a single point $x=0$,
and we 
study the initial data identification problem for $AB$-entropy solutions 
associated to an interface connection $(A,B)$.
This problem consists in identifying the set of initial data 
driven by
the corresponding 
$AB$-entropy solution 
to a given target profile~$\omega^T$, at a time horizon $T>0$.
We provide a full characterization of such a set in terms of suitable integral inequalities,
and we establish structural and geometrical properties of this set.
A distinctive feature of the initial set is that it is in general not convex, differently 
from the case of conservation laws with  convex flux
independent on the space variable.
 The results rely on the properties
 of the $AB$-backward-forward evolution operator
 introduced in~\cite{talamini_ancona_attset},
 and on a  proper concept of $AB$-genuine/interface characteristic
for $AB$-entropy solutions provided in this paper.
\end{abstract}

\tableofcontents

\section{Introduction} 
We are concerned with the 
initial value problem for a 
scalar conservation law in one space dimension
\begin{align}
    \label{conslaw}
&u_t+f(x,u)_x=0,
\quad x \in \mathbb{R}, \quad t \geq 0,
\\
\noalign{\smallskip}
    \label{initdat}
    &u(x,0) = u_0(x), \qquad x \in \mathbb{R},
\end{align}
where $u = u(x,t)$ is the state variable, and the flux $f$ is a space discontinuous function of the form
\begin{equation}\label{discflux}
f(x,u) =\begin{cases}
f_l(u), \qquad x < 0, \\
f_r(u), \qquad x > 0\,,
\end{cases}
\end{equation}
with $f_l, f_r: \R\to \R$
 twice continuously differentiable, uniformly  convex maps that satisfy
\begin{equation}
\label{eq:flux-assumption-1}
f_l''(u),\ f_r''(u) \geq a >0\,,
\end{equation}
and (up to a reparametrization)
\begin{equation}
\label{eq:flux-assumption-2}
f_l(0) = f_r(0), \quad f_l(1) = f_r(1)\,.
\end{equation}
We assume also that 
the unique critical points 
$\theta_l$, $\theta_r$ of $f_l, f_r$, respectively,
satisfy
\begin{equation}
\label{eq:flux-assumption-3}
\theta_l \geq 0, \quad \theta_r \leq 1\,.
\end{equation}
It is well known that,
because of the nonlinearity of the equation,
in order to achieve 
global in time existence and uniqueness results for problems of this type one has to consider 
weak distributional solutions
that satisfy 
the classical Kru\v{z}kov entropy inequalities
away from the  flux-discontinuity interface
$x=0$,
augmented by an appropriate \textit{interface entropy condition}
at~$x=0$. 
Here,   we will consider \textit{entropy solutions of $AB$-type}, associated to a so-called \textit{interface connection} $(A,B)$  (cfr~\cite{Adimurthi2005,BKTengquist} and see~\S \ref{subsec:connAB}). Entropy solutions of $AB$-type form an $\mathbf L^1$-contractive semigroup 
on the domain of ${\bf L^\infty}$ initial data~\cite{BKTengquist, Garavellodiscflux}.
Thus, we adopt the semigroup notation 
$u(x,t)\doteq \sabpt u_0 (x)$, $t\geq 0$, $x\in\R$, for the unique $AB$-entropy solution of~\eqref{conslaw}-\eqref{initdat},
for every initial datum $u_0\in {\bf L}^{\infty}(\mathbb R)$.
\medskip

In this paper, we study the initial data identification problem (or inverse design problem) for $AB$-entropy solutions of the equation~\eqref{conslaw}. 
This problem consists in identifying the set of initial data for which the corresponding 
$AB$-entropy solution coincides with a given target profile $\omega^T$, at a time horizon $T>0$.
Observe that we cannot expect to reach 
any desired profile $\omega^T\in {\bf L}^{\infty}(\mathbb R)$.
In fact, even in the case where $f_l=f_r$,
since the work of Ole\u\i nik \cite{oleinik1957} it is well known that, 
because of the uniform convexity of the flux, 
the Kru\v{z}kov entropy conditions imply 
that every entropy weak solution $u$ of~\eqref{conslaw}
must satisfy (in the sense of distributions)
the one-sided Lipschitz estimate
\begin{equation}
    \label{eq:Oleinik}
\partial_x u(\cdot\,, t) \leq \frac{1}{at}, 
\qquad \text{in \ $\msc D^{\prime}$},
\quad \forall~t>0.
\end{equation}
Essentially, the nonlinearity of the flux forces characteristic lines to intersect which, together with the entropy condition, produces a regularizing effect
$L^{\infty}$ to $BV$ encoded in the Ole\u\i nik inequality~\eqref{eq:Oleinik}.
In the case  of equation~\eqref{conslaw} 
with discontinuous flux~\eqref{discflux} where $f_l\neq f_r$, we have shown in~\cite{anconachiri,talamini_ancona_attset} that
the set of reachable profiles at a time $T>0$:
\begin{equation}\label{eq:attsetintro}
    \mc A^{[AB]}(T) \doteq \big\{ \mc S_T^{[AB]+} u_0 \quad \big| \quad u_0 \in \mathbf L^{\infty}(\mathbb R) \big\},
\end{equation}
is characterized in terms of suitable Ole\v{\i}nik-type estimates and
unilateral pointwise constraints. 
Note that a 
``loss of information'' takes place when characteristic lines intersect 
into a shock:  there are infinitely many ways to create the same shock discontinuity at a given time $T$. Therefore the initial data identification problem for this type of equations is highly ill-posed:
multiple initial data can be stirred by~\eqref{conslaw} into the same
attainable profile $\omega^T\in \mc A^{[AB]}(T)$
at  time $T$.
\medskip

Our goal is to characterize and study the 
properties of 
the set 
of  initial data leading to a 
given profile $\omega^T\in \mc A^{[AB]}(T)$
at time $T$:
\begin{equation}\label{eq:initialdataintro}
    \mc I^{[AB]}_T(\omega^T) \doteq \big\{ u_0 \in \mathbf L^{\infty}(\mathbb R) \quad \mid \quad \mc S_T^{[AB]+} u_0 = \omega^T\big\}.
\end{equation}

In the case of 
 conservation laws with flux independent on the space variable, 
the initial data identification problem
was firstly studied for the Burgers equation in~\cite{MR3643881,zuazua2020,zuazua2023}, and next 
 for general uniformly convex flux in~\cite{COLOMBO2020},
where it is fully characterized the initial set of data evolving to a given profile, and it is shown
that such a set is convex. 
Similar results were obtained
in~\cite{ZuazuainverseproblemHJ,ZZreachableset}
for Hamilton-Jacobi equations with convex Hamiltonian,
and in~\cite{colomboperrollazsylla2023}
for  smoothly, space dependent,
conservation laws or Hamilton-Jacobi equations.
\medskip

When the flux is of the form~\eqref{discflux}
with $f_l\neq f_r$, 
the initial data reconstruction problem
is more challenging 
because 
one has to deal with the
richer and more complicated {\it near-interface}
wave structure of $AB$-entropy solutions.
This is
due to the presence in the solution of waves that are reflected or refracted  
through the discontinuity interface $x=0$, as well as of shock discontinuities emerging from the interface
at positive times
(see the analysis in~\cite{talamini_ancona_attset}).
Nonetheless,
we are still able to provide
a full characterization of the initial set~\eqref{eq:initialdataintro} 
by suitable integral inequalities,
for every given  $\omega^T\in \mc A^{[AB]}(T)$, and we  show
that~\eqref{eq:initialdataintro}
shares almost the same geometric and topological
properties of the  initial set for
 conservation laws with uniformly convex flux
 independent on the space variable.
Notably, a distinctive difference 
from the classical smooth case
is the lack of convexity of the 
initial set~\eqref{eq:initialdataintro} 
as shown in the Example~\ref{exmp:nonconvex}.
To 
establish these results 
we will rely on:
\begin{itemize}
[leftmargin=25pt]
\item[-] a suitable definition of \textit{$AB$-backward evolution operator $\mc S^{[AB]-}_T$} given in~\cite{talamini_ancona_attset}, and on the structural properties of the range of $\mc S^{[AB]-}_T$ therein analized; 
\item[-] a proper concept of \textit{$AB$-genuine/interface characteristic}
for $AB$-entropy solutions  which can ``travel'' along the discontinuity interface $x=0$ (see Definition~\ref{interfacecharacteristics}).
\end{itemize}
Given an $AB$-entropy solution~$u$,
a time horizon $T>0$, and a point $x\in\R$,
we will let $\mathcal{C}_0(u,x)$ denote the set of the initial
positions $\zeta(0)$ of the $AB$-genuine characteristics
$\zeta$ for $u$ that
reach the point $x=\zeta(T)$ at time $T$
(cfr.~\eqref{eq:ABchar-set0}). 
We recall that any element $\omega^T\in \mc A^{[AB]}(T)$
admits one-sided limits $\omega^T(x-), \omega^T(x+)$ at every point
$x\in\R$, and that has at most countably many discontinuities (see~\cite{talamini_ancona_attset}).
Then, we summarize the main results of the paper 
in the following

\begin{thm}
\label{thm:main}
Given $\omega^T\in \mc A^{[AB]}(T)$, set
\begin{equation}
    \label{eq:vertexinitialset}
     u_0^* \doteq  \mc S^{[AB]-}_T \omega^T\,,
\end{equation}
and
\begin{equation}
    \label{eq:sol-vertexinitialset}
    u^*(\cdot,t) \doteq \mc S^{[AB]+}_t u^*_0 \qquad \forall~t \in [0,T]\,.
\end{equation}
Then, letting $\mc I^{[AB]}_T(\omega^T)$ be the set defined in~\eqref{eq:initialdataintro}, the following properties hold:
\begin{itemize}
    \item[(i)] $u_0\in \mc I^{[AB]}_T(\omega^T)$
    if and only if, for every point $\overline x$ of continuity of
    $\omega^T$, there exists $\overline y \in \mathcal{C}_0(u^*, \overline x)$ such that
    there hold
    \begin{equation}
\label{condleq}
\int_y^{\ol y} u_0(x) \dif x \leq \int_y^{\overline y}u^*_0(x) \dif x, \qquad \forall \quad y < \min \mathcal{C}_0(u^*,\overline x),
\end{equation}
and 
\begin{equation}
\label{condgeq}
\int_{\overline y}^y u_0(x) \dif x \geq \int_{\overline y}^y u_0^*(x) \dif x, \qquad \forall \quad y > \max \mathcal{C}_0(u^*,\overline x).
\end{equation}
     \item[(ii)]
     The set $\mc I^{[AB]}_T(\omega^T)$ is an infinite dimensional cone  which has vertex $u_0^*$ and is in general not convex.
\end{itemize}
\end{thm}
We will establish further geometric and topological properties of the initial set \eqref{eq:initialdataintro}
besides the ones stated in Theorem~\ref{thm:main}-(ii), which are collected in Theorem~\ref{geometrical properties} stated in \S~\ref{sec:charmain}.
%

\medskip

Initial data identification problems 
are often formulated 
as least square optimization problems
associated to observable states at a final time
(also known in the literature as \textit{data assimilation problems}).
These type of problems
arise naturally in environmental sciences~\cite{BP,bennett,MR2454278,BMPSM,kalnay}, 
but also in life sciences (see~\cite{CFMM} and references therein), to improve the forecast of a model
or to refine numerical simulations.
%
Similar issues, also related to parameter identification problems, arise  in traffic flow modeling~\cite{JD,MR2773421,XCDHMMPCSS}, in batch sedimentation~\cite{MR3042084,MR3349692}, or in petroleum reservoir engineering~\cite{ORL}.

Conservation
laws with spatially discontinuous flux have many relevant applications in physics and engineering including: 
porous media models
with changing rock types (for oil reservoir simulation)~\cite{GimseRisebro1992,GR93};
sedimentation in waste-water treatment plants~\cite{
MR2136036,
MR1462049}; traffic flow dynamics
 with roads of variable width or surface conditions~\cite{Mochon1987};
Saint Venant models of blood flow in endovascular treatments~\cite{MR1901663,Ca};
radar shape-from-shading models~\cite{MR1912069}.


The paper is organized as follows.
\begin{itemize}
[leftmargin=22pt]
    \item [-] In \S~\ref{sec:general} we collect the 
    definitions of interface connection $(A,B)$, of $AB$-entropy solution and of $AB$-backward solution operator.
    \item[-] In \S~\ref{sec:main} we introduce the  $AB$-genuine/interface characteristics 
    and state the main results, Theorem~\ref{initialdataid} (integral inequalities) and Theorem~\ref{geometrical properties} (structural and geometrical properties), which yield Theorem~\ref{thm:main}.
    \item[-] In \S~\ref{sec:charmain} we establish some basic properties enjoyed by
    the $AB$-genuine/interface characteristics.
    \item[-] In \S~\ref{sec:firstthm} we prove Theorem \ref{initialdataid}.
    \item[-] In \S~\ref{sec:secthm} we prove Theorem \ref{geometrical properties} and provide an example of non convex initial set $\mc I^{[AB]}_T(\omega^T)$.
\end{itemize}

\section{Basic definitions and general setting}\label{sec:general}

\subsection{Connections and $AB$-entropy solutions}\label{subsec:connAB}

We recall here the definitions and properties of interface connection and of admissible solution satisfying an interface entropy condition introduced in~\cite{Adimurthi2005}. 
Throughout the paper, for the one-sided limits of a function $u(x)$ we will use the notation
\begin{equation}
    u(x\,\pm)\doteq \lim_{y\to x\,\pm} u(y).
\end{equation}

\begin{defi}[{\bf Interface Connection}]\label{ABsol}
\label{def:connect}
Let $f$ be a flux as in~\eqref{discflux}
satisfying the assumptions~\eqref{eq:flux-assumption-1}-\eqref{eq:flux-assumption-3}.
A pair of values $(A,B)\in \mathbb R^2$ is called a \textit{connection} if
\begin{enumerate}
\item $f_l(A) = f_r(B)$,
\item 
$f'_l(A)\leq 0$, \ $f'_r(B)\geq 0$.
\end{enumerate} 
We will say that a connection $(A,B)$ is  \textit{critical} if 
$f'_l(A)= 0$, or $f'_r(B)= 0$,
i.e. if
$A=\theta_l$ or $B=\theta_r$.
\end{defi}

\begin{figure}
\centering
\begin{tikzpicture}[scale = 0.7]
\draw[->] (-6,0)--(4,0)node[right]{$u$};
\draw[scale = 0.5, domain=-1:5, smooth, variable=\x] plot ({\x}, {(\x-2)*(\x-2)+1});
\draw[scale = 0.5, domain=-9:3, smooth, variable=\x] plot ({\x}, {(0.5*\x+1)*(0.5*\x+2)+1});

\draw (-5,5) node{$f_l$};
\draw (3,5) node{$f_r$};

\draw[very thick, blue] (-3.8,3)--(2.13,3);
\draw[dotted] (-3.8,3)--(-3.8,0)node[below]{$A$};
\draw[dotted] (2.13,3)--(2.13,0)node[below]{$B$};

\draw[dotted] (-1.5,0.3)--(-1.5,0)node[below]{$\theta_l$};
\draw[dotted] (1,0.5)--(1,0)node[below]{$\theta_r$};
\end{tikzpicture}
\caption{An example of connection $(A,B)$ with $f_l, f_r$ strictly convex fluxes}
\label{Connectionpic}
\end{figure}
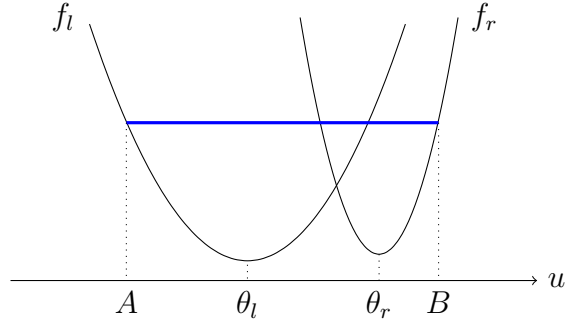

Clearly, condition (1) of Definition~\ref{ABsol} is equivalent to
$A\leq \theta_l$, $B \geq \theta_r$.
For sake of uniqueness,
it is employed in~\cite{BKTengquist} the  {\it adapted entropy} 
\begin{equation}
\label{eq:cABdef}
    \eta_{AB}(x,u)=\big|u-c^{AB}(x)\big|,\qquad\quad
    c^{AB}(x) \stackrel{\cdot}{=}
\begin{cases}
A, & x \leq 0, \\
B, & x \geq 0,
\end{cases}
\end{equation}
to select the unique
solution of the Cauchy problem~\eqref{conslaw}-\eqref{initdat} that satisfies the Kru\v{z}kov-type 
entropy inequality
\begin{equation}\label{adaptedABentropy}
\left|u-c^{AB}\right|_t + \left[\mathrm{sgn}(u-c^{AB})(f(x,u)-f(x,c^{AB}))\right]_x \leq 0, \quad \textrm{in}\ \ \mathcal{D}^{\prime}\,,
\end{equation}
in the sense of distributions, which leads to the following definition.

\begin{defi}[{\bf $AB$-entropy solution}]\label{defiAB}
Let $(A,B)$ be a connection
and let $c^{AB}$ be the function defined in~\eqref{eq:cABdef}. A function $u \in \mathbf L^{\infty}(\mathbb R \times [0,+\infty[)$ is said to be an \textit{$AB$-entropy solution} of the problem \eqref{conslaw},\eqref{initdat} if the following holds:
\begin{enumerate}
[leftmargin=25pt]
\item $u$ is a distributional solution of \eqref{conslaw} on $\R\times \,]0,+\infty[$, that is, for all test functions $\phi \in \mathcal{C}^1_c$
with compact support 
contained in $\mathbb{R}\times \,]0,+\infty[$,
there holds
\begin{equation*}
\int_{-\infty}^{\infty}\int_0^{\infty}  \big\{u \phi_t+f(x,u)\phi_x \big\}\dif x \dif t = 0\,.
\end{equation*}
\item $u$ is a Kru\v{z}kov entropy weak solution of \eqref{conslaw},\eqref{initdat} on $(\mathbb R \setminus \{0\}) \times \,]0,+\infty[$, that is, $t \mapsto u(\cdot,t)$ is a continuous map from $[0,+\infty[\,$ to $\mathbf L^1_{\textup{loc}}(\mathbb R)$,
 the initial condition~\eqref{initdat} is satisfied almost everywhere,
and:
\begin{itemize}
[leftmargin=29pt]
\item[(2.a)]
for any non-negative test function $\phi\in\mathcal{C}^1_c$ with compact support 
contained in $]-\infty,0[\,\times $ $\,]0,+\infty[$, there holds
\begin{equation*}
\int_{-\infty}^0\int_0^{\infty}  \big\{|u-k| \phi_t + \mathrm{sgn}(u-k)\left(f_l(u) - f_l(k)\right) \phi_x \big\}
\dif x \dif t \geq 0, \quad \forall k \in \mathbb R\,;
\end{equation*}
\item[(2.b)]
for any non-negative test function $\phi\in\mathcal{C}^1_c$ with compact support 
contained in $]0,+\infty[\,\times$ $]0,+\infty[$, there holds
\begin{equation*}
\int_0^{\infty}\int_0^{\infty} \big\{
|u-k| \phi_t + \mathrm{sgn}(u-k)\left(f_r(u) - f_r(k)\right) \phi_x
\big\}\dif x \dif t \geq 0, \quad \forall k \in \mathbb R\,.
\end{equation*}
\end{itemize}

\item $u$ satisfies the interface
entropy inequality relative to the connection $(A,B)$, that is, 
for any non-negative test function $\phi\in\mathcal{C}^1_c$ with compact support 
contained in $\R\times ]0,+\infty[$, there holds
\begin{equation*}
\int_{-\infty}^{\infty}\int_0^{\infty} \big\{\left|u-c^{AB}\right| \phi_t + \mathrm{sgn}(u-c^{AB})\left(f(x,u) - f(x,c^{AB})\right)\phi_x \big\}\dif x \dif t \geq 0\,.
\end{equation*}
\end{enumerate}
\end{defi}

%
\begin{remark}
\label{rem:abentr-sol-prop1}
Since
the fluxes $f_l,f_r$ in~\eqref{discflux} are uniformly convex, 
by Property (2) of Definition~\ref{defiAB} it follows
that, if $u$ is
an $AB$-entropy solution, 
then $u(\cdot\,,t)$ is a function of locally bounded variation on $\mathbb R \setminus \{0\}$,
for any~$t>0.$
On the other hand, 
relying on~ \cite{panov,MR1869441},
and because of the $\mathbf L^1_{\textup{loc}}$-continuity of the map $t \mapsto u(\cdot,t)$,
we deduce that
$u$ admits left and right strong traces at $x = 0$ for all $t>0$,  i.e. that 
there exist the one-sided limits  \begin{equation}\label{traces}
u_l(t)\doteq u(0-,t),
\qquad u_r(t)\doteq u(0+,t),
\qquad\forall~t>0\,.
\end{equation}
Moreover, 
by properties (1), (3) of Definition~\ref{defiAB},
and thanks to the analysis in~\cite{talamini_ancona_attset},
we deduce that
$u$ must satisfy at almost any time $t>0$ the
interface conditions
\begin{equation}
\label{ABtraces}
\begin{aligned}
    &f_l(u_l(t)) = f_r(u_r(t)) \geq f_l(A) = f_r(B), \\ 
    \noalign{\smallskip}
&\big(u_l(t) \leq \theta_l, \quad u_r(t) \geq \theta_r\big) \ \ \Longrightarrow \ \ u_l(t) = A, \ \ u_r(t) = B\,.
\end{aligned}
\end{equation}
\end{remark}
\medskip


It was proved in~\cite{Adimurthi2005,BKTengquist} (see also~\cite{MR2807133,Garavellodiscflux})
that  $AB$-entropy solutions of~\eqref{conslaw},\eqref{initdat}
with bounded initial data are unique,
and  $\mathbf{L}^1$-contractive 
with respect to their initial data.
Thus, one can define a semigroup map
\begin{equation}
    \sabp: [0,+\infty[\,\times\,{\bf L^\infty}(\R) \to
   {\bf L^\infty}(\R),
   \qquad\quad (t,u_0)\mapsto 
   \sabpt u_0\,,
   \end{equation}
where the function $u(x,t) \doteq \sabpt u_0(x)$ provides the unique  $AB$-entropy  solution of the Cauchy problem \eqref{conslaw}, \eqref{initdat}.
Such a map is ${\bf L^1}$-stable also with respect to the time $t$ and the
values $A,B$ of the connection, as 
shown in~\cite{talamini_ancona_attset}.

\medskip

\subsection{Backward 
solution operator} 

We  
review here the concept of 
backward solution
operator associated to a connection~$(A,B)$
introduced  in~\cite{talamini_ancona_attset},
referring to~\cite{talamini_ancona_attset}
for further details and properties.

Given a flux  $f$  as in~\eqref{discflux}
satisfying the assumptions~\eqref{eq:flux-assumption-1}-\eqref{eq:flux-assumption-3},
and a connection
 $(A,B)$,
observe that, setting
\begin{equation}
\label{eq:bar-AB-def}
    \overline B \doteq 
    ({f_r}_{|\,]-\infty, \theta_r]})^{-1}\circ f_r(B), \qquad\quad \overline A \doteq ({f_l}_{\mid [\theta_l, +\infty[\,})^{-1}\circ f_l(A)\,,
\end{equation}
where $f_{|\,I}$
denotes the restriction of the function $f$ to the interval $I$,
the pair $(\overline B, \overline A\,)$
provides a connection for the symmetric flux
\begin{equation}
\label{eq:symm-flux}
    \overline{f}(x,u) = \begin{cases}
    f_r(u) , & x \leq 0, \\
    f_l(u), & x \geq 0\,,
    \end{cases}
\end{equation}
(see Figure \ref{Connectionpicsym}).
\begin{figure}[h]
\centering
\begin{tikzpicture}[scale = 0.7]
\draw[->] (-6,0)--(4,0)node[right]{$u$};
\draw[scale = 0.5, domain=-1:5, smooth, variable=\x] plot ({\x}, {(\x-2)*(\x-2)+1});
\draw[scale = 0.5, domain=-9:3, smooth, variable=\x] plot ({\x}, {(0.5*\x+1)*(0.5*\x+2)+1});

\draw (-5,5) node{$f_l$};
\draw (3,5) node{$f_r$};

\draw[very thick, blue] (-3.8,3)--(2.13,3);
\draw[dotted] (-3.8,3)--(-3.8,0)node[below]{$A$};
\draw[dotted] (2.1,3)--(2.1,0)node[below]{$B$};
\draw[dotted] (0.8,3)--(0.8,0)node[below]{$\overline A$};
\draw[dotted] (-0.1,3)--(-0.1,0)node[below]{$\overline B$};

\end{tikzpicture}
\caption{The connection $(\overline B, \overline A)$ of the symmetric flux $\overline f(x, u)$ defined in \eqref{eq:symm-flux}.}
\label{Connectionpicsym}
\end{figure}
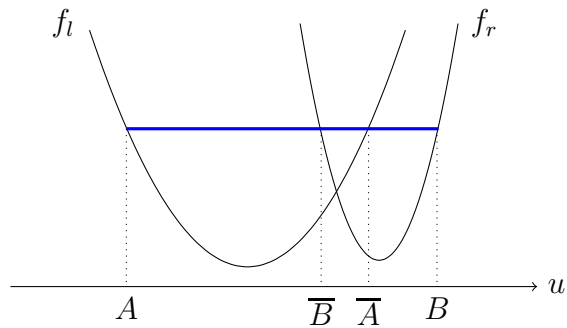
Then, letting
$\sobapt u_0(x)$ denote
the unique $\overline B\,\overline A$-entropy solution
of
\begin{equation}\label{eq:invertedproblem}
    \begin{cases}
        u_t+\overline{f}(x,u)_x = 0 & x \in \mathbb{R}, \quad t \geq 0, \\
            \noalign{\smallskip}
        u(x,0) = u_0(x) &  x \in \mathbb R,
    \end{cases}
\end{equation}
evaluated at $(x,t)$,
we shall define the {\it backward 
solution operator} associated to the connection  $(A,B)$ in terms of the operator $\sobapt$ as follows.

\begin{defi}[{\bf $AB$-Backward solution operator}]\label{def:backop}
Given a connection $(A,B)$, the \textit{backward solution operator}
associated to $(A,B)$ is the map
$\mc S_{(\cdot)}^{\, [A B]-} :
[0,+\infty)\times {\bf L^{\infty}(\R)} \to {\bf L^{\infty}(\R)}$,  defined by
\begin{equation}
     \label{eq:backw-conn-sol-def}
    \sabmt \omega (x) \doteq  \sobapt 
    \big(\omega(-\ \cdot\,)\big)
     (-x)\qquad x\in\R, \ t\geq 0\,. 
\end{equation}
\end{defi}

\smallskip

\section{Statement of the 
main results}\label{sec:main}

In this section we introduce the fundamental concept 
of {\it genuine/interface characteristic}
for an $AB$-entropy solution, 
and we collect the statement of the main results of the paper.

\subsection{Genuine/interface characteristics}\label{subsec:gics}
The definition of {\it genuine/interface characteristic}
extends to the 
setting of $AB$-entropy solutions the classical definition of genuine characteristic for a conservation law $u_t+f(u)_x=0$ (see~\cite{dafermosgenchar,Dafermoscontphysics}).
Throughout the following we fix a time $T>0$,
we
consider a fixed connection $(A,B)$, and we set 
\begin{equation}
\label{eq:fluxinteface}
    \gamma \doteq f_l(A) = f_r(B).
\end{equation}

\begin{defi}[{\bf $AB$-genuine/interface characteristics}]
\label{interfacecharacteristics}
Let $u \in \mathbf L^{\infty}(\mathbb R \times [0,+\infty[\, ; \, \R )$ be an
$AB$-entropy solution of~ \eqref{conslaw}.
We 
say that a Lipschitz continuous function $\zeta:[0,T] \to \mathbb R$
is an \textit{$AB$-genuine/interface characteristic}
(\textit{$AB$-gic}) for $u$ if  the following conditions hold:

\begin{enumerate}
    \item[(i)] for a.e. $t \in [0,T]$ with $\zeta(t) \neq 0$ it holds
    $$\dot{\zeta}(t) = f^{\prime}(u(\zeta(t)-,t),\zeta(t))= f^{\prime}(u(\zeta(t)+,t),\zeta(t));$$
\item[(ii)]  for a.e. $t \in [0,T]$ with $\zeta(t) = 0$, it holds $$f_l(u(\zeta(t)-,t) = \gamma  = f_r(u(\zeta(t)+,t)).$$
\end{enumerate}
\end{defi}

\begin{figure}
    \centering

\tikzset{every picture/.style={line width=0.75pt}} 

\tikzset{every picture/.style={line width=0.75pt}} 

\begin{tikzpicture}[x=0.75pt,y=0.75pt,yscale=-1,xscale=1]

\draw    (226,180) -- (226,33) ;
\draw [shift={(226,30)}, rotate = 90] [fill={rgb, 255:red, 0; green, 0; blue, 0 }  ][line width=0.08]  [draw opacity=0] (5.36,-2.57) -- (0,0) -- (5.36,2.57) -- cycle    ;
\draw    (150,180) -- (333,180) ;
\draw [shift={(336,180)}, rotate = 180] [fill={rgb, 255:red, 0; green, 0; blue, 0 }  ][line width=0.08]  [draw opacity=0] (5.36,-2.57) -- (0,0) -- (5.36,2.57) -- cycle    ;
\draw [line width=1.5]    (296,50) -- (226,91.2) -- (226,130) -- (176,180) ;
\draw    (150,50) -- (336,50) ;
\draw    (236,110) -- (298.09,90.6) ;
\draw [shift={(300,90)}, rotate = 162.65] [color={rgb, 255:red, 0; green, 0; blue, 0 }  ][line width=0.75]    (6.56,-1.97) .. controls (4.17,-0.84) and (1.99,-0.18) .. (0,0) .. controls (1.99,0.18) and (4.17,0.84) .. (6.56,1.97)   ;
\draw    (494,180) -- (494,33) ;
\draw [shift={(494,30)}, rotate = 90] [fill={rgb, 255:red, 0; green, 0; blue, 0 }  ][line width=0.08]  [draw opacity=0] (5.36,-2.57) -- (0,0) -- (5.36,2.57) -- cycle    ;
\draw    (424,180) -- (571,180) ;
\draw [shift={(574,180)}, rotate = 180] [fill={rgb, 255:red, 0; green, 0; blue, 0 }  ][line width=0.08]  [draw opacity=0] (5.36,-2.57) -- (0,0) -- (5.36,2.57) -- cycle    ;
\draw [line width=1.5]    (534,50) -- (494,90) -- (494,120) -- (554,180) ;
\draw    (431,50) -- (574,50) ;
\draw    (490,110) -- (451.79,90.89) ;
\draw [shift={(450,90)}, rotate = 26.57] [color={rgb, 255:red, 0; green, 0; blue, 0 }  ][line width=0.75]    (6.56,-1.97) .. controls (4.17,-0.84) and (1.99,-0.18) .. (0,0) .. controls (1.99,0.18) and (4.17,0.84) .. (6.56,1.97)   ;

\draw (291,27.4) node [anchor=north west][inner sep=0.75pt]  [font=\footnotesize]  {$t_{x}$};
\draw (337,182.4) node [anchor=north west][inner sep=0.75pt]  [font=\footnotesize]  {$x$};
\draw (319,72.4) node [anchor=north west][inner sep=0.75pt]  [font=\scriptsize]  {$f( u_{l}) \ =\ f( u_{r}) \ =\ \gamma $};
\draw (575,182.4) node [anchor=north west][inner sep=0.75pt]  [font=\footnotesize]  {$x$};
\draw (531,27.4) node [anchor=north west][inner sep=0.75pt]  [font=\footnotesize]  {$t_{x}$};

\end{tikzpicture}

    \caption{Example of a member of $\mc C(u, x)$.}
    \label{fig:exchar}
\end{figure}
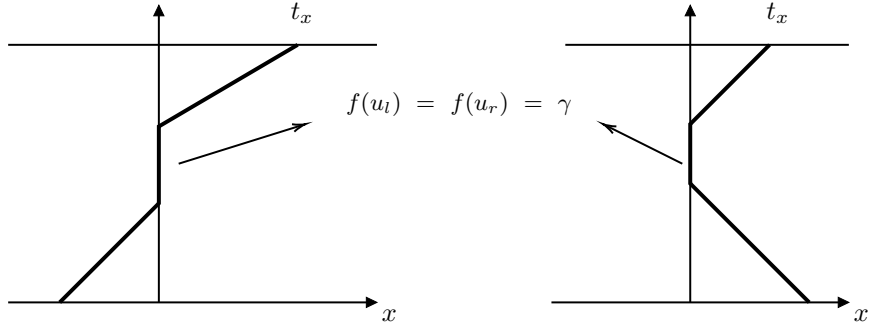

\begin{remark}
    \label{rem:oncharact}
    Applying the classical theory of generalized characteristics~\cite{dafermosgenchar}
    it follows that any 
    $AB$-gic $\zeta \in \mr{Lip}([0,T]\, ; \, \mathbb R)$ is a piecewise affine function 
    for which there exist $0\leq \tau_1\leq \tau_2\leq T$,
    such that:
     \begin{itemize}
      [leftmargin=25pt]
     \item[-] $\zeta(t)=0$ 
     and $f_l(u_l(t))=f_r(u_r(t))=\gamma$, for all $t\in [\tau_1,\tau_2]$;
     \item[-]  the restriction of $\zeta$
     to $[0, \tau_1[$ and to $]\tau_2, T]$
     is either a classical genuine characteristic 
     for the conservation law $u_t+f_l(u)_x=0$\,
     on $\{x<0\}$, or it  
     is a classical genuine characteristic 
     for the conservation law $u_t+f_r(u)_x=0$\,
     on $\{x>0\}$.
     \end{itemize}
    Note that, if $\tau_1=\tau_2\in\{0,T\}$, then
    a curve defined by a function $\zeta$ satisfying the above conditions
    can cross the interface $x=0$ only
    at its starting or terminal points. Thus, in this case
    $\zeta$
    is a a classical genuine characteristic 
     for $u_t+f_l(u)_x=0$ 
    on $\{x<0\}$ or for $u_t+f_r(u)_x=0$\,
     on $\{x>0\}$, in the whole interval $]0,T[$.

    \end{remark}

    \begin{remark}
    \label{rem:oninterface-charact}
    By Definition \ref{interfacecharacteristics}
    an $AB$-gic
    can ``travel" along the discontinuity interface $x = 0$ in an interval $[\tau_1, \tau_2]$ only if in such an interval the flux of the solution is the minimum possible, i.e. if $f_l(u_l(t))=f_r(u_r(t)) = f_l(A)=f_r(B)$
    for all $t\in [\tau_1, \tau_2]$,
    with $u_l, u_r$ as in~\eqref{traces}.
    This definition can be motivated by the following observation. 
    Let $f$ be a smooth
   convex flux. Then,  
    relying on the  inequality 
$$
f(u) -f(v) -f^{\prime}(u) (v-u) \geq 0 \qquad \forall \; u, v \in \mathbb R,
$$
    one can verify that
    a classical genuine characteristic $\zeta:[0,T] \to \mathbb R$ for a solution $u$ of 
    the conservation law $u_t+f(u)_x=0$
    satisfies at  a.e. $t \in [0,T]$ the equality
$$f(u(\zeta(t), t)) - \dot \zeta (t) u(\zeta(t), t) = \min_{v \in \mathbb R} \big\{f(v) - \dot \zeta (t) v\big\}.
$$
Therefore if the characteristic is ``vertical" (i.e. $\dot{ \zeta} = 0$) we simply obtain 
$$
f(u(\zeta(t), t)) = \min_{v \in \mathbb R} f(v),
\quad\ \text{for a.e.}\ \ t\in[0,T].
$$
In 
view of the interface constraint~\eqref{ABtraces}
for $AB$-entropy solutions, it is then natural to require in this setting that a ``characteristic" 
lying on the interface $x=0$
be called 
``genuine" only if it minimizes the admissible flux
at the interface, i.e. if it
satisfies condition~(ii) of 
Definition~\ref{interfacecharacteristics}.
\end{remark}

Next,  given an \textit{$AB$-entropy solution} $u$ of~ \eqref{conslaw},
we consider
the set of $AB$-gic
passing through a point $x\in\R$ at time $t=T$, and the set of the corresponding initial points at time $t=0$,
setting:
\begin{equation}
\label{eq:ABchar-set}
\begin{aligned}
    \mathcal{C}(u,x)
    &\doteq \Big\{\zeta\in \mr{Lip}([0,T]\, ; \, \mathbb R)
    \ | \ \zeta(T)=x \ \ \text{and}
    \ \ \zeta \ \text{is an \, $AB$-gic\, 
    for $u$}
    \Big\},
\end{aligned}
\end{equation}
and 
\begin{equation}
\label{eq:ABchar-set0}
\mathcal{C}_0(u,x) \doteq \Big\{\zeta(0) \; \mid \;  \zeta \in \mathcal{C}(u,x)\Big\}.
\end{equation}
%


The set $\mathcal{C}_0(u,x)$ is a fundamental tool to  analyze the set of initial 
data leading to an attainable profile
$\omega^T$.
To this end, throughout this section we consider the initial
datum $u_0^*$ in~\eqref{eq:vertexinitialset} 
    defined as the image of $\omega^T$
    through the backward solution operator $\mc S^{[AB]-}_T$,
    and we let $u^*(x,t)$ denote the corresponding $AB$-entropy solution 
    with initial datum $u_0^*$, defined in~\eqref{eq:sol-vertexinitialset}.
    Moreover, we let $\mc A^{[AB]}(T)$ denote the set of reachable profiles at time $T>0$ defined in~\eqref{eq:attsetintro}.
    We recall that any element of
    $\mc A^{[AB]}(T)$
has at most countably many discontinuities (see~\cite{talamini_ancona_attset}).
\medskip

     \subsection{Examples}\label{subsec:examples}
     We consider here different examples of $AB$-entropy solutions $u$ that reach the same attainable profile
     $\omega^T\in\mc A^{[AB]}(T)$ at time $T$, which illustrate  various  structures and properties
     of the sets $\mathcal C(u, x)$, $\mathcal C_0(u, x)$.
     Although we choose a
     relatively simple profile, 
     it gives already the possibility to
     capture the essence and the key points of Definition~\ref{interfacecharacteristics}. 
     Namely, given $L_0<0$, we define
\begin{equation}
\omega_1(x) = \begin{cases}
p & x < L_0, \\
A & L_0 < x < 0, \\
\ol B & 0 < x,
\end{cases}
\end{equation}
choosing 
\begin{equation}
\label{eq:def-p}
    p>\bs v\doteq \bs v[\,L_0,A, f_l],
\end{equation}
where $\bs v[\,L_0,A, f_l]$
denotes the quantity defined in~\cite[\S~3.2]{talamini_ancona_attset}, 
   that satisfies
   \begin{equation}
   \label{eq:ex1-constraint-12}
       A<\bs v
       <\ol A
   \end{equation}
   and $\ol A, \ol B$ are defined as in~\eqref{eq:bar-AB-def}.
   Here we are assuming that the connection $(A,B)$ is not critical. Moreover, we assume that
   \begin{equation}
   \label{eq:assumpt-omega1}
        f'_l(A)<{L_0}/{T}<f'_l(\bs v)\,.
    \end{equation}
Note that, since $f'_r(\overline B)\leq 0$
    it follows that $\ms R= \ms R[\omega_1, f_r]=0$,
    while~\eqref{eq:def-p}, \eqref{eq:assumpt-omega1} imply $\ms L= \ms L[\omega_1, f_l]=L_0$.
    One can readly verify that $\omega_1$ fulfills the conditions (i)-(ii) of~\cite[Theorem 4.7]{talamini_ancona_attset} 
characterizing a class of attainable profiles in $\mc A^{[AB]}(T)$.
By the analysis in~\cite[see Remark~4.5]{talamini_ancona_attset}
it follows that, because of \eqref{eq:assumpt-omega1}, 
any $AB$-entropy solution reaching 
the profile~$\omega_1$ at time~$T$
must necessary contain at least one shock, located in $\{x \leq 0\}$,
that produces at time~$T$
the discontinuity occurring at $x=L_0$.
We shall now briefly describe four
different $AB$-entropy solutions driving~\eqref{conslaw}, \eqref{discflux}
to $\omega_1$ at time $T$, that are represented in Figures~\ref{exampleu*}-\ref{exampleu3}, with the shock curves coloured in red.

\begin{itemize}
    \item In Figure \ref{exampleu*} it is represented the solution $u^*$ defined as in~\eqref{eq:vertexinitialset}-\eqref{eq:sol-vertexinitialset} by $u^*(\cdot, t) = \mc S^{[AB]+}_t\circ \mc S^{[AB]-}_T \omega_1$.
    This solution contains in particular a 
  compression wave that creates a shock discontinuity  at $(L_0,T)$, which is located on the left of a rarefaction wave centered at the point $(L_0-T\cdot f'_l(\bs v),0)$.
 This rarefaction impinges (from the left) on a shock curve emerging from the interface $x=0$,
 at some time $t=\bs\sigma$,
 which has right state equal to $A$.
 The left trace of $u^*$ at $x=0$ is equal to
 $\ol A$ in the interval $[0, \bs\sigma[$\,,
 and it is equal to $A$ in the interval $\,]\bs\sigma, T]$. Instead the right trace of $u^*$ at $x=0$ is always equal to $\ol B$. 
At any point $(x,T)$, $x \in \,]L_0, 0[\,$,
we can trace
 a unique backward genuine characteristic with slope $f_l^{\prime}(A)$, which meets
the interface $x=0$ at time $t=T-x/f'_l(A)$.
We can then define an $AB$-gic prolonging this characteristic on the side $\{x>0\}$ 
with slope $f_r^{\prime}(\ol B)$.
Another possible choice to backward define an $AB$-gic
is to travel along the interface $x=0$ until some time $\tau$, and then to prolong it
either on the right
(again with slope $f_r^{\prime}(\ol B)$), or on the left if $\tau \leq \bs\sigma$
(with slope $f_l^{\prime}(\ol A)$).
Therefore we  have two distinct minimal and maximal polygonal lines in the set $\mathcal{C}(u^*,  x)$, represented by the blue polygonal lines in Figure \ref{exampleu*}, while all the other blue dashed lines are the segmens of the other elements in $\mathcal{C}(u^*,  x)$.
A more detailed description of these sets for
a profile similar  to $\omega_1$ is given in Remark~\ref{rem:nofinerpart}.

\item The solution $u_1$ represented in Figure \ref{exampleu1} contains a shock 
located in $\{x < 0\}$ which has left state $p$ and right state $A$. Here we are assuming that  the corresponding Rankine-Hugoniot speed 
$\lambda_l(p,A)$ satisfy $L_0-T\lambda_l(p,A) < 0$, which is certainly true if we take $p$ 
sufficiently close to $\ol A$.
In this case we cannot have 
$AB$-gics starting at $(x,T)$, $x \in \,]L_0, 0[\,$, that are backward prolonged on the side  $\{x<0\}$  since the left trace of $u_1$ at $x=0$ is always equal to $A$, and $f'_l(A)< 0$
by Definition~\ref{ABsol}.
Hence, 
the set $\mathcal{C}(u_1,  x)$ is smaller 
than in the previous case and  we have $\mathcal{C}(u_1,  x)\subset \mathcal{C}(u^*,  x)$.

\item  In the case of the solution $u_2$ 
 represented in Figure~\ref{exampleu2}, 
 two rarefaction waves coming from both sides impinge on the interface $x=0$ in the time interval $[0, \bs\sigma[\,$. Therefore, in this interval the left trace of $u_2$ at $x=0$ 
 has values $u_{2,l}>\ol A$, while the
 right trace  at $x=0$ 
 has values $u_{2,r}<\ol B$.
 As a consequence, the only $AB$-gic starting at $(x,T)$, $x \in \,]L_0, 0[\,$, that can be backward prolonged on the side  $\{x<0\}$ after traveling on the interface is the one 
 that remains on the interface in the
 time interval $[\bs\sigma, T-x/f'_l(A)]$, 
 and then continues with slope $f'_l(\ol A)$
 on the side  $\{x<0\}$ in the time interval $[0,\bs\sigma]$. Similarly, the leftmost 
 $AB$-gic starting at $(x,T)$, $x \in \,]L_0, 0[\,$, that is backward prolonged on the side  $\{x>0\}$,
 is the one that remains on the interface in the
 time interval $[\bs\sigma, T-x/f'_l(A)]$, 
 and then continues with slope $f'_r(\ol B)$
 on the side  $\{x>0\}$ in the time interval $[0,\bs\sigma]$. We deduce from this analysis
 that, differently from the other cases, 
 here the set $\mathcal{C}_0(u_2,  x)$ is not an interval. 

\item Finally, we consider the solution $u_3$ 
 represented in Figure~\ref{exampleu2}, where
 besides the shock located in $\{x \leq 0\}$
 reaching the point $(x,T)$, there is another
 a shock located in $\{x \geq 0\}$.
 This shock
 emerges from  the interface $x=0$
 at some time $\tau_1$, and is then reabsorbed by the interface at some later time $\tau_2>\tau_1$, due to the interaction with rarefaction and compression waves coming from the right.
 Here we see that, differently from the previous cases, we have $\max \mathcal{C}_0(u_3,  x)
 < \max \mathcal{C}_0(u^*,  x)$. 

\end{itemize}

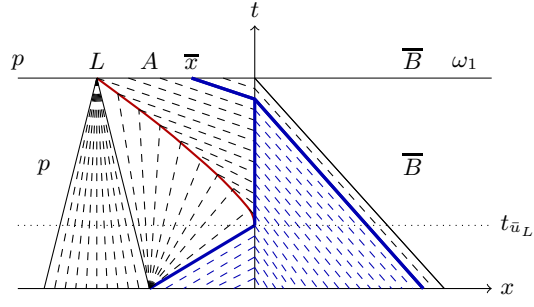
\begin{figure}
\centering
\scriptsize{
\begin{tikzpicture}[scale = 0.7]

\draw (-2,4) node[above]{$A$};
\draw (-4.5,4) node[above]{$p$};
\draw (3,4) node[above]{$\ol B$};

\draw[->] (-4.5,0)--(4.5,0)node[right]{$x$};
\draw[->] (0,0)--(0,5)node[above]{$t$};
\draw[very thin] (-4.5,4)--(4.5,4);
\draw (4,4)node[above]{$\omega_1$};
\draw (-4,2)node[above]{$p$};
\draw (3,2)node[above]{$\ol B$};
\draw (-3,4)node[above]{$L$}--(-4,0);


\draw (-3,4)--(-2,0);
\draw[smooth, red!70!black, thick] plot coordinates{(-3,4)(-1.8,3.1)(-0.7,2.15)(-0.1,1.5)(0,1.2)};
\draw (-1.2, 4)node[above]{$\ol x$};
\draw[blue!70!black, very thick] (0,1.2)--(-2,0);
\draw[dashed] (-3,4)--(-3.8,0);
\draw[dashed] (-3,4)--(-3.6,0);
\draw[dashed] (-3,4)--(-3.4,0);
\draw[dashed] (-3,4)--(-3.2,0);
\draw[dashed] (-3,4)--(-3,0);
\draw[dashed] (-3,4)--(-2.8,0);
\draw[dashed] (-3,4)--(-2.6,0);
\draw[dashed] (-3,4)--(-2.4,0);
\draw[dashed] (-3,4)--(-2.2,0);

\draw[dashed] (-2,0)--(-2.6,3.7);
\draw[dashed] (-2,0)--(-2.2,3.4);
\draw[dashed] (-2,0)--(-1.8,3.1);
\draw[dashed] (-2,0)--(-1.4,2.8);
\draw[dashed] (-2,0)--(-1,2.5);
\draw[dashed] (-2,0)--(-0.7,2.17);
\draw[dashed] (-2,0)--(-0.4,1.9);
\draw[dashed] (-2,0)--(-0.1,1.5);

\draw[dashed]  (-2.6,3.7)--(0,2.8);
\draw[dashed, blue!70!black] (0,2.8)--(2.4,0);
\draw[dashed]  (-2.2,3.4)--(0,2.6);
\draw[dashed, blue!70!black](0,2.6)--(2.2,0);
\draw[dashed]  (-1.8,3.1)--(0,2.4);
\draw[dashed, blue!70!black](0,2.4)--(2,0);
\draw[dashed]  (-1.4,2.8)--(0,2.2);
\draw[dashed, blue!70!black](0,2.2)--(1.8,0);
\draw[dashed]  (-1,2.5)--(0,2);
\draw[dashed, blue!70!black](0,2)--(1.6,0);
\draw[dashed]  (-0.7,2.17)--(0,1.8);
\draw[dashed, blue!70!black](0,1.8)--(1.4,0);
\draw[dashed]  (-0.4,1.9)--(0,1.6);
\draw[dashed, blue!70!black](0,1.6)--(1.2,0);
\draw[dashed]  (-0.1,1.5)--(0,1.4);
\draw[dashed, blue!70!black](0,1.4)--(1,0);
\draw[dashed, blue!70!black] (0,1.2)--(0.82,0);
\draw[dashed, blue!70!black] (0,1)--(0.68,0);
\draw[dashed, blue!70!black] (0,0.8)--(0.55,0);
\draw[dashed, blue!70!black] (0,0.6)--(0.4,0);
\draw[dashed, blue!70!black] (0,0.4)--(0.25,0);
\draw[dashed, blue!70!black] (0,0.2)--(0.12,0);

\draw[dashed, blue!70!black] (0,1)--(-1.7,0);
\draw[dashed, blue!70!black] (0,0.8)--(-1.35,0);
\draw[dashed, blue!70!black] (0,0.6)--(-1.05,0);
\draw[dashed, blue!70!black] (0,0.4)--(-0.8,0);
\draw[dashed, blue!70!black] (0,0.2)--(-0.4,0);

\draw[dashed] (-3,4)--(0,3);
\draw[dashed, blue!70!black] (0,3)--(2.6,0);
\draw[dashed] (-2.4,4)--(0,3.2);
\draw[dashed, blue!70!black] (0,3.2)--(2.8,0);
\draw[dashed] (-1.8,4)--(0,3.4);
\draw[dashed, blue!70!black](0,3.4)--(3,0);
\draw[blue!70!black, very thick] (-1.2,4)--(0,3.6)--(3.2,0);
\draw[dashed] (-0.6,4)--(0,3.8)--(3.4,0);
\draw (0,4)--(3.6,0);
\draw[blue!70!black, very thick] (-1.2,4)--(0,3.6)--(3.2,0);
\draw[blue!70!black, very thick] (0,3.6)--(0,1.2);

\draw (0,4)--(3.6,0);

\draw[dotted] (-4.5, 1.2)--(4.5, 1.2)node[right]{$t_{\bar u_L}$};

\end{tikzpicture}
}
\caption{The solution $u^*$.}
\label{exampleu*}
\end{figure}

\begin{figure}
\centering
\scriptsize{
\begin{tikzpicture}[scale = 0.7]
\draw[->] (-4.5,0)--(4.5,0)node[right]{$x$};
\draw[->] (0,0)--(0,5)node[above]{$t$};
\draw[very thin] (-4.5,4)--(4.5,4);
\draw (4,4)node[above]{$\omega_1$};
\draw (-4,2)node[above]{$p$};
\draw (-1,2)node[above]{$A$};
\draw (3,2)node[above]{$\ol B$};
\draw (-3,4)node[above]{$L$};


\draw (-1.2, 4)node[above]{$\ol x$};

\draw[dashed, blue!70!black] (0,2.8)--(2.4,0);
\draw[dashed, blue!70!black](0,2.6)--(2.2,0);
\draw[dashed, blue!70!black](0,2.4)--(2,0);
\draw[dashed, blue!70!black](0,2.2)--(1.8,0);
\draw[dashed, blue!70!black](0,2)--(1.6,0);
\draw[dashed, blue!70!black](0,1.8)--(1.4,0);
\draw[dashed, blue!70!black](0,1.6)--(1.2,0);
\draw[dashed, blue!70!black](0,1.4)--(1,0);
\draw[dashed, blue!70!black] (0,1.2)--(0.82,0);
\draw[dashed, blue!70!black] (0,1)--(0.68,0);
\draw[dashed, blue!70!black] (0,0.8)--(0.55,0);
\draw[dashed, blue!70!black] (0,0.6)--(0.4,0);
\draw[dashed, blue!70!black] (0,0.4)--(0.25,0);
\draw[dashed, blue!70!black] (0,0.2)--(0.12,0);

\draw[red!70!black, thick] (-3,4)--(-1.5,0);

\draw[dashed, blue!70!black] (0,3)--(2.6,0);
\draw[dashed, blue!70!black] (0,3.2)--(2.8,0);
\draw[dashed, blue!70!black](0,3.4)--(3,0);
\draw[blue!70!black, very thick] (-1.2,4)--(0,3.6)--(3.2,0);
\draw[blue!70!black, very thick] (0,3.6)--(0,0);
\draw[dashed] (0,3.8)--(3.4,0);
\draw (0,4)--(3.6,0);

\end{tikzpicture}
}
\caption{The solution $u_1$.}
\label{exampleu1}
\end{figure}
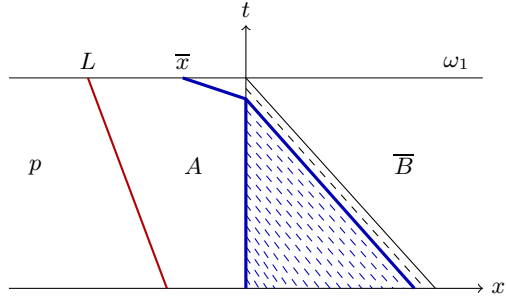

\begin{figure}
\centering
\scriptsize{
\begin{tikzpicture}[scale = 0.7]
\draw[->] (-4.5,0)--(4.5,0)node[right]{$x$};
\draw[->] (0,0)--(0,5)node[above]{$t$};
\draw[very thin] (-4.5,4)--(4.5,4);
\draw (4,4)node[above]{$\omega_1$};
\draw (-4,2)node[above]{$p$};
\draw (3,2)node[above]{$\ol B$};
\draw (-3,4)node[above]{$L$}--(-4,0);


\draw (-3,4)--(-2,0);
\draw[smooth, red!70!black, thick] plot coordinates{(-3,4)(-1.8,3.1)(-0.7,2.15)(-0.1,1.5)(0,1.2)};
\draw (-1.2, 4)node[above]{$\bar x$};
\draw[blue, very thick] (0,1.2)--(-2,0);
\draw[dashed] (-3,4)--(-3.8,0);
\draw[dashed] (-3,4)--(-3.6,0);
\draw[dashed] (-3,4)--(-3.4,0);
\draw[dashed] (-3,4)--(-3.2,0);
\draw[dashed] (-3,4)--(-3,0);
\draw[dashed] (-3,4)--(-2.8,0);
\draw[dashed] (-3,4)--(-2.6,0);
\draw[dashed] (-3,4)--(-2.4,0);
\draw[dashed] (-3,4)--(-2.2,0);

\draw[dashed] (-2,0)--(-2.6,3.7);
\draw[dashed] (-2,0)--(-2.2,3.4);
\draw[dashed] (-2,0)--(-1.8,3.1);
\draw[dashed] (-2,0)--(-1.4,2.8);
\draw[dashed] (-2,0)--(-1,2.5);
\draw[dashed] (-2,0)--(-0.7,2.17);
\draw[dashed] (-2,0)--(-0.4,1.9);
\draw[dashed] (-2,0)--(-0.1,1.5);

\draw[dashed]  (-2.6,3.7)--(0,2.8);
\draw[dashed, blue!70!black] (0,2.8)--(2.4,0);
\draw[dashed]  (-2.2,3.4)--(0,2.6);
\draw[dashed, blue!70!black](0,2.6)--(2.2,0);
\draw[dashed]  (-1.8,3.1)--(0,2.4);
\draw[dashed, blue!70!black](0,2.4)--(2,0);
\draw[dashed]  (-1.4,2.8)--(0,2.2);
\draw[dashed, blue!70!black](0,2.2)--(1.8,0);
\draw[dashed]  (-1,2.5)--(0,2);
\draw[dashed, blue!70!black](0,2)--(1.6,0);
\draw[dashed]  (-0.7,2.17)--(0,1.8);
\draw[dashed, blue!70!black](0,1.8)--(1.4,0);
\draw[dashed]  (-0.4,1.9)--(0,1.6);
\draw[dashed, blue!70!black](0,1.6)--(1.2,0);
\draw[dashed]  (-0.1,1.5)--(0,1.4);
\draw[dashed, blue!70!black](0,1.4)--(1,0);
\draw[very thick, blue] (0,1.2)--(0.82,0);

\draw[dashed] (-3,4)--(0,3);
\draw[dashed, blue!70!black] (0,3)--(2.6,0);
\draw[dashed] (-2.4,4)--(0,3.2);
\draw[dashed,blue!70!black] (0,3.2)--(2.8,0);
\draw[dashed] (-1.8,4)--(0,3.4);
\draw[dashed, blue!70!black](0,3.4)--(3,0);
\draw[blue!70!black, very thick] (-1.2,4)--(0,3.6)--(3.2,0);
\draw[dashed] (-0.6,4)--(0,3.8)--(3.4,0);
\draw (0,4)--(3.6,0);
\draw[blue!70!black, very thick] (-1.2,4)--(0,3.6)--(3.2,0);
\draw[blue!70!black, very thick] (0,3.6)--(0,1.2);

\draw (0,4)--(3.6,0);
\draw[dotted] (-4.5, 1.2)--(4.5, 1.2)node[right]{$t_{\bar u_L}$};


\draw[dashed] (-2,0)--(0,1);
\draw[dashed] (-2,0)--(0,0.8);
\draw (-2,0)--(0,0.6);
\draw[dashed] (-1.5,0)--(0,0.45);
\draw[dashed] (-1,0)--(0,0.3);
\draw[dashed] (-0.5,0)--(0,0.15);

\draw [dashed] (0,1)--(0.82,0);
\draw [dashed] (0,0.8)--(0.82,0);
\draw (0,0.6)--(0.82,0);
\draw [dashed](0,0.45)--(0.6,0);
\draw[dashed] (0,0.3)--(0.4,0);
\draw[dashed] (0,0.15)--(0.2,0);
\end{tikzpicture}
}
\caption{The solution $u_2$.}
\label{exampleu2}
\end{figure}
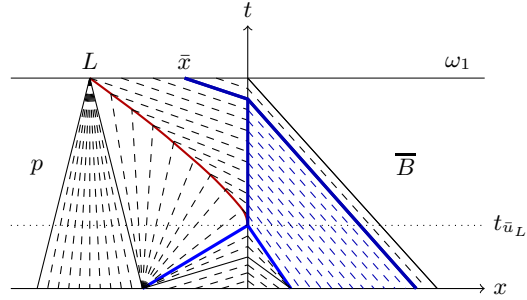

\begin{figure}
\centering
\scriptsize{
\begin{tikzpicture}[scale = 0.7]
\draw[->] (-4.5,0)--(4.5,0)node[right]{$x$};
\draw[->] (0,0)--(0,5)node[above]{$t$};
\draw[very thin] (-4.5,4)--(4.5,4);
\draw (4,4)node[above]{$\omega_1$};
\draw (-4,2)node[above]{$p$};
\draw (3,2)node[above]{$\ol B$};
\draw (-3,4)node[above]{$L$}--(-4,0);


\draw (-3,4)--(-2,0);
\draw[smooth,thick,  red!70!black, thick] plot coordinates{(-3,4)(-1.8,3.1)(-0.7,2.15)(-0.1,1.5)(0,1.2)};
\draw (-2.4, 4)node[above]{$\bar x$};
\draw[blue!70!black, very thick] (0,1.2)--(-2,0);
\draw[dashed] (-3,4)--(-3.8,0);
\draw[dashed] (-3,4)--(-3.6,0);
\draw[dashed] (-3,4)--(-3.4,0);
\draw[dashed] (-3,4)--(-3.2,0);
\draw[dashed] (-3,4)--(-3,0);
\draw[dashed] (-3,4)--(-2.8,0);
\draw[dashed] (-3,4)--(-2.6,0);
\draw[dashed] (-3,4)--(-2.4,0);
\draw[dashed] (-3,4)--(-2.2,0);

\draw[dashed] (-2,0)--(-2.6,3.7);
\draw[dashed] (-2,0)--(-2.2,3.4);
\draw[dashed] (-2,0)--(-1.8,3.1);
\draw[dashed] (-2,0)--(-1.4,2.8);
\draw[dashed] (-2,0)--(-1,2.5);
\draw[dashed] (-2,0)--(-0.7,2.17);
\draw[dashed] (-2,0)--(-0.4,1.9);
\draw[dashed] (-2,0)--(-0.1,1.5);

\draw[dashed]  (-2.6,3.7)--(0,2.8);

\draw[dashed]  (-2.2,3.4)--(0,2.6);
\draw[dashed]  (-1.8,3.1)--(0,2.4);

\draw[dashed]  (-1.4,2.8)--(0,2.2);
\draw[dashed]  (-1,2.5)--(0,2);
\draw[dashed, blue!70!black](0,2)--(1.6,0);
\draw[dashed]  (-0.7,2.17)--(0,1.8);
\draw[dashed, blue!70!black](0,1.8)--(1.4,0);
\draw[dashed]  (-0.4,1.9)--(0,1.6);
\draw[dashed, blue!70!black](0,1.6)--(1.2,0);
\draw[dashed]  (-0.1,1.5)--(0,1.4);
\draw[dashed, blue!70!black](0,1.4)--(1,0);
\draw[dashed, blue!70!black] (0,1.2)--(0.82,0);
\draw[dashed, blue!70!black] (0,1)--(0.68,0);
\draw[dashed, blue!70!black] (0,0.8)--(0.55,0);
\draw[dashed, blue!70!black] (0,0.6)--(0.4,0);
\draw[dashed, blue!70!black] (0,0.4)--(0.25,0);
\draw[dashed, blue!70!black] (0,0.2)--(0.12,0);

\draw[dashed, blue!70!black] (0,1)--(-1.7,0);
\draw[dashed, blue!70!black] (0,0.8)--(-1.35,0);
\draw[dashed, blue!70!black] (0,0.6)--(-1.05,0);
\draw[dashed, blue!70!black] (0,0.4)--(-0.8,0);
\draw[dashed, blue!70!black] (0,0.2)--(-0.4,0);

\draw[dashed] (-3,4)--(0,3);
\draw[very thick, blue!70!black] (-2.4,4)--(0,3.2);
\draw[dashed] (-1.8,4)--(0,3.4);
\draw[dashed] (-1.2,4)--(0,3.6);
\draw (0,3.6)--(3.2,0);
\draw[dashed] (-0.6,4)--(0,3.8)--(3.4,0);
\draw (0,4)--(3.6,0);
\draw[blue!70!black, very thick] (0,3.2)--(0,1.2);
\draw [very thick, blue] (0,2.2)--(1.75,0);
\draw (0,4)--(3.6,0);
\draw[dotted] (-4.5, 1.2)--(4.5, 1.2)node[right]{$t_{\bar u_L}$};

\draw[dashed] (0,3)--(0.1,3.1);
\draw[dashed] (0,2.8)--(0.18,2.98);
\draw[dashed] (0,2.6)--(0.25,2.85);
\draw[dashed] (0,2.4)--(0.2,2.6);

\draw[smooth, thick, red!70!black, thick] plot coordinates{(0,3.6)(0.1,3.2)(0.3,2.6)};
\draw[smooth,thick, red!70!black, thick] plot coordinates{(0,2.2)(0.1,2.4)(0.3,2.6)};
\draw (0.3,2.6)--(3.2,0);
\draw (0.3,2.6)--(1.75,0);

\end{tikzpicture}
}
\caption{The solution $u_3$.}
\label{exampleu3}
\end{figure}
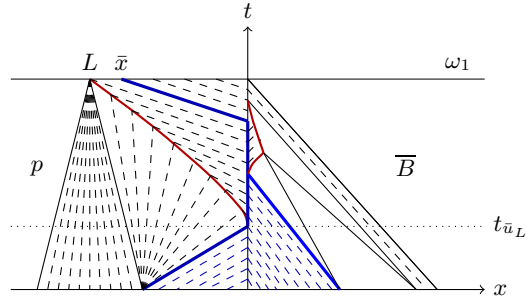

\medskip

    \subsection{Main results}\label{subsec:main-results}
    The first main result of the paper
    provides a characterization of the
    set $\mc I_T^{[AB]}(\omega^T)$
    in~\eqref{eq:initialdataintro}. 
    By the analysis in~\cite{talamini_ancona_attset}
     we know that if $\omega^T\in \mc A^{[AB]}(T)$,
     then the $AB$-entropy solution $u^*$
defined by~\eqref{eq:vertexinitialset}-\eqref{eq:sol-vertexinitialset} satisfies
$u^*(\cdot, T)=\omega^T$, which means that
     $u_0^*\in\mc I_T^{[AB]}(\omega^T)$.
    Our next Theorem gives a characterization of the possible elements $u_0\in\mc I_T^{[AB]}(\omega^T)$
    which are different from $u_0^*$.
     


\begin{thm}\label{initialdataid}
Given  $\omega^T \in \mc A^{[AB]}(T)$,
let $\mathcal{C}_0(u^*, x)$ denote the set defined in~\eqref{eq:ABchar-set0}
for the $AB$-entropy solution $u^*$
defined by~\eqref{eq:vertexinitialset}-\eqref{eq:sol-vertexinitialset},
and let $u_0 \in \mathbf L^{\infty}(\mathbb R)$.
Then $u_0 \in  \mc I_T^{[AB]}(\omega^T)$ if and only if for every point $\overline x\in\R$ 
there exists $\overline y \in \mathcal{C}_0(u^*,\overline x)$ such that
there hold
\begin{equation}
\label{condleq-2}
\int_y^{\overline y} u_0(x) \dif x \leq \int_y^{\overline y}u^*_0(x) \dif x, \qquad \forall~y < \min \mathcal{C}_0(u^*,\overline x)
\end{equation}
\indent and 
\begin{equation}
\label{condgeq-2}
\int_{\overline y}^y u_0(x) \dif x \geq \int_{\overline y}^yu_0^*(x) \dif x, \qquad \forall~y > \max \mathcal{C}_0(u^*,\overline x)
\end{equation}
\end{thm}
\medskip

\begin{remark}
Given  $\omega^T \in \mc A^{[AB]}(T)$, one can verify that the set of initial data $\mathcal{I}_T^{[AB]}(\omega^T)$
shares the same topological properties
enjoyed by the  set
of initial data leading at time $T$ to an attainable profile for a conservation laws with uniformly convex flux independent on the space variable
(see~\cite[Proposition~5.1]{COLOMBO2020}).
Namely, 
with respect to the $\mathbf{L}^1_{\mathrm{loc}}$ topology, 
    we have: 
    \begin{enumerate}
        \item[(i)] for every $M>0$, the set $\mathcal{I}_T^{[AB]}(\omega^T) \cap \{u_0 : \left\Vert u_0\right\Vert_{\mathbf{L}^{\infty}} \leq M\}$ is closed, and $\mathcal{I}_T^{[AB]}(\omega^T)$ is an $F_{\sigma}$ set; 
\item[(ii)] the set $\mathcal{I}_T^{[AB]}(\omega^T)$ 
has empty interior.
    \end{enumerate}
    The first property follows immediately by the $\bf L^1$-contractivity  of the semigroup  of $AB$-entropy solutions.
    Concerning property (ii), consider two points $0<x_1<x_2$
    of continuity for $\omega^T$,  such that the classical genuine characteristics $\vartheta_{x_1}, \vartheta_{x_2} :[0,T]\to \R$
    for $u_t+f_r(u)_x=0$,
    passing at time $T$ through $x_1, x_2$, respectively, never cross the interface $x=0$.
    Let $u^*$ be the $AB$-entropy solution defined in~\eqref{eq:sol-vertexinitialset}.
    Then, by Remark~\ref{rem:oncharact}, $\vartheta_{x_1}, \vartheta_{x_2}$ are the unique $AB$-gic for $u^*$  that reach at time~$T$ the points
    $x_1, x_2$, respectively. By definition~\eqref{eq:ABchar-set0} this means that  
     $\mathcal{C}_0(u^*,x_i) = \{\vartheta_{x_i}(0)\}$, $i=1,2$.
    Note that, by the non crossing property of genuine characteristics, we have
  $\theta_{x_1}(0) < \theta_{x_2}(0)$. Next,
  applying the inequality~\eqref{condleq} for $\ol x = x_2, y = \vartheta_{x_1}(0)$, we find that any element $u_0\in\mathcal{I}_T^{[AB]}(\omega^T)$ satisfies
\begin{equation}\label{thetacondleq}
\int_{\vartheta_{x_1}(0)}^{\vartheta_{x_2}(0)}u_0(x) \dif x \leq \int_{\vartheta_{x_1}(0)}^{\vartheta_{x_2}(0)}u_0^*(x) \dif x
\end{equation}
On the other hand, 
applying the inequality~\eqref{condgeq} for $\ol x = x_1, y = \vartheta_{x_2}(0)$, we find that any element $u_0\in \mathcal{I}_T^{[AB]}(\omega^T)$ satisfies
\begin{equation}\label{thetacondgeq}
\int_{\vartheta_{x_1}(0)}^{\vartheta_{x_2}(0)}u_0(x) \dif x \geq \int_{\vartheta_{x_1}(0)}^{\vartheta_{x_2}(0)}u_0^*(x) \dif x\,.
\end{equation}
The inequalities \eqref{thetacondleq}, \eqref{thetacondgeq} together imply that every element $u_0 \in \mathcal{I}_T^{[AB]}(\omega^T)$ satisfies
\begin{equation}\label{thetacondeq}
\int_{\vartheta_{x_1}(0)}^{\vartheta_{x_2}(0)}u_0(x) \dif x = \int_{\vartheta_{x_1}(0)}^{\vartheta_{x_2}(0)}u_0^*(x) \dif x.
\end{equation}
Then, letting $G: \mathbf L^{\infty}(\R)\to \R$
be the linear map defined by $G(u_0)=\int_{\vartheta_{x_1}(0)}^{\vartheta_{x_2}(0)}u_0(x) \dif x$,
we deduce from~\eqref{thetacondeq}
that
\begin{equation*}
  \mathcal{I}_T^{[AB]}(\omega^T)\subset \{u_0\in  \mathbf L^{\infty}(\R)\:|\: G(u_0)=G(u_0^*),
   \}
\end{equation*}
which shows that $\mathcal{I}_T^{[AB]}(\omega^T)$ has an empty interior since
is is contained 
in an hyperplane of~${\bf L}^{\infty}(\R)$).
\end{remark}

The second main contribution of this paper establishes some structural and geometrical 
properties of the set
$\mc I_T^{[AB]}(\omega^T)$.

\begin{thm}\label{geometrical properties}
Given  $\omega^T \in \mc A^{[AB]}(T)$, 
with the same notations of Theorem~\ref{initialdataid}
the following properties hold.
\begin{enumerate}

\item[(i)] The set $\mathcal{I}_T^{[AB]}(\omega^T)$ reduces to the singleton 
$\{u_0^*\}$
if and only if $\left|\mathcal{C}_0(u^*,x)\right| = 1$ 
for every $x \in \mathbb R$. In particular, if $\mathcal{I}_T^{[AB]}(\omega^T)=\{u_0^*\}$
then $\omega^T$ is continuous
on $\R\setminus\{0\}$.

\item[(ii)] The set $\mathcal{I}_T^{[AB]}(\omega^T)$ is an affine cone having $u_0^*$ as its vertex (i.e. the set $\mathcal{I}_T^{[AB]}(\omega^T)-u_0^*$ is a linear cone).
Moreover, $u_0^*$ is the unique extremal point of 
the set $\mathcal{I}_T^{[AB]}(\omega^T)$. 

\item[(iii)] If, setting
\begin{equation}
   \label{eq:LR-def}
    \begin{aligned}
    \ms L\doteq \ms L[\omega^T, f_l] &\doteq  \sup \big\{ L < 0 \; : \; x-T \cdot f_l^{\prime}(\omega^T(x)) \leq 0 \quad \forall \; x \leq L\big\},
    \\
    \noalign{\smallskip}
    \ms R\doteq \ms R[\omega^T, f_r] &\doteq  \inf \big\{ R > 0 \; : \; x-T \cdot f_r^{\prime}(\omega^T(x)) \geq 0  \quad \forall \; x \geq R\big\},
\end{aligned}
\end{equation}
and
\begin{equation}
\mathcal{X}
\doteq 
\mathcal{X}(\omega^T) 
\doteq 
\Big\{
x\in\R
\;\big|\; 
\ \left|\mathcal{C}_0(u^*, x)\right| = 1\Big\},
\end{equation}
for every point $\overline x\in\,]\ms L, \ms R[\,$ of continuity  of $\omega^T$ there holds 
\begin{equation}
\label{refinedconditionconvex}
\mathcal{C}_0(u^*, \bar x) \cap \mathrm{cl}\left(\bigcup_{\,x \in  \mathcal{X}} \mc C_0(u^*,x)\right) \neq \emptyset, 
\end{equation}
then,  the set $\mathcal{I}_T^{[AB]}(\omega^T)$ is convex.
\end{enumerate}
\end{thm}

Theorem~\ref{initialdataid},
together with Theorem~\ref{geometrical properties}-(ii) and 
Example in \S~\ref{exmp:nonconvex} (showing that 
    the set $\mc I_T^{[AB]}(\omega^T)$ can well be non convex if condition~\eqref{refinedconditionconvex} is not verified), yield Theorem~\ref{thm:main} stated in the Introduction.

\begin{remark}
    Note that
    the stronger condition 
    \begin{equation}
    \label{eq:refinedconditionconvex-2}
        \left|\mathcal{C}_0(u^*, \overline x)\right| = 1\qquad\text{for every point $\overline x\in\,]\ms L, \ms R[\,$ of continuity of $\omega^T$,}
    \end{equation}
    clearly implies~\eqref{refinedconditionconvex}
    and thus ensures the convexity of $\mc I_T^{[AB]}(\omega^T)$. Actually, we will first show that the set $\mc I_T^{[AB]}(\omega^T)$
    is convex under condition~\eqref{eq:refinedconditionconvex-2}.
    Next, we will extend the result to the case where~\eqref{refinedconditionconvex}
    is verified at every point $\overline x\in\,]\ms L, \ms R[\,$ of continuity  of $\omega^T$.
\end{remark}

\begin{remark}\label{rem:nofinerpart}
    In~\cite[Theorems 4.2, 4.7, 4.9]{talamini_ancona_attset} it is shown that the attainable set $\mc A^{[AB]}(T)$ can be partitioned in classes of attainable profiles $\omega^T$
    which depend on the quantities 
    $\ms L, \ms R$ defined in~\eqref{eq:LR-def}
    and on the relative positions of $f'_l(A)/T$ with respect to $\ms L$,
    or of $f'_r(B)/T$ with respect to $\ms R$.
    These classes of profiles do not provide a finer partition 
    than the one given by the two sets
    \begin{equation*}
        \big\{\omega\in \mc A^{[AB]}(T)\;|\; \mathcal{I}_T^{[AB]}(\omega)\ \ \text{is convex}\big\},\qquad\
        \big\{\omega\in \mc A^{[AB]}(T)\;|\; \mathcal{I}_T^{[AB]}(\omega)\ \ \text{is not convex}\big\}.
    \end{equation*}
    In fact, there are profiles $\omega_2, \omega_3 \in\mc A^{[AB]}(T)$ that belong to one same class of attainable profiles 
    described in~\cite{talamini_ancona_attset}, but such that
    $\mathcal{I}_T^{[AB]}(\omega_2)$ is convex
    while $\mathcal{I}_T^{[AB]}(\omega_3)$ is not convex.
    For example, 
   we
    consider the profile
     defined in \S~\ref{subsec:examples},
    but replacing $p$ with $\bs v$, i.e. setting
    \begin{equation}
        \label{eq:profile-ex1}
        \omega_2 (x) = \begin{cases}
        \bs v
        & x < L_0, \\
        A & x  \in \,]\,L_0, 0[\,, 
        \\
        \ol B  & x > 0.
    \end{cases}
    \end{equation}
    As observed in \S~\ref{subsec:examples}
    we have 
    $\ms R= \ms R[\omega_2, f_r]=0$,
    and $\ms L= \ms L[\omega_2, f_l]=L_0$.
    One can readly verify that $\omega_2$ fulfills the conditions (i)-(ii) of~\cite[Theorem 4.7]{talamini_ancona_attset}, as does 
    the profile $\omega_3$
    in~\eqref{eq:profile-nonconvex-initialset} considered in Example of \S~\ref{exmp:nonconvex}.
    We will show in \S~\ref{exmp:nonconvex}
    that the set of initial data $\mathcal{I}_T^{[AB]}(\omega_3)$ is not convex. On the other hand, we will see here
    that, setting 
    \begin{equation}
    \label{eq:indatum-entr-sol-omega1}
        u_0^* \doteq  \mc S^{[AB]-}_T \omega_2,\qquad\quad  u^*(\cdot,t) \doteq \mc S^{[AB]+}_t u^*_0 \qquad \forall~t \in [0,T]\,,
    \end{equation}
    at every point $\ol x\in\,]L_0, 0[$\,  there holds~\eqref{refinedconditionconvex}.
    Thus, the set $\mathcal{I}_T^{[AB]}(\omega_2)$ is convex because of Theorem~\ref{geometrical properties}-(iii).

    In order to determine the sets $\mc C_0(u^*,x)$, $x\in\R$
    (and then check~\eqref{refinedconditionconvex}),
    we construct explicitly the $AB$-entropy solution $u^*$ defined in~\eqref{eq:indatum-entr-sol-omega1}, following the procedure described in~\cite[\S~5.4]{talamini_ancona_attset}.
    Namely, because of condition~\eqref{eq:ex1-constraint-12} the solution $u^*$ contains a 
    shock curve  starting at the interface $x=0$,
    and then lying in the semiplane $\{x<0\}$,
    which reaches the point $x=L_0$ at the time $T$. In fact,
    according with the analysis in~\cite[\S~3.5]{talamini_ancona_attset},
    there exist a constant  $\bs\sigma\doteq \bs \sigma[\,L_0, A, f_l]$, and a map
    $\gamma : [\bs\sigma, T]\to \,]-\infty, 0]$,  
    with the properties that $\gamma(\bs\sigma)=0$,
    $\gamma(T)=L_0$, and that $t\to (\gamma(t), t)$ is a shock curve for the conservation law
    $u_t+f_l(u)_x=0$, which connects 
    the left states 
    $(f^{\prime}_l)^{-1}\big(\big({\gamma(t)-L_0+T \cdot f^{\prime}_l(\bs v
    )}\big)/{t}\big)$, $t\in [\bs\sigma, T]$,
    with the right state~$A$.
    On the left of $\gamma(t)$ there is a rarefaction wave, connecting the left state 
    $\bs v$ with the right state $\ol A$,
    and centered at the point $(L_0-T \cdot f^{\prime}_l(\bs v), 0)$. Moreover, there holds
    \begin{equation}
    \label{eq:sigma-id-1}
        \bs \sigma=\dfrac{T\cdot f'_l(\bs v)-L_0}{f'_l(\,\ol A)}.
    \end{equation}
    Then,
    setting
    \begin{equation*}
    \begin{aligned}
        \eta_-(t)&\doteq L_0-(T-t) \cdot f^{\prime}_l(\bs v),\quad t\in [0,T],
        \\
        \noalign{\smallskip}
        \eta_+(t)&\doteq L_0-T \cdot f^{\prime}_l(\bs v)+ t \cdot f'_l(\,\ol A),
        \quad t\in [0,\bs \sigma],
    \end{aligned}
    \end{equation*}
    we find that the function $u^*$ in~\eqref{eq:indatum-entr-sol-omega1} is given by (see Figure \ref{fig:nofinerpart})

    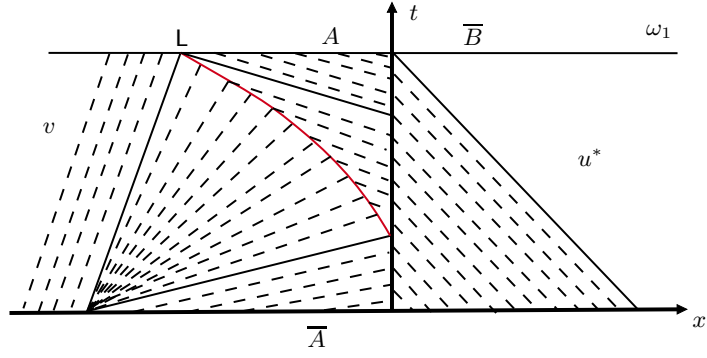
\begin{figure}
        \centering

\tikzset{every picture/.style={line width=0.75pt}} 

\begin{tikzpicture}[x=0.75pt,y=0.75pt,yscale=-0.7,xscale=0.7]

\draw [line width=1.5]    (335.59,254.96) -- (335.59,36.2) ;
\draw [shift={(335.59,32.2)}, rotate = 90] [fill={rgb, 255:red, 0; green, 0; blue, 0 }  ][line width=0.08]  [draw opacity=0] (6.97,-3.35) -- (0,0) -- (6.97,3.35) -- cycle    ;
\draw [line width=1.5]    (60,256) -- (544,253.36) ;
\draw [shift={(548,253.33)}, rotate = 179.69] [fill={rgb, 255:red, 0; green, 0; blue, 0 }  ][line width=0.08]  [draw opacity=0] (6.97,-3.35) -- (0,0) -- (6.97,3.35) -- cycle    ;
\draw    (88.13,68.46) -- (541.33,68.67) ;
\draw  [dash pattern={on 4.5pt off 4.5pt}]  (324.16,69.29) -- (336.53,71.87) ;
\draw  [dash pattern={on 4.5pt off 4.5pt}]  (217.67,89.02) -- (118,251.33) ;
\draw  [dash pattern={on 4.5pt off 4.5pt}]  (197.67,77.02) -- (116,254.67) ;
\draw  [dash pattern={on 4.5pt off 4.5pt}]  (278.33,133.69) -- (116,254.67) ;
\draw  [dash pattern={on 4.5pt off 4.5pt}]  (263.67,119.69) -- (116,254.67) ;
\draw  [dash pattern={on 4.5pt off 4.5pt}]  (249.67,108.35) -- (116,254.67) ;
\draw  [dash pattern={on 4.5pt off 4.5pt}]  (234.33,97.69) -- (118,251.33) ;
\draw    (183.33,68.35) -- (116,254.67) ;
\draw    (334.33,200.35) -- (116,254.67) ;
\draw [color={rgb, 255:red, 208; green, 2; blue, 27 }  ,draw opacity=1 ]   (183.33,68.35) .. controls (239.67,101.02) and (243,101.69) .. (271,125.02) .. controls (299,148.35) and (317.67,171.02) .. (334.33,200.35) ;
\draw  [dash pattern={on 4.5pt off 4.5pt}]  (276.33,129.02) -- (334,150.69) ;
\draw  [dash pattern={on 4.5pt off 4.5pt}]  (249.67,108.35) -- (331.67,138.35) ;
\draw  [dash pattern={on 4.5pt off 4.5pt}]  (319,176.35) -- (336.33,182.35) ;
\draw  [dash pattern={on 4.5pt off 4.5pt}]  (291.67,143.02) -- (116,254.67) ;
\draw  [dash pattern={on 4.5pt off 4.5pt}]  (217.67,89.02) -- (334,127.69) ;
\draw  [dash pattern={on 4.5pt off 4.5pt}]  (309,161.69) -- (335.67,171.69) ;
\draw  [dash pattern={on 4.5pt off 4.5pt}]  (295.67,146.35) -- (333,160.35) ;
\draw  [dash pattern={on 4.5pt off 4.5pt}]  (303.67,157.69) -- (116,254.67) ;
\draw  [dash pattern={on 4.5pt off 4.5pt}]  (315,172.35) -- (118,251.33) ;
\draw  [dash pattern={on 4.5pt off 4.5pt}]  (325,185.02) -- (116,254.67) ;
\draw  [dash pattern={on 4.5pt off 4.5pt}]  (286.67,254) -- (336.53,246.87) ;
\draw  [dash pattern={on 4.5pt off 4.5pt}]  (228,254.67) -- (334.33,235.69) ;
\draw  [dash pattern={on 4.5pt off 4.5pt}]  (148.67,255.33) -- (333.67,212.35) ;
\draw  [dash pattern={on 4.5pt off 4.5pt}]  (188,255.33) -- (332.33,224.35) ;
\draw  [dash pattern={on 4.5pt off 4.5pt}]  (296.72,68.92) -- (335.2,79.2) ;
\draw  [dash pattern={on 4.5pt off 4.5pt}]  (265.85,68.74) -- (335.2,87.2) ;
\draw  [dash pattern={on 4.5pt off 4.5pt}]  (207.67,68.35) -- (335.2,107.2) ;
\draw  [dash pattern={on 4.5pt off 4.5pt}]  (241.2,69.87) -- (333.67,97.69) ;
\draw  [dash pattern={on 4.5pt off 4.5pt}]  (336.53,113.87) -- (470.67,254) ;
\draw  [dash pattern={on 4.5pt off 4.5pt}]  (337.87,99.87) -- (484.67,254) ;
\draw  [dash pattern={on 4.5pt off 4.5pt}]  (337.2,83.87) -- (500,254.67) ;
\draw  [dash pattern={on 4.5pt off 4.5pt}]  (334,127.69) -- (454.67,253.33) ;
\draw    (335.33,67.33) -- (513.33,254) ;
\draw  [dash pattern={on 4.5pt off 4.5pt}]  (132,69.2) -- (70,252.67) ;
\draw  [dash pattern={on 4.5pt off 4.5pt}]  (335.59,143.58) -- (441.33,254.67) ;
\draw  [dash pattern={on 4.5pt off 4.5pt}]  (335.1,159.44) -- (424.67,254) ;
\draw  [dash pattern={on 4.5pt off 4.5pt}]  (145.33,69.53) -- (80,258) ;
\draw  [dash pattern={on 4.5pt off 4.5pt}]  (157.33,68.2) -- (91.26,255.58) ;
\draw  [dash pattern={on 4.5pt off 4.5pt}]  (170,68.87) -- (104.67,252) ;
\draw  [dash pattern={on 4.5pt off 4.5pt}]  (336.67,239.87) -- (350.67,254) ;
\draw  [dash pattern={on 4.5pt off 4.5pt}]  (336.67,223.87) -- (364,255.33) ;
\draw  [dash pattern={on 4.5pt off 4.5pt}]  (335.33,207.2) -- (380,256) ;
\draw  [dash pattern={on 4.5pt off 4.5pt}]  (337.33,192.53) -- (394,253.33) ;
\draw  [dash pattern={on 4.5pt off 4.5pt}]  (336,175.2) -- (410.67,256) ;
\draw    (183.33,68.35) -- (336.53,113.87) ;

\draw (177.33,49.73) node [anchor=north west][inner sep=0.75pt]  [font=\footnotesize]  {\scriptsize{$\mathsf{L}$}};
\draw (516.67,46.07) node [anchor=north west][inner sep=0.75pt]  [font=\footnotesize]  {\scriptsize{$\omega _{1}$}};
\draw (550,256.73) node [anchor=north west][inner sep=0.75pt]  [font=\footnotesize]  {\scriptsize{$x$}};
\draw (346,31.73) node [anchor=north west][inner sep=0.75pt]  [font=\footnotesize]  {\scriptsize{$t$}};
\draw (282.67,48.4) node [anchor=north west][inner sep=0.75pt]  [font=\footnotesize]  {\scriptsize{$A$}};
\draw (384.67,45.4) node [anchor=north west][inner sep=0.75pt]  [font=\footnotesize]  {\scriptsize{$\overline{B}$}};
\draw (272.67,263.4) node [anchor=north west][inner sep=0.75pt]  [font=\footnotesize]  {\scriptsize{$\overline{A}$}};
\draw (82,116.6) node [anchor=north west][inner sep=0.75pt]    {\scriptsize{$v$}};
\draw (467.33,135.27) node [anchor=north west][inner sep=0.75pt]    {\scriptsize{$u^{*}$}};

\end{tikzpicture}

        \caption{The profile $\omega_1$ with its solution $u^*$ constructed in Remark \ref{rem:nofinerpart}\\}
        \label{fig:nofinerpart}
    \end{figure}
    \begin{equation}
    \label{eq:u^*-ex1}
        u^*(x,t)=
        \begin{cases}
            \bs v\ \ &\text{if} \quad \ x < \eta_-(t), \ \ t\in [0,T],
            \\
            \noalign{\smallskip}
            (f^{\prime}_l)^{-1}\Big(\dfrac{x-L_0+T \cdot f^{\prime}_l(\bs v
            )}{t}\Big)
            \ \ &\text{if}\quad
            \left\{
            \begin{aligned}
                &\eta_-(t)<x<\gamma(t), \ \ t\in [\bs \sigma, T],
            \\
            \noalign{\smallskip}
            &\eta_-(t)<x<\eta_+(t),
            \ \ t\in \,]0,\bs \sigma],
            \end{aligned}
            \right.
            \\
            \noalign{\smallskip}
            A \ \ &\text{if} \quad\  \gamma(t)<x<0, \ \ t\in [\bs \sigma, T],
            \\
            \noalign{\smallskip}
            \overline A \ \ &\text{if} \quad\  \eta_+(t)<x<0, \ \ t\in [0,\bs \sigma],
            \\
            \noalign{\smallskip}
            \overline B \ \ &\text{if} \quad\  x>0, \ \ t\in [0,T].
        \end{cases}
    \end{equation}
    Observe that every $AB$-gic for $u^*$ that reaches a point $x\in \,]L_0, 0[$ at time $T$
    has either the expression
    \begin{equation*}
        \eta_{_{\tau_1}}(t)=
            \begin{cases}
            x- (T-t)\cdot f'_l(A)
            \ \ &\text{if}\quad\ t\in [\tau_2,T],
            \\
            0 \ \ &\text{if}\quad\ t\in [\tau_1,\tau_2],
            \\
            (t-\tau_1) \cdot f'_r(\,\ol B)
            \ \ &\text{if}\quad\ t\in [0,\tau_1],
            \end{cases}
    \end{equation*}
    with $\tau_2\doteq T- x/f'_l(A)$, and 
    $\tau_1\in [0,\tau_2]$, or the expression
   \begin{equation*}
        \widetilde\eta_{_{\,\widetilde\tau_1}}(t)=
            \begin{cases}
            x- (T-t)\cdot f'_l(A)
            \ \ &\text{if}\quad\ t\in [\tau_2,T],
            \\
            0 \ \ &\text{if}\quad\ t\in [\widetilde\tau_1,\tau_2],
            \\
            (t-\widetilde \tau_1) \cdot f'_l(\,\ol A)
            \ \ &\text{if}\quad\ t\in [0,\widetilde\tau_1],
            \end{cases}
    \end{equation*}
    with $\tau_2$ as above and
    $\widetilde\tau_1\in[0,\bs \sigma]$.
    %
    By definition~\eqref{eq:ABchar-set} this means that 
    \begin{equation*}
        \mc C(u^*,x)=
        \big\{\eta_{_{\tau_1}}\;|\; 
        \tau_1\in[0,\tau_2]\big\}\,\cup\,
        \big\{\widetilde \eta_{_{\widetilde\tau_1}}\;|\; 
        \widetilde\tau_1\in[0,\,\bs \sigma]\big\}.
    \end{equation*}
    Since we have 
    \begin{equation*}
    \begin{aligned}
        \big\{\eta_{\tau_1}(0)\,|\,\tau_1\in[0,\tau_2]\big\}&=
    \big[0,\, (x/f'_l(A)-T)\cdot f'_r(\,\ol B)\big],
    \\
    \noalign{\smallskip}
   \big\{\widetilde\eta_{\widetilde\tau_1}(0)\,|\,\widetilde\tau_1\in[0,\,\bs \sigma]\big\}&=
    \big[-\bs\sigma\cdot f'_l(\,\ol A),0\big],
    \end{aligned}
    \end{equation*}
    by definition~\eqref{eq:ABchar-set0} and by virtue of~\eqref{eq:sigma-id-1}  we 
    then find that 
    \begin{equation}
    \label{eq:C0-ex1-1}
        \mc C_0(u^*,x)=\big[\,L_0 - T\cdot f'_l(\bs v),\, (x/f'_l(A)-T)\cdot f'_r(\,\ol B)\big]
        \qquad\forall~x\in\,]\,L_0, 0[\,.
    \end{equation}
    On the other hand, 
    since $\omega_2$ is constant for  $x<L_0$, there exists a unique $AB$-gic for $u^*$ that reaches a point  $x<L_0$ at time $T$,
    which is a classical genuine characteristic
    \begin{equation*}
        \vartheta_x(t)=x-(T-t)\cdot f'_l(\bs v),
        \qquad t\in [0,T],
    \end{equation*}
    because it never crosses the interface $x=0$.
    Hence, we have
    \begin{equation}
        \label{eq:C0-ex1-2}
        \mc C_0(u^*,x)=\{x-T\cdot f'_l(\bs v)\}\qquad\forall~x<L_0\,.
    \end{equation}
    Therefore, from~\eqref{eq:C0-ex1-1}, \eqref{eq:C0-ex1-2}, we deduce
    \begin{equation*}
        \mathcal{C}_0(u^*, \ol x) \cap \mathrm{cl}\left(\bigcup_{\,x< L_0} \mc C_0(u^*,x)\right)
        =\big\{\,L_0 - T\cdot f'_l(\bs v)\big\}
        \qquad\forall~\ol x\in\,]\,L_0, 0[\,,
    \end{equation*}
    which proves~\eqref{refinedconditionconvex}, and thus concludes the proof of the convexity of $\mathcal{I}_T^{[AB]}(\omega_2)$.
\end{remark}

\medskip

\section{Properties of genuine/interface characteristics
}\label{sec:charmain}

In this section we 
establish some basic properties 
enjoyed by the $AB$-genuine/interface cha\-ra\-cteristics
for an $AB$-entropy solution $u$, and by the sets $\mathcal{C}(u,x)$, $\mathcal{C}_0(u,x)$, introduced in~\S~\ref{sec:main}.



\begin{prop}
\label{lemma:Cnotempty}
Let $u$ be an $AB$-entropy solution to \eqref{conslaw}. Then the following properties hold.
\begin{enumerate}
 \item[(i)] $\mc C(u,x) \neq \emptyset$ \ 
 for all $x\in\R$;
 \item[(ii)] the  map $x \mapsto \mc C(u, x)$ has closed graph as a set-valued map from
    $\R$ into the power set of the space $\mr{Lip}([0,T]\; ; \mathbb R)$ with the topology of uniform convergence;
\item[(iii)] 
the map $x \mapsto \mc C_0(u, x)$ has closed graph as a set-valued map from
    $\R$ into the power set of
    $\R$;
\item[(iv)] the maps $x \mapsto \min\mc C_0(u, x)$,
$x \mapsto \max\mc C_0(u, x)$ are monotone nondecreasing.
\end{enumerate}
\end{prop}
\begin{proof} Throughout the proof we set $\omega^T(x)\doteq u(x,T),$ $x\in\R$,
and we let $u_l(t), u_r(t)$ denote the one-sided traces of $u(t,\cdot)$ at $x=0$.

 \textbf{1.} Proof of (i). 
 Given $x>0$, consider the minimal backward characteristic  
 for the conservation law $u_t+f_r(u)_x=0$,
 in the semiplane $\{x>0\}$,
 starting from $(x,T)$,
 defined by $\vartheta_{x,-}(t)=x-(T-t) \cdot f_r^{\prime}(\omega^T(x-))$.
 If 
$x-T\cdot f_r^{\prime}(\omega^T(x-)) \geq 0$, then 
$\vartheta_{x,-}$ is  a classical genuine characteristic for $u$ 
in the whole interval $[0,T]$, since it never 
crosses the interface $x=0$ but at most at $t=0$.
Therefore, according with Definition~\ref{interfacecharacteristics},
the map
$$\zeta(t)=x-(T-t) \cdot f_r^{\prime}(\omega^T(x-)),
\qquad t\in [0,T],$$
is an AB-genuine/interface characteristic,
and hence by~\eqref{eq:ABchar-set} it holds $\zeta \in C(u,x)$.
Otherwise, we have $x-T\cdot f_r^{\prime}(\omega^T(x-)) < 0$, and thus $\vartheta_{x,-}$
 impacts the interface at the time:
\begin{equation}
\label{eq:tau-def-1}
\tau_-(x)\doteq T-\frac{x}{f_r^{\prime}(\omega^T(x-))}>0.
\end{equation}
Then, consider the set
\begin{equation}
    \label{eq:def-E}
    E \doteq \Big\{ t \in [0, \tau_-(x)] \; \big|\; \text{either}\; u_l(t)>\theta_l\; \ \text{or}\; \ u_r(t)<\theta_r\Big\},
\end{equation}
and 
let 
\begin{equation}
\label{eq:def-supE}
    \overline \tau\doteq \sup E,
\end{equation}
where we understand that $\overline \tau=0$ when $E=\emptyset$.
Because of the  non-intersection 
property of classical genuine characteristics in the domains $\{x<0\}, \{x > 0\}$, 
and since uniform limit of classical genuine characteristics is a classical genuine characteristic as well (e.g. cfr.~\cite[proof o Lemma C.1]{talamini_ancona_attset}), 
we deduce that 
\begin{equation}
\label{eq:supE=maxE}
    \overline \tau\in E\qquad\text{if}\quad\ 
    E\neq\emptyset.
\end{equation}
Thus, when $E\neq\emptyset$, if $u_l(\overline \tau)>\theta_l$ we can consider the minimal backward characteristic  
 for the conservation law $u_t+f_l(u)_x=0$,
 in the semiplane $\{x<0\}$,
 starting from $(0,\overline\tau)$,
 defined by $\vartheta_{\overline\tau,-}(t)=(t-\overline\tau) \cdot f_l^{\prime}(u_l(\overline\tau))$.
 Otherwise, if $u_r(\overline \tau)<\theta_r$ we can consider the maximal backward characteristic
 for the conservation law $u_t+f_r(u)_x=0$,
 in the semiplane $\{x>0\}$,
 starting from $(0,\overline\tau)$,
 defined by $\vartheta_{\overline\tau,+}(t)=(t-\overline\tau) \cdot f_r^{\prime}(u_r(\overline\tau))$.
 On the other hand, by definition of $E$, and recalling the interface condition~\eqref{ABtraces}, we find that
 \begin{equation}
 \label{eq:intcond-zeta-1}
     u_l(t)= A, \quad u_r(t) =B \qquad 
\forall~t \in \,]\,\ol \tau, \tau_-(x)].
 \end{equation}
 Note in particular that
 \begin{equation}
  \label{eq:trace-ur-tau} \ol\tau<\tau_-(x)\quad\Longrightarrow\quad u_r(\tau_-(x))=\omega^T(x-)=B.
 \end{equation}
Therefore,  the piecewise affine map
\begin{equation}
    \label{eq:char-int}
    \zeta(t) = \begin{cases}
x-(T-t) \cdot f_r^{\prime}(\omega^T(x-)), & t \in [\tau_-(x), T], \\
0 , & t \in \,]\overline\tau , \tau_-(x)[, \\
(t-\overline\tau) \cdot f_l^{\prime}(u_l(\overline\tau)), & t \in [0, \overline\tau],\ \ \text{if} \ u_l(\overline \tau)>\theta_l,
\\
(t-\overline\tau) \cdot f_r^{\prime}(u_r(\overline\tau)), & t \in [0, \overline\tau],\ \ \text{if} \ u_r(\overline \tau)<\theta_r,
\end{cases}
\end{equation}
satisfy the conditions of Definition~\ref{interfacecharacteristics},
and thus it is  an $AB$-gic
belonging to the set $\mc C(u,x)$.
Note that it may well happen that $\overline\tau =\tau_-(x)$,
in which case there will be in~\eqref{eq:char-int} no nontrivial interval where the
characteristic is travelling along the interface $x=0$. Instead, in the case 
$\overline\tau =0$, the $AB$-gic
in~\eqref{eq:char-int} lies on the interface $x=0$
in the whole interval $[0,\tau_-(x)]$.

Clearly, the same analysis can be carried out to show that $\mc C(u,x) \neq \emptyset$ also for  $x<0$. It remains to consider the case
$x=0$. Notice that this case would follow from (ii) and from (i) for $x \neq 0$, however for clarity we write the construction explicitly. If we assume that 
$\omega^T(0-)>\theta_l$,
then the minimal backward characteristic for
$u_t+f_l(u)_x=0$,
 in the semiplane $\{x<0\}$,
 starting from $(0,T)$, 
is  a classical genuine characteristic for $u$ 
in the whole interval $]0,T]$, and hence it 
it is  an $AB$-gic
belonging to the set $\mc C(u,0)$.
Similarly, if $\omega^T(0+)<\theta_r$,
then the maximal backward characteristic for
$u_t+f_r(u)_x=0$,
 in the semiplane $\{x>0\}$,
 starting from $(0,T)$, 
is  a classical genuine characteristic for $u$ 
in the whole interval $[0,T]$, and hence 
it is  an $AB$-gic
belonging to the set $\mc C(u,0)$.
Finally, if $\omega^T(0-)\leq \theta_l$
and $\omega^T(0+)\geq \theta_r$, by
 the interface condition~\eqref{ABtraces}, we
 deduce that $\omega^T(0-)=A$, $\omega^T(0+)=B$.
 Then, set
 $$
 \overline\tau=\sup E,
 \qquad\quad
E \doteq \Big\{ t \in [0, T] \; \big|\; \text{either}\; u_l(t)>\theta_l\; \ \text{or}\; \ u_r(t)<\theta_r\Big\}.
$$
With the same type of analysis 
as above we find that $\overline\tau\in E$
and that the map
\begin{equation}
    \label{eq:char-int-2}
    \zeta(t) = \begin{cases}
0 , & t \in \,]\overline\tau , T], \\
(t-\overline\tau) \cdot f_l^{\prime}(u_l(\overline\tau)), & t \in [0, \overline\tau],\ \ \text{if} \ u_l(\overline \tau)>\theta_l,
\\
(t-\overline\tau) \cdot f_r^{\prime}(u_r(\overline\tau)), & t \in [0, \overline\tau],\ \ \text{if} \ u_r(\overline \tau)<\theta_r,
\end{cases}
\end{equation}
is  an $AB$-gic
belonging to the set $\mc C(u,0)$, thus completing the proof of (i).
\smallskip

\begin{figure}
\centering
\footnotesize{
\begin{tikzpicture}[scale = 0.8]
\draw[->] (-4.5,0)--(4.5,0)node[right]{$x$};
\draw[->] (0,0)--(0,4.5);
\draw[ultra thin] (-4.5,4)--(4.5,4);

\draw (-3.5,4)node[above]{$L$}--(0,1.5)--(1.5,0);
\draw (-3.5,4)--(-4,0);
\draw [dashed] (2,4)--(0,3.5);
\draw (2,4)node[above]{$R$}--(4,0);

\draw (-0.3,4)node[above]{$A$};
\draw (0.5,4)node[above]{$B$};
\draw (-0.8,4)node[above]{$\bar x$};

\draw[dashed] (-3.5,4)--(0,0.25);
\draw[dashed] (-3.5,4)--(0,0);
\draw[dashed] (-3.5,4)--(-0.2,0);
\draw[dashed] (-3.5,4)--(-0.4,0);
\draw[dashed] (-3.5,4)--(-0.6,0);
\draw[dashed] (-3.5,4)--(-0.8,0);
\draw[dashed] (-3.5,4)--(-1,0);
\draw[dashed] (-3.5,4)--(-1.2,0);
\draw[dashed] (-3.5,4)--(-1.4,0);
\draw[dashed] (-3.5,4)--(-1.6,0);
\draw[dashed] (-3.5,4)--(-1.8,0);
\draw[dashed] (-3.5,4)--(-2,0);
\draw[dashed] (-3.5,4)--(-2.2,0);
\draw[dashed] (-3.5,4)--(-2.4,0);
\draw[dashed] (-3.5,4)--(-2.6,0);
\draw[dashed] (-3.5,4)--(-2.8,0);
\draw[dashed] (-3.5,4)--(-3,0);
\draw[dashed] (-3.5,4)--(-3.2,0);
\draw[dashed] (-3.5,4)--(-3.4,0);
\draw[dashed] (-3.5,4)--(-3.6,0);
\draw[dashed] (-3.5,4)--(-3.8,0);

\draw[dashed] (-3.5,4)--(0,0.5)--(0.2,0);
\draw[dashed] (-3.5,4)--(0,0.7)--(0.5,0);
\draw[dashed] (-3.5,4)--(0,0.9)--(0.8,0);
\draw[dashed] (-3.5,4)--(0,1.1)-- (1.1,0);
\draw[dashed] (-3.5,4)--(0,1.3)-- (1.3,0);

\draw (-1.5,4)--(0,2.5)--(2,0);
\draw[dashed] (-2,4)--(0,2.3)--(1.9,0);
\draw [dashed](-2.5,4)--(0,2.1)-- (1.8,0);
\draw[dashed] (-3,4)--(0,1.9)-- (1.7,0);
\draw[dashed] (-3.5,4)--(0,1.7)-- (1.6,0);

\draw[dashed] (-1.25,4)--(0,2.75);
\draw[dashed] (-1,4)--(0,3);
\draw[very thick, blue!70!black] (-0.75,4)--(0,3.25);
\draw[very thick, blue!70!black] (0,3.25)--(0,2.5)--(2,0);
\draw[dashed] (-0.5,4)--(0,3.5);
\draw[dashed] (-0.25,4)--(0,3.75);
\draw[dashed] (1,4)--(0,3.75);

\draw[smooth, red!70!black, thick] plot coordinates{(0,2.5)(0.2,3)(1,3.6)(2,4)};

\draw (2,4)--(2,0);

\draw[dashed] (0,2.75)--(0.1,2.85);
\draw[dashed] (0,3)--(0.4,3.15);
\draw[dashed] (0,3.25)--(0.8,3.55);

\draw[dashed] (0.2,3)--(2,0);
\draw[dashed] (1.5,3.8)--(2,0);
\draw[dashed] (0.5,3.3)--(2,0);

\draw[dashed] (1,3.6)--(2,0);

\draw[dashed] (2,4)--(2.25,0);
\draw[dashed] (2,4)--(2.5,0);
\draw[dashed] (2,4)--(2.75,0);
\draw[dashed](2,4)--(3,0);
\draw[dashed](2,4)--(3.25,0);
\draw[dashed] (2,4)--(3.5,0);
\draw[dashed] (2,4)--(3.75,0);

\draw (4,4) node[above]{$\omega$};

\end{tikzpicture}
}
\caption{There are no backward generalized characteristics with time of existence $[0,T]$ from the point $\bar x$ which reach time $t = 0$ if we use the classical definition. If, instead, we consider elements in $\mathcal{C}(u, \bar x)$ (the blue line), we see that, also if at the time at which the characteristic reaches the interface it cannot be prolonged on the other side in classical sense, there is at least an element in $\mathcal{C}(u, \bar x)$ that is defined on the whole $[0,T]$.}
\end{figure}
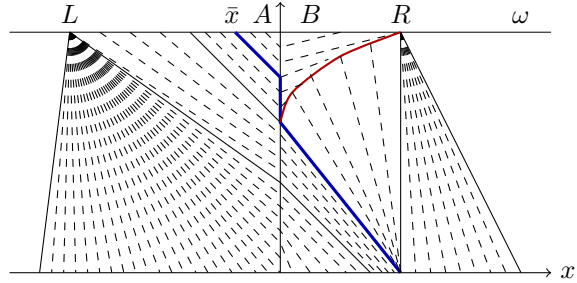

 \textbf{2.} Proof of (ii). 
 The closed graph property 
 of the map $x \mapsto \mc C(u, x)$ is equivalent to 
 \begin{equation}
 \label{eq:graphclosureC}
     \Big(x_n \to x, \quad  \zeta_n \in \mc C(u,x_n),\quad \zeta_n \to \zeta \quad \text{uniformly}\Big) \quad \Longrightarrow \quad \zeta \in \mc C(u,x).
 \end{equation}
Then, let $\{x_n\}_n$ be a sequence 
converging to $x\geq 0$, and
consider a sequence of $AB$-gic
$\zeta_n\in \mc C(u,x_n)$, that 
converge uniformly to some $\zeta\in \mr{Lip}([0,T]\; ; \mathbb R)$. By Remark~\ref{rem:oncharact}, for every~$n$ there will be
$0\leq\tau_{1,n}\leq\tau_{2,n}\leq T$, such that
\begin{equation}
\label{interf-gen-char-cond-n1}
    \zeta_n(t)=0,\qquad f_l(u_l(t))=f_r(u_r(t))=\gamma
    \qquad\ \forall~t\in [\tau_{1,n},\,\tau_{2,n}],
\end{equation}
and such that
the restriction of $\zeta_n$
     to $]0, \tau_{1,n}[$ and to $]\tau_{2,n}, T[$
     is either a classical genuine characteristic 
     for  $u_t+f_l(u)_x=0$\,
     on $\{x<0\}$, or it 
     is a classical genuine characteristic 
     for  $u_t+f_r(u)_x=0$\,
     on $\{x>0\}$.
     This, in particular, implies that
     \begin{equation}
     \label{eq:der-char-n}
         \dot\zeta_n(t)=
         \begin{cases}
             \dfrac{x_n}{T-\tau_{2,n}}
             \quad&\forall~t \in \,
              ]\tau_{2,n}, T[\,,
              \\
              \noalign{\medskip}
              -\dfrac{\zeta_n(0)}{\tau_{1,n}}
             \quad&\forall~t \in \,
              ]0,\tau_{1,n}[\,.
         \end{cases}
     \end{equation}
     Possibly considering a subsequence we can assume that $\{\tau_{i,n}\}_n$ converge to some $\tau_i\in [0,T]$, $i=1,2$, with $\tau_1\leq \tau_2$.
     Suppose that $\tau_1>0$, $\tau_2<T$. The
     cases where $\tau_1=0$, or/and $\tau_2=T$
     can be treated with entirely similar
     and simpler arguments.
     Up to extracting a further subsequence
     we may also assume that $x_n>0$
     for all $n$, and that
     \begin{equation}
     \label{interf-gen-char-cond-n2}
     \zeta_n(t)<0\qquad \forall~t\in [0, \tau_{1,n}[\,,\qquad\qquad
         \zeta_n(t)>0\qquad \forall~t\in \,]\tau_{2,n},\,T],  
         \quad\forall~n\,.
\end{equation}
Again, the cases where $\zeta_n(t)>0$ for all $t\in [0, \tau_{1,n}[\,$, or/and $x_n<0$, $\zeta_n(t)<0$ for all $t\in \,]\tau_{2,n},\,T]$, can be analyzed in an entirely similar way.
By the uniform convergence of $\zeta_n$ to $\zeta$
and since $\tau_{i,n}\to \tau_i$, $i=1,2$, it follows from~\eqref{interf-gen-char-cond-n1} that
\begin{equation}
\label{interf-gen-char-cond-1}
    \zeta(t)=0,\qquad f_l(u_l(t))=f_r(u_r(t))=\gamma
    \qquad\ \forall~t\in \,]\tau_{1},\,\tau_{2}[\,.
\end{equation}
and
\begin{equation}
\label{interf-gen-char-cond-2}
 \zeta(t)\leq 0\quad \ \forall~t\in [0, \tau_{1}],\qquad\quad
         \zeta(t)\geq 0\quad\ \forall~t\in [\tau_{2},\,T].
\end{equation}
Moreover, we have 
\begin{equation}
\label{interf-gen-char-cond-T}
    \zeta(T)=x,
\end{equation}
since $x_n\to x$ and
$x_n=\zeta_n(T)\to \zeta(T)$.
Note also that, because of~\eqref{eq:der-char-n}, there holds
\begin{equation}
     \label{eq:der-char}
         \dot\zeta(t)=\lim_n \dot\zeta_n(t)=
         \begin{cases}
             \dfrac{x}{T-\tau_{2}}
             \quad&\forall~t \in \,
              ]\tau_{2}, T[\,,
              \\
              \noalign{\medskip}
              -\dfrac{\zeta(0)}{\tau_{1}}
             \quad&\forall~t \in \,
              ]0,\tau_{1}[\,.
         \end{cases}
     \end{equation}
     
Now, if we assume that $x>0$, 
it follows from~\eqref{eq:der-char} that
that $\zeta(t)>0$ for all $t\in\,]\tau_2, T]$.
On the other hand, since 
uniform limit of classical genuine characteristics is a classical genuine characteristic as well,
we deduce that the restriction of $\zeta$ to 
$\,]\tau_2, T]$ is a classical genuine characteristic for $u_t+f_r(u)_x$. 

Next, if we assume that $x=0$, then
the uniform convergence of $\zeta_n$ to $\zeta$, together with~\eqref{interf-gen-char-cond-n2}, \eqref{eq:der-char}, imply that
\begin{equation}
\label{eq:limchar-der-2}
    \zeta(t)=\dot\zeta(t)=0\qquad \forall~t \in \,
              ]\tau_{2}, T[\,,
\end{equation}
and
\begin{equation}
\label{eq:limchar-der-3}
f'_r(u_r(t))=
\lim_n f'_r(u(\zeta_n(t),t))=\lim_n \dot\zeta_n(t)=0
    \qquad \forall~t \in \,
              ]\tau_{2}, T[\,.
\end{equation}
In turn, \eqref{eq:limchar-der-3} implies that $u_r(t)=\theta_r=B$ for all $t \in \, ]\tau_{2}, T[$, and that $(A,B)$ is a critical connection. On the other hand, because of 
the interface condition~\eqref{ABtraces}, it follows that $f_l(u_l(t))=f_r(u_r(t))=\gamma$
for all $t \in \, ]\tau_{2}, T[$,
which proves that the restriction of $\zeta$
to the interval $[\tau_2, T]$ satisfies the condition (ii) of Definition~\ref{interfacecharacteristics}.
With entirely similar arguments one can show that the restriction of $\zeta$
to the interval $[0,\tau_1]$ satisfies the condition (i) or (ii) of Definition~\ref{interfacecharacteristics},
which, together with~\eqref{interf-gen-char-cond-1}, completes the proof that $\zeta$ is
an $AB$-gic
belonging to the set $\mc C(u,x)$. This completes the proof of~\eqref{eq:graphclosureC}
whenever $\{x_n\}_n$ is a sequence 
converging to $x\geq 0$. The case where the limit point $x$ of $\{x_n\}_n$ is non positive can be treated in entirely similar way.
\smallskip


\textbf{3.} Proof of (iii). 
Let $\{x_n\}_n$ be a sequence 
converging to $x\in\R$, and let 
$\{y_n\}_n$
be a sequence of elements of $\mc C_0(u,x_n)$
converging to some point $y\in\R$.
Then, there will be a sequence of $AB$-gic 
$\zeta_n\in \mc C(u,x_n)$,
such that $y_n=\zeta_n(0)$ for all n.
Observe that by Definition~\ref{interfacecharacteristics}
 it follows that 
 \begin{equation}
     \label{eq:boundchar-derchar-n}
     |\zeta_n(t)|\leq |x_n|+LT,\qquad |\dot\zeta_n(t)|\leq
     L,\qquad\forall~t\in[0,T],\quad \forall~n\,,
 \end{equation}
 for some constant $L>0$ depending on $\|u\|_{\bf L^\infty}$. Hence, applying Ascoli-Arzel\`a Theorem, we deduce that up to a subsequence $\{\zeta_{n}\}_n$  converges uniformly to some $\zeta\in \mr{Lip}([0,T]\; ; \mathbb R)$.
 Thus, 
 in particular we have
 \begin{equation}
 \label{eq:initial-char}
     \zeta(0)=\lim_m\zeta_n(0)=\lim_n y_n=y.
 \end{equation}
Then, in view of property (ii) established 
at previous point, we find that $\zeta\in \mc C(u,x)$, and~\eqref{eq:initial-char} implies
$y\in \mc C_0(u,x)$, completing the proof of (iii). 
\smallskip

 \textbf{4.} Proof of (iv).
Given $x_1<x_2$, let $y_1=\max \mc C_0(u, x_1)$, and consider  $\zeta_1 \in \mc C(u, x_1)$ such that $\zeta_1(0) = y_1$. Choose any $\zeta_2 \in \mc C(u, x_2)$ 
and define 
$$
\zeta(t) \doteq \max \{\zeta_1(t), \, \zeta_2(t)\} \qquad t \in [0,T].
$$
Observe that, by definition the maximum of two $AB$-gic is still an $AB$-gic, and $\zeta(T) = x_2$, so that one has $\zeta \in \mc (u,x_2)$. Moreover:
$$
\max \mc C_0(u, x_1) =y_1 = \zeta_1(0) \leq \zeta(0) \leq \max \mc C_0(u, x_2).
$$
This proves (iv), and 
thus concludes the proof of the proposition.
\end{proof}

\smallskip
The next Proposition states that  
the $AB$-entropy solution $u^*$
defined in~\eqref{eq:sol-vertexinitialset} has always at least an $AB$-gic in common with every $AB$-entropy solutions $u$
satisfying $u(\cdot\,,T)=u^*(\cdot\,,T)$.
\begin{prop}\label{Clemma}
Given $\omega^T \in \mathcal{A}^{[AB]}(T)$,
let $u^*$ be the 
 $AB$-entropy solution defined by~\eqref{eq:vertexinitialset}-\eqref{eq:sol-vertexinitialset},
and let $u$ be any other 
 $AB$-entropy solution
to \eqref{conslaw} 
with initial datum $u_0\in  \mc I^{[AB]}_T(\omega^T)$.
Then, there holds 
\begin{equation}
\label{eq:Cuu^*}
\mathcal{C}(u^*, x) \cap \mathcal{C}(u, x) \neq \emptyset\qquad\forall~x\in\R\,.
\end{equation}
\end{prop}

\begin{proof}
To fix the ideas, we will assume that,
letting $\ms L, \ms R$ be the quantities
defined in~\eqref{eq:LR-def}, there holds
$\ms L=0$, $\ms R\in \,]0, T\cdot f'_r(B)[$\,, and that
$\omega^T$
fulfills the conditions (i)'-(ii)' of~\cite[Theorem 4.7]{talamini_ancona_attset}
for a non critical connection
$(A,B)$, which in particular require
    \begin{align}
    \label{eq:omegaT-cond-0}
         \omega^T(x-)&\geq \omega^T(x+)
         \qquad\forall~x\neq 0,
         \\
         \label{eq:omegaT-cond-1}
        \omega^T(x)&\geq B\qquad\qquad \forall~x\in\,]0,\ms R[\,.
    \end{align}
The cases where $\omega^T$ belongs to other classes of reachable profiles described in~\cite[Theorems 4.2, 4.7, 4.9]{talamini_ancona_attset} can be analyzed with entirely similar arguments.
Throughout the proof we  let $u_l(t), u_r(t)$, and $u^*_l(t), u^*_r(t)$, denote the one-sided traces of $u(t,\cdot)$
and $u^*(t,\cdot)$, respectively, at $x=0$.
\smallskip

\textbf{1.} 
Relying on the fact that any sequence
$\{\zeta_n\}_n$ of $AB$-gic 
(for $u^*$ and $u$) 
admits a subsequence uniformly convergent to some 
$\zeta\in \mr{Lip}([0,T]\; ; \mathbb R)$
(see point {\bf 3.} of the proof of Proposition~\ref{lemma:Cnotempty}),
and since the map
$x\to \mathcal{C}(u^*, x) \cap \mathcal{C}(u, x)$ has closed graph by
Proposition~\ref{lemma:Cnotempty}-(ii),
it will be sufficient to show that
$\mathcal{C}(u^*, x) \cap \mathcal{C}(u, x) \neq \emptyset$ holds for all point $x$ of continuity for~$\omega^T$.
Moreover, for every point $x\in\,]\!-\infty, 0[\ \cup\ ]\ms R, +\infty[$\, of continuity for $\omega^T$, 
there exists a unique $AB$-gic for $u^*$ and $u$ that reaches the point $x$ at time $T$,
    which is a classical genuine characteristic $\vartheta_x$
    for $u^*$ and $u$ (since it never crosses the interface $x=0$,
    but at most at $t=0$,
    by definition~\eqref{eq:LR-def}).
    Thus we have $\mathcal{C}(u^*, x) \cap \mathcal{C}(u, x)=\{\vartheta_x\}$ for all point 
    $x\in\,]\!-\infty, 0[\ \cup\ ]\ms R, +\infty[$\, of continuity for $\omega^T$. As a consequence, in order to establish~\eqref{eq:Cuu^*} it will
    be sufficient to show
    \begin{equation}
\label{eq:Cuu^*-2}
\mathcal{C}(u^*, x) \cap \mathcal{C}(u, x) \neq \emptyset\qquad\quad\text{for all $x\in\,]0, \ms R[\,$ of continuity for $\omega^T$}.
\end{equation}
To this end, given any $x\in\,]0, \ms R[\,$ of continuity for $\omega^T$, we consider the $AB$-gic
$\zeta\in \mathcal{C}(u, x)$
 defined in~\eqref{eq:char-int},
 with $\ol\tau$ as in~\eqref{eq:def-supE}
 and 
 \begin{equation}
\label{eq:tau-def-2}
    \tau(x) \doteq T-\frac{x}{f_r^{\prime}(\omega^T(x))},
\end{equation}
 in place of $\tau_-(x)$. We will show that $\zeta$ also belongs to $\mathcal{C}(u^*, x)$.
 Note that by definition of $\ms R$ at~\eqref{eq:LR-def} we have $\tau(x) >0$. 
  \smallskip
 
 \textbf{2.} 
We determine here  explicitly
the 
 $AB$-entropy solution $u^*$ defined by~\eqref{eq:vertexinitialset}-\eqref{eq:sol-vertexinitialset}
 when $\omega^T$
satisfies the conditions (i)'-(ii)' of~\cite[Theorem 4.7]{talamini_ancona_attset}
for a non critical connection,
with
$\ms L=0$, $\ms R\in \,]0, T\cdot f'_r(B)[$\,. These conditions require in particular that
\begin{align}
\label{eq:omegaT-cond0}
\omega^T(0-)&\geq\pi(\omega^T(0+)),
\\
\label{eq:omegaT-cond1}
    \omega^T(x)&\geq B,\qquad 
    \forall~x\in\,]0, \ms R[\,,
    \\
    \label{eq:omegaT-cond2}
    \omega^T(\ms R+)&\leq \bs u\,,
\end{align}
where 
\begin{equation}
    \pi(u)\doteq  ({f_l}_{\mid [\theta_l,+\infty)})^{-1} \circ f_r(u),\qquad u\in\R,
\end{equation}
and $\bs u\doteq \bs u[\ms R, B, f_r]$
is the quantity defined in~\cite[\S~3.1]{talamini_ancona_attset}
   that satisfies
   \begin{equation}
   \label{eq:prop-constraint-1}
       B>\bs u
       >\ol B,
       \qquad\quad
       f'_r(\bs u)<\ms R/T,
   \end{equation}
   (with $\ol B$  defined as in~\eqref{eq:bar-AB-def}).
Because of condition~\eqref{eq:omegaT-cond1}, according with the analysis in~\cite[\S~5.4]{talamini_ancona_attset} the solution $u^*$ contains a shock curve
starting at the interface $x=0$
and reaching the point $\ms R$ at time $T$,
which is parametrized by a map
    $\gamma : [\bs\tau, T]\to [0,\infty[\,$,  
    with the properties that $\gamma(\bs\tau)=0$,
    $\gamma(T)=\ms R$, 
    where $\bs\tau\doteq \bs \tau[\ms R, B, f_r]$
    is a quantity defined as
    in~\cite[\S~3.4]{talamini_ancona_attset}.
    The curve
     $t\to (\gamma(t), t)$ is the location of a shock  for the conservation law
    $u_t+f_r(u)_x=0$, which connects 
    the left state~$B$
    with the
    right states 
    $(f'_r)^{-1}\big(\big({\gamma(t)-\ms R+T \cdot f'_r(\bs u
    )}\big)/{t}\big)$, $t\in [\bs\tau, T]$.
    On the right of $\gamma(t)$ there is a rarefaction wave, connecting the left state~$\ol B$ 
    with the right state~$\bs u$,
    and centered at the point $(\ms R-T \cdot f'_r(\bs u), 0)$. Moreover, there holds
    \begin{equation}
    \label{eq:tau-id-1}
        \bs \tau=\dfrac{T\cdot f'_r(\bs u)-\ms R}{f'_r(\,\ol B)}.
    \end{equation}
    Following the procedure described in~\cite[\S~5.4]{talamini_ancona_attset}, in order to define $u^*$ we introduce some notations 
    for the polygonal lines along which 
    $u^*$ takes constant values
    in each region $\{x<0\}$, 
    $\{x>0\}$
    (that correspond to $AB$-gic for $u^*$).
    We define
\begin{equation}
\begin{aligned}
 \vartheta_{\rm 0, -}(t) &\doteq
 (t-T)\cdot f'_l(\omega^T(0-)),
    \\
    \vartheta_{\rm 0, +}(t) &\doteq
    (t-T)\cdot f'_l\big(\pi(\omega^T(0+))\big),
    \\
    \vartheta_{\rm R, -}(t) &\doteq
    \begin{cases}
        \ms R-(T-t) \cdot f'_r\big(\omega^T(\ms R-),
        \ \ &\text{if}\quad \tau_-(\ms R)\leq t\leq T,
        \\
        \noalign{\smallskip}
        \big(t-\tau_-(\ms R)\big)
 \cdot f'_l\circ\pi(\omega^T(\ms R-)),
 \ \  &\text{if}\quad 0\leq t\leq \tau_-(\ms R),
    \end{cases}
    \\
    \vartheta_{\rm R, +}(t) &\doteq
    \ms R-(T-t) \cdot f'_r\big(\omega^T(\ms R+)\big),
\end{aligned}
\end{equation}
and, for every $y \in \,]-\infty, 0\,[\,\cup \,]\ms R, +\infty[$\,, we define
    \begin{equation}
\label{eq:pollines}
     \vartheta_{y, \pm}(t) \doteq \!\begin{cases}
y-(T-t) \cdot f_l^{\prime}\big(\omega^T(y\pm)\big), & \text{ if \ $y < 0, \quad 0\leq t \leq T$},\\
\noalign{\smallskip}
  y-(T-t) \cdot f'_r\big(\omega^T(y\pm)\big), & \text{ if \ $0< y < {\ms R}, \quad \tau_{\pm}(y) 
  \leq t \leq  T$},\\
  \noalign{\smallskip}
 \big(t-\tau_{\pm}(y)\big)
 \cdot f'_l\circ\pi(\omega^T(y\pm)), & \text{ if \ $0< y < {\ms R}, \quad 0 \leq t < \tau_{\pm}(y),$}\\
 \noalign{\smallskip}
    y-(T-t) \cdot f'_r\big(\omega^T(y\pm)\big), & \text{ if \ $y>  \ms R, \quad 0\leq t \leq T$},
        \end{cases}
\end{equation}
    where
    \begin{equation}
    \label{eq:tau-def-pm}
    \begin{aligned}
    \tau_{\pm}(y)&\doteq
    T-\frac{y}{f'_r(\omega^T(y\pm))},\qquad y>0.
    \end{aligned}
    \end{equation}
    Moreover, letting $\{y_n\}_n$ denote the (at most) countably many discontinuity points of $\omega^T$
    in the intervals $]-\infty, 0]$,
    $]\ms R, +\infty[$, we set 
    \begin{equation*}
    \begin{aligned}
        \mc I^n_{0}
        &=\,]x_n^-,x_n^+[\,, \quad x_n^\pm=\vartheta_{y_n,\pm}(0),
\quad y_n\in \,]-\infty, 0]\,,
\\
\mathcal{I}^n_{\ms R}&=\,]x_n^-,x_n^+[\,, \quad x_n^\pm=\vartheta_{y_n,\pm}(0),
\quad y_n\in  \,]\ms R,+\infty[\,,
        \end{aligned}
    \end{equation*}
    (here we consider the possibility of a jump 
    of $\omega^T$ in $x=0$ when $\omega^T(0-)>\pi(\omega^T(0+))$).
    The intervals $\mc I^n_{0}$, $\mathcal{I}^n_{\ms R}$, consist of the starting points of compression waves in $u^*$ that generate a shock at $(y_n, T)$.
    
    Next, we introduce the polygonal lines
    connecting two points $(z,0)$, $(y,T)$
    (that correspond to compression fronts
    for $u^*$ generating a shock at the point 
    $(y,T)$) defined by
    \begin{equation}
        \eta_{y,z}\doteq
        \begin{cases}
            y -(T-t)\cdot \frac{(y-z)}{T},\quad &\text{if}\quad y\in\,]\!-\infty, 0]\,\cup\,]\mr R, +\infty[\,,\quad 0\leq t \leq T,
            \\
            \noalign{\smallskip}
             y -(T-t)\cdot f'_r(u_{y,z})
             \quad &\text{if}\quad 0<y<\ms R,
             \quad T-{y}/{f'_r(u_{yz})}\leq t\leq T,
             \\
            \noalign{\smallskip}
            \big(t-T+{y}/{f'_r(u_{yz})}\big)\cdot f'_l\circ \pi(u_{y,z})
            \quad &\text{if}\quad 0<y<\ms R,
             \quad 0\leq t<T-{y}/{f'_r(u_{yz})},
        \end{cases}
    \end{equation}
    where $u_{y,z}$ is the unique constant $u\geq (f'_r)^{-1}(y/T)$ satisfying 
    \begin{equation*}
        \Big(\frac{y}{f'_r(u)}-T\Big)\cdot f'_l\circ \pi(u)=z
    \end{equation*}
    (see~\cite[\S~5.4.1]{talamini_ancona_attset}).
    Finally, we set
    \begin{equation*}
    \begin{aligned}
        r_-(t)&\doteq 
                \ms R-T \cdot f'_r(\bs u)+ t \cdot f'_r(\,\ol B),
        \quad\ t\in [0,\bs \tau],
        \\
        \noalign{\smallskip}
        r_+(t)&\doteq   \ms R-(T-t) \cdot f'_r(\bs u),\quad\ t\in [0,T].
    \end{aligned}
    \end{equation*}
    %
    Then, the function $u^*$ defined by~\eqref{eq:vertexinitialset}-\eqref{eq:sol-vertexinitialset} is given by
\begin{equation}
    \label{eq:u^*-ex3}
        u^*(x,t)=
        \begin{cases}
\omega^T(y\pm), & \text{if \ $x=\vartheta_{y,\pm}(t)$ \ for some \ $y \in \,]-\infty, 0[\, \cup \,]\ms R, +\infty[\,$},\\
\noalign{\smallskip}
\omega^T(y\pm), & \text{if \ 
$x=\vartheta_{y,\pm}(t)>0$ \ for some \ $y \in \,]0, \ms R[ \,$},\\
\noalign{\smallskip}
\pi
(\omega^T(y\pm)), & \text{if \ 
$x=\vartheta_{y,\pm}(t)<0$ \ for some \ $y \in \,]0,\ms R[ \,$},\\
\noalign{\smallskip}
(f_r^{\prime})^{-1}\big(\frac{y_n-z}{T}\big), & \text{if \ $ x = \eta_{y_n, z}(t)$ \ for some \ $z \in \mc I^n_{\ms R}$},\\
\noalign{\smallskip}
(f_l^{\prime})^{-1}\big(\frac{y_n-z}{T}\big), & \text{if \ $ x = \eta_{y_n, z}(t)$ \ for some \ $z \in \mc I^n_{0}$},\\
\noalign{\smallskip}
\ B & \text{if}\ 
           \left\{
            \begin{aligned}
                &\vartheta_{\rm R, -}(t)\leq x<\gamma(t), \ \ t\in [\tau_-(\ms R), T],
            \\
            \noalign{\smallskip}
            &0<x<\gamma(t),
            \ \ t\in [\bs\tau, \tau_-(\ms R)],
            \end{aligned}
            \right.
            \\
\noalign{\smallskip}
\ \ol A  & \text{if}\ \ \vartheta_{\rm R, -}(t)\leq x<0, \ \ t\in [0,\tau_-(\ms R)],
\\
\noalign{\smallskip}
\ \ol B  & \text{if}\ \ 0<x\leq 
r_-(t), \ \ t\in [0, \bs \tau],
\\
\noalign{\smallskip}
(f'_r)^{-1}\Big(\dfrac{x-\ms R+T \cdot f'_r(\bs u
            )}{t}\Big)
            \ \ &\text{if}\ 
            \left\{
            \begin{aligned}
                &\gamma(t)<x<r_+(t), \ \ t\in [\bs \tau, T],
            \\
            \noalign{\smallskip}
            &r_-(t)<x<r_+(t),
            \ \ t\in [0,\bs \tau],
            \end{aligned}
            \right.
\\
\noalign{\smallskip}
(f'_r)^{-1}\big(\frac{\ms R-x}{T-t}\big) & \text{if}\ \ 
r_+(t)\leq x\leq \vartheta_{\rm R, +}(t), \ \ t\in [0, T[\,.
\end{cases}
\end{equation} 
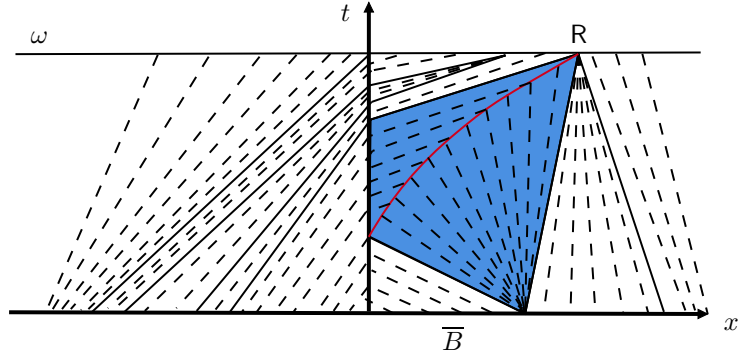
\begin{figure}
    \centering
    
\tikzset{every picture/.style={line width=0.75pt}} 

\begin{tikzpicture}[x=0.75pt,y=0.75pt,yscale=-0.7,xscale=0.7]

\draw  [fill={rgb, 255:red, 74; green, 144; blue, 226 }  ,fill opacity=0.64 ] (526.47,84.33) -- (488.13,271.67) -- (375.8,216) -- (376.8,131.67) -- cycle ;
\draw [line width=1.5]    (375.54,269.96) -- (375.54,51.2) ;
\draw [shift={(375.54,47.2)}, rotate = 90] [fill={rgb, 255:red, 0; green, 0; blue, 0 }  ][line width=0.08]  [draw opacity=0] (6.97,-3.35) -- (0,0) -- (6.97,3.35) -- cycle    ;
\draw [line width=1.5]    (615.67,271.86) -- (116.33,270.62) ;
\draw [shift={(619.67,271.87)}, rotate = 180.14] [fill={rgb, 255:red, 0; green, 0; blue, 0 }  ][line width=0.08]  [draw opacity=0] (6.97,-3.35) -- (0,0) -- (6.97,3.35) -- cycle    ;
\draw    (614.6,83.3) -- (120.93,84.5) ;
\draw  [dash pattern={on 4.5pt off 4.5pt}]  (492.13,105) -- (488.13,271.67) ;
\draw  [dash pattern={on 4.5pt off 4.5pt}]  (512.13,93) -- (488.13,271.67) ;
\draw  [dash pattern={on 4.5pt off 4.5pt}]  (431.47,149.67) -- (488.13,271.67) ;
\draw  [dash pattern={on 4.5pt off 4.5pt}]  (526.47,84.33) -- (520.8,271) ;
\draw  [dash pattern={on 4.5pt off 4.5pt}]  (446.13,135.67) -- (488.13,271.67) ;
\draw  [dash pattern={on 4.5pt off 4.5pt}]  (526.47,84.33) -- (501.47,271) ;
\draw  [dash pattern={on 4.5pt off 4.5pt}]  (526.47,84.33) -- (538.8,271) ;
\draw  [dash pattern={on 4.5pt off 4.5pt}]  (526.47,84.33) -- (558.8,271) ;
\draw  [dash pattern={on 4.5pt off 4.5pt}]  (460.13,124.33) -- (488.13,271.67) ;
\draw  [dash pattern={on 4.5pt off 4.5pt}]  (526.47,84.33) -- (576.8,271) ;
\draw  [dash pattern={on 4.5pt off 4.5pt}]  (475.47,113.67) -- (488.13,271.67) ;
\draw    (526.47,84.33) -- (488.13,271.67) ;
\draw    (526.47,84.33) -- (588.13,270.33) ;
\draw    (526.47,84.33) -- (376.8,131.67) ;
\draw    (375.47,216.33) -- (488.13,271.67) ;
\draw [color={rgb, 255:red, 208; green, 2; blue, 27 }  ,draw opacity=1 ]   (526.47,84.33) .. controls (470.13,117) and (466.8,117.67) .. (438.8,141) .. controls (410.8,164.33) and (392.13,187) .. (375.47,216.33) ;
\draw  [dash pattern={on 4.5pt off 4.5pt}]  (433.47,145) -- (375.8,166.67) ;
\draw  [dash pattern={on 4.5pt off 4.5pt}]  (460.13,124.33) -- (378.13,154.33) ;
\draw  [dash pattern={on 4.5pt off 4.5pt}]  (390.8,192.33) -- (373.47,198.33) ;
\draw  [dash pattern={on 4.5pt off 4.5pt}]  (418.13,159) -- (488.13,271.67) ;
\draw  [dash pattern={on 4.5pt off 4.5pt}]  (492.13,105) -- (375.8,143.67) ;
\draw  [dash pattern={on 4.5pt off 4.5pt}]  (400.8,177.67) -- (374.13,187.67) ;
\draw  [dash pattern={on 4.5pt off 4.5pt}]  (414.13,162.33) -- (376.8,176.33) ;
\draw  [dash pattern={on 4.5pt off 4.5pt}]  (406.13,173.67) -- (488.13,271.67) ;
\draw  [dash pattern={on 4.5pt off 4.5pt}]  (394.8,188.33) -- (488.13,271.67) ;
\draw  [dash pattern={on 4.5pt off 4.5pt}]  (384.8,201) -- (488.13,271.67) ;
\draw  [dash pattern={on 4.5pt off 4.5pt}]  (393.6,271.55) -- (372.8,261) ;
\draw  [dash pattern={on 4.5pt off 4.5pt}]  (421.13,271.98) -- (375.47,251.67) ;
\draw  [dash pattern={on 4.5pt off 4.5pt}]  (464.18,270.95) -- (376.13,228.33) ;
\draw  [dash pattern={on 4.5pt off 4.5pt}]  (442.73,271.24) -- (377.47,240.33) ;
\draw  [dash pattern={on 4.5pt off 4.5pt}]  (473.95,85.39) -- (375.47,115) ;
\draw  [dash pattern={on 4.5pt off 4.5pt}]  (395.64,85.62) -- (375.47,91.67) ;
\draw  [dash pattern={on 4.5pt off 4.5pt}]  (378.13,154.33) -- (294.8,270.33) ;
\draw  [dash pattern={on 4.5pt off 4.5pt}]  (375.8,143.67) -- (284.13,270.33) ;
\draw  [dash pattern={on 4.5pt off 4.5pt}]  (375.47,115) -- (236.8,267.67) ;
\draw  [dash pattern={on 4.5pt off 4.5pt}]  (376.13,111) -- (214.13,269.67) ;
\draw  [dash pattern={on 4.5pt off 4.5pt}]  (377.47,96.33) -- (192.13,268.33) ;
\draw  [dash pattern={on 4.5pt off 4.5pt}]  (473.95,85.39) -- (376.13,111) ;
\draw  [dash pattern={on 4.5pt off 4.5pt}]  (376.13,127) -- (264.8,269) ;
\draw  [dash pattern={on 4.5pt off 4.5pt}]  (438.13,85) -- (377.47,96.33) ;
\draw    (376.13,107) -- (199.47,270.33) ;
\draw  [dash pattern={on 4.5pt off 4.5pt}]  (502.8,83.67) -- (376.13,127) ;
\draw  [dash pattern={on 4.5pt off 4.5pt}]  (375.47,91.67) -- (187.47,267.67) ;
\draw    (473.95,85.39) -- (377.47,119) ;
\draw    (377.47,119) -- (251.47,269.67) ;
\draw    (473.95,85.39) -- (376.13,107) ;
\draw    (375.96,84.65) -- (176.13,270.33) ;
\draw    (376.8,131.67) -- (275.47,270.33) ;
\draw  [dash pattern={on 4.5pt off 4.5pt}]  (375.47,251.67) -- (361.8,269.33) ;
\draw  [dash pattern={on 4.5pt off 4.5pt}]  (373.47,198.33) -- (320.8,271) ;
\draw  [dash pattern={on 4.5pt off 4.5pt}]  (376.8,176.33) -- (308.13,269.67) ;
\draw  [dash pattern={on 4.5pt off 4.5pt}]  (373.47,235.67) -- (350.13,269.67) ;
\draw  [dash pattern={on 4.5pt off 4.5pt}]  (375.47,216.33) -- (337.8,270.67) ;
\draw  [dash pattern={on 4.5pt off 4.5pt}]  (260.13,86.33) -- (147.8,268.33) ;
\draw  [dash pattern={on 4.5pt off 4.5pt}]  (296.8,86.33) -- (154.47,268.33) ;
\draw  [dash pattern={on 4.5pt off 4.5pt}]  (323.47,85) -- (160.8,268.33) ;
\draw  [dash pattern={on 4.5pt off 4.5pt}]  (347.8,84.33) -- (168.13,269) ;
\draw  [dash pattern={on 4.5pt off 4.5pt}]  (362.47,84.67) -- (173.47,268.33) ;
\draw  [dash pattern={on 4.5pt off 4.5pt}]  (223.47,85) -- (143.13,269.67) ;
\draw  [dash pattern={on 4.5pt off 4.5pt}]  (553.6,84.3) -- (610.6,269.3) ;
\draw  [dash pattern={on 4.5pt off 4.5pt}]  (572.6,84.3) -- (619.67,271.87) ;
\draw  [dash pattern={on 4.5pt off 4.5pt}]  (536.63,86.11) -- (603.6,269.3) ;

\draw (518.47,61.4) node [anchor=north west][inner sep=0.75pt]  [font=\footnotesize]  {$\mathsf{R}$};
\draw (129.8,66.4) node [anchor=north west][inner sep=0.75pt]  [font=\footnotesize]  {$\omega $};
\draw (629.13,272.73) node [anchor=north west][inner sep=0.75pt]  [font=\footnotesize]  {$x$};
\draw (353.13,46.73) node [anchor=north west][inner sep=0.75pt]  [font=\footnotesize]  {$t$};
\draw (426.13,278.71) node [anchor=north west][inner sep=0.75pt]  [font=\footnotesize]  {$\overline{B}$};

\end{tikzpicture}

    \caption{The solution $u^*$}
    \label{fig:u*lemma}
\end{figure}
Observe that the left and right traces of $u^*$ satisfy
\begin{equation}
\label{eq:ul*cond1}
\begin{aligned}
        u^*_l(t)&\geq \ol A,\qquad u^*_r(t)\geq B,
    \qquad \forall~t\in\,]\tau_-(\ms R),T],
    \\
    u^*_l(t)&= \ol A,\qquad u^*_r(t)= B,
    \qquad \forall~t\in\,]\bs\tau,\tau_-(\ms R)],
    \\
    u^*_l(t)&= \ol A,\qquad u^*_r(t)= \ol B,  \qquad \forall~t\in\,]0,\bs \tau].
\end{aligned}
\end{equation}
Moreover, since $x$ is a point of continuity for $\omega^T$, 
it follows that 
the restriction of $\zeta$ 
to $\,]\tau(x), T]$
is a (classical) genuine characteristic both for $u$
and $u^*$ 
with slope $f'_r(\omega^T(x))>0$
so that, 
recalling~\eqref{ABtraces}, there holds
\begin{equation}
\label{eq:ur*cond1}
    u_r(\tau(x))=\omega^T(x)=u^*_r(\tau(x))
    >\theta_r,
    \qquad\quad
    f_l(u_l(\tau(x)))=f_r(u^*_r(\tau(x))). 
\end{equation}
Now, we will distinguish two cases
according with the position of $\ol\tau$ with respect to the time 
$\tau_-(\ms R)$ defined as in~\eqref{eq:tau-def-pm}.
Note that by definition of $\ms R$ at~\eqref{eq:LR-def} we have 
$\tau_-(\ms R) \geq 0$.
 \smallskip
 
 \textbf{3.} 
 Assume that $\overline\tau \geq \tau_-(\ms R)$, and suppose first that $\tau(x)=\ol \tau> \tau_-(\ms R)$. Note that, since $\ol\tau>0$
 is an element of the set $E$
 in~\eqref{eq:def-E}, and
 because of
 ~\eqref{eq:ul*cond1}, 
 \eqref{eq:ur*cond1},
 we have $u_l(\tau(x))=u^*_l(\tau(x))>\theta_l$. Therefore, also
 the restriction of $\zeta$ 
to $[0,\tau(x)[\,$
is a (classical) genuine characteristic both for $u$
and $u^*$. Hence, when 
$\tau(x)=\ol \tau>\tau_-(\ms R)$,
the map $\zeta$ in~\eqref{eq:char-int}
 is an $AB$-gic also for~$u^*$,
 proving that $\zeta\in \mathcal{C}(u^*, x) \cap \mathcal{C}(u, x)$.
 
 Next, consider the subcase 
 $\tau(x)>\ol\tau\geq \tau_-(\ms R)$, and 
observe that by~\eqref{eq:intcond-zeta-1}, \eqref{eq:trace-ur-tau}, 
\eqref{eq:ur*cond1},
we have
\begin{equation}
\label{eq:urtrace}
u_l(t)=A,\quad
u_r(t)=B\qquad\forall~t\in\,]\,\ol\tau, \tau(x)],
\qquad\quad u^*_r(\tau(x))=B\,.
\end{equation}
Moreover, we claim that
\begin{equation}
\label{eq:clu*trace}
\begin{aligned}
     \omega^T(z)&=B,\qquad\forall~z\in
    [x,\,\ol x[\,,
    \quad \ol x \doteq 
    (T-\ol\tau)\cdot f'_r(B),
    \\
    \noalign{\smallskip}
    u_r^*(t)&=B,
    \qquad\forall~t\in \,]\,\ol\tau,\tau(x)]\,.
\end{aligned}
\end{equation}
Note that the first equality in~\eqref{eq:clu*trace}  implies the second one by tracing the backward (genuine) characteristics for $u^*$ 
at time $T$, from points $z\in [x,\,\ol x[\,$.
In order to prove the first equality in~\eqref{eq:clu*trace}, we trace 
the minimal backward characteristic 
$\vartheta_{z,-}$ for the solution~$u$, at time~$T$, from points $z \in [x, \ol x]$.
Observe that 
$\vartheta_{z,-}$
impacts the interface $x=0$ at time 
$\tau_-(z)\doteq T-{z}/{f'_r(\omega^T(z-))}$.
Moreover, since
$\omega^T(\ms R-)\geq B$ because of~\eqref{eq:omegaT-cond-1},
we deduce from $\ol \tau\geq \tau_-(\ms R)$ that
\begin{equation}
    \label{eq:barx}
    \ol x \leq  \ms R.
\end{equation}
Furthermore, 
we  have
\begin{equation}
    \label{eq:omega-B-1}
    \tau_-(z)\geq 
T-\frac{z}{f'_r(B)}\geq 
\ol\tau \qquad\ \  \forall~z \in [x,\,
\ol x],
\end{equation}
since
$\omega^T(z-) \geq B$
by virtue of~\eqref{eq:omegaT-cond-1}. 
On the other hand, by~\eqref{eq:tau-def-2}
we know that
the (genuine) characteristic $\vartheta_x$ for $u$, starting at time $T$ from the point $x$,
reaches the interface $x=0$ at time $\tau(x)$.
Since $\vartheta_x$, $\vartheta_{z,-}$
are (classical) genuine characteristics 
that cannot cross in the
domain $\{x>0\}$, it follows that
\begin{equation}
\label{eq:omega-B-2}
    \tau(x)\geq \tau_-(z)
    \qquad\ \  \forall~z \in [x,\,
    \ol x]\,.
\end{equation}
Combining together~\eqref{eq:omega-B-1}, \eqref{eq:omega-B-2},
we deduce that, 
for every $z\in [x,\,\ol x]$,
the minimal backward characteristics $\vartheta_{z,-}$
reaches the interface $x=0$ 
at time $\tau_-(z)\in [\,\ol\tau,\tau(x)]$.
Hence, because of~\eqref{eq:urtrace},
we find that for all $z\in [x,\,\ol x]$
there holds $\omega^T(z-) = u_r(\tau_-(z))=B$, and this 
yields the first equality in~\eqref{eq:clu*trace}, concluding the proof of claim \eqref{eq:clu*trace}.

 Relying on~\eqref{eq:clu*trace}
we will show now that 
$\zeta$ 
satisfies the condition of an $AB$-gic
also for $u^*$
on the interval $[0, \tau(x)]$.
To this end, observe that~\eqref{eq:ul*cond1} \eqref{eq:clu*trace} together imply
\begin{equation}
    \label{eq:cru*trace-2}
    u_l^*(t)=\ol A,
    \qquad\forall~t\in 
    \,]\,\ol\tau,\tau(x)]\,.
\end{equation} 
Hence, because of~\eqref{eq:clu*trace},
\eqref{eq:cru*trace-2}, $\zeta$ satisfies the condition (ii) of Definition~\ref{interfacecharacteristics}
of an $AB$-gic for $u^*$ on the interval
$\,]\ol\tau, \tau(x)]$.
Moreover, let $\vartheta_{t_n}^*$
denote the (classical genuine) backward characteristic for $u^*$, on the
region $\{x<0\}$, starting at time $t_n\in \,]\,\ol\tau,\tau(x)]\,$
from $x=0$,
for a sequence $t_n\downarrow \ol\tau$.
Note that, because of~\eqref{eq:cru*trace-2}, all 
$\vartheta_{t_n}^*$ have slope $f_l'(\,\ol A)$. Thus $\{\vartheta_{t_n}^*\}_n$
converges uniformly to a function $\vartheta^*: [0,\ol\tau]\to\R$
that is as well a (classical) genuine characteristic 
for~$u^*$ with slope $f_l'(\,\ol A)$ and such that $\vartheta^*(\ol\tau)=0$.
This in turn implies that
\begin{equation}
    \label{eq:cluu*trace-1}
u^*_l(\ol\tau) = 
\ol A\,.
\end{equation}

Next, we will prove that 
\begin{equation}
\label{eq:cluu*trace-2}
    u_l(\ol\tau) 
    = \ol A\,.
\end{equation}
%
 To this end
 observe that~\eqref{eq:omegaT-cond-0}, 
 \eqref{eq:omegaT-cond-1}, 
 \eqref{eq:clu*trace}, \eqref{eq:barx}
 together imply
$\omega^T(\ol x-)= B=\omega^T(\ol x+)$.
This means that the 
   characteristic
 $\vartheta_{\ol x}$ starting
 at time~$T$ from~$\ol x$,
 and reaching $x=0$ at time~$\ol \tau$,
is a (classical) genuine characteristic for $u$
(on the semiplane $\{x>0\}$),
and hence we deduce that 
\begin{equation}
    \label{eq:crutrace-2}
    u_r(\ol \tau)=\omega^T(\ol x)=B.
\end{equation}
Recalling~\eqref{eq:supE=maxE}
and the definition~\eqref{eq:def-E}
of the set $E$, we derive from~\eqref{eq:crutrace-2}
and from 
condition~(2) of Definition~\ref{ABsol}
that 
\begin{equation}
    \label{eq:ltracebartau}
    u_l(\ol\tau)>\theta_l.
\end{equation}
In turn, condition~\eqref{eq:ltracebartau}, together with~\eqref{eq:urtrace},
implies~\eqref{eq:cluu*trace-2} by a blow-up argument as in~\cite[\S~5.2.6]{talamini_ancona_attset}. 
Namely, we can consider 
the blow ups 
of $u$ at the point $(0, \ol\tau\,)$:
\begin{equation}
\label{eq:bup-def}
    u_n(x,t) \doteq u\big(x/n,\, \ol\tau +{(t-\ol\tau)}/{n}\big)
    \qquad x\in\R, \ t\geq 0\,,\ \ n\in\mathbb{N},
\end{equation} 
and observe that,
because of~\eqref{eq:urtrace},
the left and right traces of $u_n(\cdot, t)$ at $x=0$
satisfy
\begin{equation}
\label{eq:tracesit-3}
    (u_{n,l}(t), u_{n,r}(t))=(A,B)\quad\qquad\forall~t\in\big]\ol\tau, \, \ol\tau +n\big(\tau(x)-\ol\tau\big)\big[\,.
\end{equation}
When $n\to\infty$,
up to a subsequence, 
the blow ups $u_n(\cdot, t)$
converge in ${\bf L^1_{loc}}$  to a limiting $AB$ entropy solution $v(\cdot, t)$,
for all $t>0$,
and there holds 
\begin{equation}
\label{eq:indata-blup}
    v(x, \ol\tau\,)= 
\begin{cases}
    u_l(\ol\tau), & \text{if $x < 0$},\\
    u_r(\ol\tau), & \text{if $x > 0$},
\end{cases} 
\end{equation}
\begin{equation}
\label{eq:tracesit}
    v(0-,t)\in\{A,\ol A\,\},\qquad\qquad v(0+,t)\in\{B,\ol B\,\},\qquad\forall~t>\ol\tau\,.
\end{equation}
Then, by a direct inspection
we find that, if 
an $AB$ entropy solution
of a Riemann problem for~\eqref{conslaw}, with initial datum~\eqref{eq:indata-blup}
at time $\ol\tau$, enjoys the properties
~\eqref{eq:ltracebartau}, \eqref{eq:tracesit}, 
it follows that the left initial datum 
at time $\ol\tau$ 
 must be 
\begin{equation*}
    v(x,\ol\tau)=u_l(\ol\tau) 
 = \ol A,\qquad \forall~x<0,
\end{equation*} 
thus proving~\eqref{eq:cluu*trace-2}.

The two conditions~
equalities~\eqref{eq:cluu*trace-1}, 
\eqref{eq:cluu*trace-2} 
and the definition~\eqref{eq:char-int}
imply that 
the restriction of~$\zeta$ 
to~$[0, \ol\tau[\,$
is a (classical) genuine characteristic both for $u$
and $u^*$ 
with slope $f'_l(\,\ol A)>0$.
Therefore we can conclude that
$\zeta$ 
satisfies the condition of an $AB$-gic
also for $u^*$
on the interval $[0, \tau(x)]$,
and hence on the whole interval $[0,T]$ by the analysis in the point {\bf 2}.
This completes the proof that $\zeta\in \mathcal{C}(u^*, x) \cap \mathcal{C}(u, x)$ when $\ol\tau\geq \tau_-(\ms R)$.
 
\smallskip

 \textbf{4.} 
 Assume that $\overline\tau <\tau_-(\ms R)$,
 with $\tau_-(\ms R)$ as in~\eqref{eq:tau-def-pm}.
If we suppose that~\eqref{eq:ltracebartau} 
holds, since~\eqref{eq:urtrace} is still verified we can deduce as above 
that \eqref{eq:cluu*trace-2} holds as well,
and then we conclude that $\zeta\in \mathcal{C}(u^*, x) \cap \mathcal{C}(u, x)$
with the same arguments of point  {\bf 3}.

Therefore, let us assume that $u_l(\ol \tau) \leq \theta_l$ and that $\zeta(t) > 0$
for all $t\in [0, \ol\tau[$\,. 
Observe that because of~\eqref{eq:supE=maxE},
and by definition~\eqref{eq:def-E}
of the set $E$, we have 
\begin{equation}
    \label{eq:rtracebartau}
    u_r(\ol\tau)<\theta_r.
\end{equation}
Then, relying on~\eqref{eq:urtrace},
we deduce with the same blow up argument of above that 
\begin{equation}
\label{eq:cruu*trace-2}
    u_r(\ol\tau) 
    = \ol B\,.
\end{equation}
On the other hand, if we show that
\begin{equation}
\label{eq:tau-bar-tau-ineq}
    \ol\tau \leq \bs \tau,
\end{equation}
it would follow from~\eqref{eq:ul*cond1} that
\begin{equation}
\label{eq:cruu*trace-3}
    u^*_r(\ol\tau)= \ol B\,.
\end{equation}
The two conditions~\eqref{eq:cruu*trace-2},
\eqref{eq:cruu*trace-3} 
and the definition~\eqref{eq:char-int}
imply that 
the restriction of~$\zeta$ 
to~$[0, \ol\tau[\,$
is a (classical) genuine characteristic both for $u$
and $u^*$ 
with slope $f'_r(\,\ol B)<0$.
Therefore, if~\eqref{eq:tau-bar-tau-ineq}
holds, we can conclude that
$\zeta$ 
satisfies the condition of an $AB$-gic
also for $u^*$
on the interval $[0, \tau(x)]$,
and hence on the whole interval $[0,T]$ by the analysis in the point {\bf 2}.
Hence, in order to completes the proof that $\zeta\in \mathcal{C}(u^*, x) \cap \mathcal{C}(u, x)$ when $\ol\tau<\tau_-(\ms R)$,
it remains to establish~\eqref{eq:tau-bar-tau-ineq}.

By contradiction, assume that $\ol\tau > \bs \tau$. Define the curve

$$
\xi(t) \, \doteq \, \inf \Big\{ R >0 \; \big |\; x-t \, f_r^{\prime}(u(x,t)) \geq 0 \ \ \ \forall~R>0\Big\} \qquad t \in [\ol\tau, T]. 
$$
Notice that, because of~\eqref{eq:cruu*trace-2}, we have
$\xi(\ol\tau) =0$,
while the definition~\eqref{eq:LR-def} yields $\xi(T) = \ms R$. But now using a comparison argument between $\xi(t)$ and the map $\gamma(t)$ 
defining the shock curve of $u^*$ at point {\bf 2},
we obtain as in~\cite[\S 5.2.3]{talamini_ancona_attset}
that $\xi(t)<\gamma(t)$
for all $t\in [\ol\tau, T]$.
Thus, we find in particular
that
$\xi(T)  <\gamma(T)=\ms R$, 
which gives a contradiction. This concludes the proof of the proposition.

\end{proof}

\smallskip
The next Lemma shows that the initial positions
of the  $AB$-gics of 
the $AB$-entropy solution~$u^*$ defined in~\eqref{eq:sol-vertexinitialset} provide a partition of $\R$.

\begin{lemma}\label{lemma:Csurjective}
Given $\omega^T \in \mathcal{A}^{[AB]}(T)$,
let $u^*$ be the 
 $AB$-entropy solution defined by~\eqref{eq:vertexinitialset}-\eqref{eq:sol-vertexinitialset}.
Then, there holds 
\begin{equation}
\label{eq:partition-R}
    \mathbb R = \bigcup_{x \in \mathbb R} \mc C_0(u^*, x).
\end{equation}
\end{lemma}
\begin{proof}
By the analysis in \cite[\S 5.4.3]{talamini_ancona_attset} we deduce that
$\R$ can be partitioned as the union of sets 
containing points of  three types: 
\begin{itemize}
[leftmargin=25pt]
    \item[-]
    starting points of compression fronts (possibly refracted by the interface $x=0$) which meet together generating a shock at time $T$; 
    \item[-]
    starting points of classical genuine characteristics or of polygonal lines made of two segments consisting of classical genuine characteristics in each semiplane $\{x<0\}$,
    $\{x>0\}$, which reach at time $T$ a point of continuity of $\omega^T$.
    \item[-] starting points $y$ of polygonal lines $\xi:[0,T] \to \mathbb R$ with $\xi(0) = y$, composed of three segments of the form 
    $$
    \xi(t) = \begin{cases}
        y+t \, f^\prime(u(t,\xi(t)), \xi(t)), & \text{if $0 \leq t \leq t_1$},\\
        0 , & \text{if $t_1 \leq t \leq t_2$},\\
        (t-t_2) f^\prime(u(t,\xi(t)), \xi(t)) & \text{if $t_2 \leq t \leq T$} 
    \end{cases}
    $$
    where 
    $$
    f(u(t, 0\pm)) = \gamma \qquad \forall t \in (t_1, t_2).
    $$
    Notice that these polygonal lines may belong to the near-interface regions $\Delta_{\ms L}, \Gamma_{\ms R}$ defined in (\cite{talamini_ancona_attset}, \S 5.4.4).
\end{itemize}
In all cases they are starting points of segments
or of polygonal lines which are $AB$-gics for $u^*$,
and the result follows.
\end{proof}

\medskip


We introduce now a functional that, 
for any given function $v(x,t)$, 
measures the total amount of flux of the vector field $\big(f(x,v(x,t)),\, v(x,t)\big)$ passing through a curve
$t\mapsto (\alpha(t),t)$, from each side
of the curve. 
\begin{defi}
Given a function $v \in \mathbf L^{\infty}(\mathbb R \times [0,T]\, ; \, \mathbb R)$   that admits one-sided limits $v(x\pm,t)$ at every point $(t,x) \in \,]0, T] \times \mathbb{R}$, and  $\alpha\in  \mr{Lip}([0,T]\, ; \, \mathbb R)$,
we define
\begin{equation}
\label{eq:def-F}
\mathcal{F}_t(\alpha\pm, v) \doteq \int_t^T \Big\{f\big(\alpha(t)\pm, v(\alpha(t) \pm,t)\big) - \dot{\alpha}(t)\,v(\alpha(t) \pm,t)\Big\} \dif t,\qquad t\in [0,T],
\end{equation}
where $f(x,u)$ is the flux \eqref{discflux}.
We also set 
\begin{equation}
\mathcal F(\alpha \pm, v) \doteq\mathcal F_0(\alpha \pm, v).
\end{equation}
\end{defi}

\begin{remark}
\label{rem:def-F-on-gic}
Notice that if $u$ is an $AB$-entropy  solution of \eqref{conslaw}, 
since $u$ is in particular a distributional solution of~\eqref{conslaw} on $\R\times \,]0,+\infty[$,
it follows that 
for any curve $\alpha\in  \mr{Lip}
([0,T]\, ; \, \mathbb R)$, 
the Rankine-Hugoniot conditions yield,
for a.e. $t\in [0,T]$, the equality
\begin{equation}
\label{eq:RH}
    f\big(\alpha(t)-, u(\alpha(t)-,t)\big)-\dot{\alpha}(t)\,u(\alpha(t)-,t) = f\big(\alpha(t)+, u(\alpha(t)+,t)\big)-\dot{\alpha}(t)\,u(\alpha(t)+,t).
\end{equation}
Therefore, in this case we have
$\mc F(\alpha+, u)= \mc F(\alpha-, u)$. Hence, since there is no ambiguity, 
 whenever $u$ is an $AB$-entropy  solution of~\eqref{conslaw},
 we will simply write
$$\mc F_t(\alpha, u) \doteq \mc F_t(\alpha+, u) \equiv \mc F_t(\alpha-, u) \ \ \ \forall~t, \qquad 
\mc F(\alpha, u) \doteq \mc F(\alpha+, u) \equiv \mc F(\alpha-, u).
$$
\end{remark}

\begin{lemma}\label{Flemma}
Let $u,u^* \in \mathbf L^{\infty}(\mathbb R \times [0,T]\,;\,\R)$ be $AB$-entropy solutions to  \eqref{conslaw}, and let $\zeta \in \mathcal{C}(u^*,x)$,
$x \in \mathbb R$. Then, there holds
\begin{equation}
\label{eq:Fest1}
\mathcal{F}_t(\zeta, u) \geq \mathcal{F}_t(\zeta, u^*), \quad \forall \; t \in [0,T].
\end{equation}
Moreover, one has 
\begin{equation}
\label{eq:Fest2}
    \mathcal{F}(\zeta, u) = \mathcal{F}(\zeta, u^*)
    \quad \Longleftrightarrow\quad 
    \zeta \in \mc C (u^*,x) \cap \mc C(u,x).
\end{equation}
\end{lemma}

\begin{proof}
Let $0\leq\tau_1\leq\tau_2\leq T$,
be the partition of $[0,T]$ 
for $\zeta\in \mathcal{C}(u^*,x)$ given
by Remark~\ref{rem:oncharact},
so that there holds
\begin{equation*}
\label{interf-gen-char-cond-5}
    \zeta(t)=0,\qquad f_l(u^*_l(t))=f_r(u^*_r(t))=\gamma
    \qquad\ \forall~t\in [\tau_1,\,\tau_2].
\end{equation*}
To fix the ideas we assume 
that 
\begin{equation}
     \label{interf-gen-char-cond-n4}
     \zeta(t)<0\qquad \forall~t\in [0, \tau_1[\,,\qquad\quad
         \zeta(t)>0\qquad \forall~t\in \,]\tau_2,\,T].
\end{equation}
The cases where  $\zeta(t)>0$ for all $t\in [0, \tau_1[\,,$ or $\zeta(t)<0$
for all $t\in \,]\tau_2,\,T]$ are entirely similar.
Then, setting
\begin{equation}
\label{eq:ggstardef}
\begin{aligned}
    g(t)&\doteq
    \begin{cases}
        f_r(u(\zeta(t)+,t))-\dot{\zeta}(t)\, u(\zeta(t)+,t)
        \quad &\text{if}\quad t\in \,]\tau_2,\,T],
        \\
        \noalign{\smallskip}
        f_r(u_r(t))
        \quad &\text{if}\quad t\in [\tau_1, \tau_2],
        \\
        \noalign{\smallskip}
        f_l(u(\zeta(t)+,t))-\dot{\zeta}(t)\, u(\zeta(t)+,t)
        \quad &\text{if}\quad t\in [0,\tau_1[\,,
    \end{cases}
    \\
    \noalign{\medskip}
    g^*(t)&\doteq
    \begin{cases}
        f_r(u^*(\zeta(t),t))-\dot{\zeta}(t)\, u^*(\zeta(t),t)
        \quad &\text{if}\quad t\in \,]\tau_2,\,T],
        \\
        \noalign{\smallskip}
        \ \gamma
        \quad &\text{if}\quad t\in [\tau_1, \tau_2],
        \\
        \noalign{\smallskip}
        f_l(u^*(\zeta(t),t))-\dot{\zeta}(t)\, u^*(\zeta(t),t)
        \quad &\text{if}\quad t\in [0,\tau_1[\,,
        \end{cases}
\end{aligned}
\end{equation}
we write 
\begin{equation}
\label{eq:Fg1}
    \mathcal{F}_t(\zeta, u)=\int_t^T g(t)\dif t,
    \qquad\quad
    \mathcal{F}_t(\zeta, u^*)=\int_t^T g^*(t)\dif t,\qquad\forall~t\in [0,T].
\end{equation}
Note that, because of~\eqref{eq:RH}, in the definition of $g$ we may equivalently 
take $u(\zeta(t)-,t)$ instead of $u(\zeta(t)+,t)$, while in the definition of $g^*$ we take $u^*$ continuous at $(\zeta(t),t)$
since $\zeta$ is a classical genuine characteristic for $u^*$
when $t\in \,]0,\tau_1[\, \cup\,]\tau_2, T[$\,.
Observe that, since $u$ is an $AB$-entropy solutions to~\eqref{conslaw},
by the interface condition~\eqref{ABtraces} we have
\begin{equation}
\label{eq:gest1}
    f_r(u_r(t))\geq \gamma \qquad \forall~t\in [\tau_1, \tau_2].
\end{equation}
On the other hand, because of the convexity of $f_r$ there holds
\begin{equation}\label{eq:convexest}
f_r(v) - f_r^{\prime}(w)\,v \geq f_r(w) - f_r^{\prime}(w)\,w, \quad \forall \; v,w \in \mathbb R.
\end{equation}
Moreover, since $\zeta$ is an $AB$-gic for $u^*$, and because of~\eqref{interf-gen-char-cond-n4}, 
note that the restriction of $\zeta$
to $[0,\tau_1[\,$ is a classical genuine characteristic for $u^*$ as solution of $u_t+f_l(u)_x=0$, 
and the restriction of $\zeta$
to $\,]\tau_2, T]$ is a classical genuine characteristic for $u^*$ as solution of $u_t+f_r(u)_x=0$.
Hence, it follows that
\begin{equation}
    \label{eq:charspeed-1}
    \dot\zeta(t)=
    \begin{cases}
 f'_l(u^*(\zeta(t),t))\quad &\text{if}\ \quad t \in \,]0,\tau_1[\,,
 \\
 \noalign{\smallskip}
 \ 0 &\text{if}\ \quad t \in \,]\tau_1, \tau_2[\,,
 \\
 \noalign{\smallskip}
 f'_r(u^*(\zeta(t),t))\quad &\text{if}\ \quad t \in \,]\tau_2,T[\,.
    \end{cases}
\end{equation}
Thus, \eqref{eq:convexest}-\eqref{eq:charspeed-1} together imply
\begin{equation}
\label{eq:gest2}
\begin{aligned}
    f_l(u(\zeta(t)+,t))-\dot\zeta(t)\, u(\zeta(t)+,t)&\geq f_l(u^*(\zeta(t),t))-
    \dot\zeta(t)\,
    u^*(\zeta(t),t)
    \qquad\forall~t\in 
    \,]0, \tau_1[\,,
    \\
    \noalign{\smallskip}
    f_r(u(\zeta(t)+,t))-\dot\zeta(t)\, u(\zeta(t)+,t)&\geq f_r(u^*(\zeta(t),t))-
    \dot\zeta(t)\,
    u^*(\zeta(t),t)
    \qquad\forall~t\in 
    \,]\tau_2,T[\,.
\end{aligned}
\end{equation}
Therefore, from~\eqref{eq:gest1},
\eqref{eq:gest2} we deduce that
\begin{equation}
\label{eq:gest3}
    g(t)\geq g^*(t)\qquad\forall~t\in [0,T],
\end{equation}
which, because of~\eqref{eq:Fg1}, yields~\eqref{eq:Fest1}.

Concerning~\eqref{eq:Fest2},  if $\zeta \in \mc C (u^*,x) \cap \mc C(u,x)$, then the inequality~\eqref{eq:Fest1} is verified 
also when $u$ and $u^*$ switch their places, so that we have 
$\mathcal{F}_t(u,\zeta) \geq \mathcal{F}_t(u^*,\zeta)$ for all $t\in [0,T]$,
thus proving
\begin{equation}
\label{eq:Fest3}
    \zeta \in \mc C (u^*,x) \cap \mc C(u,x)
    \quad \Longrightarrow\quad 
    \mathcal{F}(u,\zeta) = \mathcal{F}(u^*,\zeta).
\end{equation}
Next, given $\zeta\in \mathcal{C}(u^*,x)$,
assume that 
\begin{equation}
\label{eq:Fest4}
    \mathcal{F}(\zeta, u) = \mathcal{F}(\zeta, u^*).
\end{equation}
Since, by the above analysis we have~\eqref{eq:gest1}, it follows from~
\eqref{eq:Fest4} that $g(t)=g^*(t)$
for a.e. $t\in [0,T]$. Because of~\eqref{eq:ggstardef}, \eqref{eq:charspeed-1},  this in particular implies that, for a.e. $t\in [0,T]$, there holds
\begin{equation}
    \label{eq:gest5}
    f_r(u_r(t))=\gamma\qquad\text{if}
    \quad t\in [\tau_1,\tau_2]\,,
\end{equation}
and
\begin{equation}
\label{eq:gest4}
    \begin{aligned}
    &f_l(u(\zeta(t)+,t))-f'_l(u^*(\zeta(t),t))\,
    u(\zeta(t)+,t)=
        \\
        &\qquad\quad=
        f_l(u^*(\zeta(t),t))-f'_l(u^*(\zeta(t),t))\,
    u^*(\zeta(t),t)
        \qquad\text{if}\quad 
        t\in 
    \,]0,\tau_1[\,,
    \\
    \noalign{\medskip}
        &f_r(u(\zeta(t)+,t))-f'_r(u^*(\zeta(t),t))\,
        u(\zeta(t)+,t)=
        \\
        &\qquad\quad =
        f_r(u^*(\zeta(t),t))-f'_r(u^*(\zeta(t),t))\,
        u^*(\zeta(t),t)
        \qquad\text{if}\quad 
        t\in 
    \,]\tau_2,T[\,.
    \end{aligned}
\end{equation}
Since $f_l, f_r$ are strictly convex functions, we deduce from~\eqref{eq:gest4} that
$u(\zeta(t)+,t)=u^*(\zeta(t),t)$
for a.e. $t\in \,]0,\tau_1[\, \cup\,]\tau_2, T[$\,. If we repeat the same analysis taking  $u(\zeta(t)-,t)$ instead of $u(\zeta(t)+,t)$ in the definition~\eqref{eq:ggstardef} of $g$, we find that also $u(\zeta(t)-,t)=u^*(\zeta(t),t)$
for a.e. $t\in \,]0,\tau_1[\, \cup\,]\tau_2, T[$\,. This shows that 
the restriction of
$\zeta$ to $[0,\tau_1[\,$ 
and to $\,]\tau_2, T]$ is a classical genuine characteristic for $u$ as well,
as 
solution of $u_t+f_l(u)_x=0$, 
and  of $u_t+f_r(u)_x=0$, respectively.
Hence, by Remark~\ref{rem:oncharact}
we deduce that $\zeta$ is an $AB$-igc
also for $u$, which means that $\zeta\in \mathcal{C}(u,x)$, completing the proof of
\begin{equation}
\label{eq:Fest5}
    \mathcal{F}(\zeta, u) = \mathcal{F}(\zeta, u^*)
    \quad \Longrightarrow\quad 
    \zeta \in \mc C (u^*,x) \cap \mc C(u,x),
\end{equation}
and thus concluding the proof of the Lemma.
\end{proof}

\section{Proof of Theorem \ref{initialdataid}}\label{sec:firstthm}
In this section we provide a proof of the initial data identification Theorem \ref{initialdataid}. To this end we first state a technical Lemma that 
we are going to use repeatedly in the proof of Theorem~\ref{initialdataid}.

\begin{lemma}\label{Curves}
Let $u$ 
be an $AB$-entropy solution to~\eqref{conslaw}, \eqref{discflux}, 
and let $\alpha,\beta:[\tau,T] \to \mathbb R$, $\tau<T$, be two Lipschitz continuous maps such that $\alpha (t)\leq \beta(t)$ for all $t\in [\tau, T]$. Then it holds
\begin{equation}
\label{eq:divthm-int-equal}
\begin{aligned}
    &\int_{\alpha(T)}^{\beta(T)} u(x,T)\dif x-\int_{\alpha(\tau)}^{\beta(\tau)} u(x,\tau)\dif x~= \mathcal{F}_\tau(\alpha-, u)-
    \mathcal{F}_\tau(\beta+, u)\,.
\end{aligned}
\end{equation}
\end{lemma}

\begin{proof}
Observe that, by property (1) of Definition~\ref{defiAB}, $u$ is a 
weak distributional solution 
to~\eqref{conslaw},~\eqref{discflux}.
Moreover, by Remark~\ref{rem:abentr-sol-prop1}, $u(t,\cdot)$
is a function of locally bounded variation 
on $\{x<0\}$, $\{x>0\}$, 
 and it admits left and right strong traces at $x = 0$, for all $t>0$.
Thus, we can recover 
   the equality~\eqref{eq:divthm-int-equal} 
   recalling definition~\eqref{eq:def-F}, applying the divergence theorem to the vector field $(f(x,u), u)$ on each domain 
   $\Delta\cap \{x<\rho\}$, $\Delta\cap\{x>\rho\}$, 
   with $\Delta\doteq\{(x,t)\; | \; \alpha(t)\leq x\leq \beta(t), \ t\in[t_0, T]  \}$,
   and then taking the limit as $\rho\to 0$.
\end{proof}

\smallskip

\begin{proof}[Proof of Theorem~\ref{initialdataid}]
Given $\omega^T \in \mathcal{A}^{[AB]}(T)$,
let $u^*$ be the 
 $AB$-entropy solution defined by~\eqref{eq:vertexinitialset}-\eqref{eq:sol-vertexinitialset}.

 \textbf{1.} We will show that if 
 $u_0 \in  \mc I_T^{[AB]}(\omega^T)$, then for every point $\overline x\in\R$ 
 there exists $\overline y \in \mathcal{C}_0(u^*,\overline x)$ such that
there hold~\eqref{condleq-2}. The proof of~\eqref{condgeq-2} is entirely similar.
By Proposition~\ref{Clemma}, 
choose $\zeta_{\ol x} \in 
\mathcal{C}(u^*,\ol x) \cap \mathcal{C}(u,\ol x)$,
with $u \doteq \sabp u_0(\cdot)$,
and set $\ol y \doteq \zeta_{\ol x}(0)$.
Then, consider any $y < \min \mathcal{C}_0(u^*,\ol x)$.  
By Lemma \ref{lemma:Csurjective},
and because of Proposition~\ref{lemma:Cnotempty}-(iv), there will be some $x< \bar x$,
and some $\zeta_x \in \mathcal{C}(u^*,x)$, such that
$y = \zeta_x(0)$.
Hence, applying Lemma \ref{Flemma}, we deduce that 
\begin{equation}\label{Fineq}
\mathcal{F}(u,\zeta_x) \geq \mathcal{F}(u^*, \zeta_x).
\end{equation} 
Moreover, since  $\zeta_{\bar x} \in 
\mathcal{C}(u^*,\ol x) \cap \mathcal{C}(u,\ol x)$, by the second part of Lemma \ref{Flemma} we have
\begin{equation}\label{Feq}
\mathcal{F}(u^*,\zeta_{\bar x}) = \mathcal{F}(u, \zeta_{\bar x}).
\end{equation}
On the other hand, applying Lemma \ref{Curves} to the solution $u^*$ with $\alpha = \zeta_x, \beta = \zeta_{\ol x}$, $\tau=0$, 
and recalling Remark~\ref{rem:def-F-on-gic}, one obtains
\begin{equation}
\label{eq:Fineq2}
\int_x^{\bar x} \omega^T(\xi) \dif \xi - \int_y^{\bar y} u_0^*(\xi) \dif \xi = \mathcal{F}(u^*,\zeta_x)-\mathcal{F}(u^*,\zeta_{\bar x}).
\end{equation}
With the same arguments, applying Lemma \ref{Curves} to the solution $u$, and relying on~\eqref{Fineq}, \eqref{Feq},
we find
\begin{equation}
\label{eq:Fineq3}
\int_x^{\ol x} \omega^T(\xi) \dif \xi-\int_y^{\ol y} u_0(\xi) \dif \xi = \mathcal{F}(u, \zeta_x) -\mathcal{F}(u, \zeta_{\ol x})\geq  \mathcal{F}(u^*,\zeta_x)-\mathcal{F}(u^*,\zeta_{\ol x}).
\end{equation}
Combining~\eqref{eq:Fineq2}, \eqref{eq:Fineq3} we deduce
\begin{equation*}
\label{eq:Fineq1}
 - \int_y^{\bar y} u_0(\xi) \dif \xi \geq -\int_y^{\ol y} u_0^*(\xi) \dif \xi,
\end{equation*}
which yields~\eqref{condleq-2}.
\\

\begin{figure}
\centering

\tikzset{every picture/.style={line width=0.75pt}} 

\begin{tikzpicture}[x=0.75pt,y=0.75pt,yscale=-0.7,xscale=0.7]

\draw    (10,230) -- (281.09,230) ;
\draw [shift={(284.09,230)}, rotate = 180] [fill={rgb, 255:red, 0; green, 0; blue, 0 }  ][line width=0.08]  [draw opacity=0] (5.36,-2.57) -- (0,0) -- (5.36,2.57) -- cycle    ;
\draw    (340,230) -- (588.14,230) ;
\draw [shift={(591.14,230)}, rotate = 180] [fill={rgb, 255:red, 0; green, 0; blue, 0 }  ][line width=0.08]  [draw opacity=0] (5.36,-2.57) -- (0,0) -- (5.36,2.57) -- cycle    ;
\draw    (454.09,230) -- (454.09,33) ;
\draw [shift={(454.09,30)}, rotate = 90] [fill={rgb, 255:red, 0; green, 0; blue, 0 }  ][line width=0.08]  [draw opacity=0] (5.36,-2.57) -- (0,0) -- (5.36,2.57) -- cycle    ;
\draw    (340,60) -- (594.09,60.77) ;
\draw    (10,60.77) -- (284.09,60.77) ;
\draw    (147.05,230) -- (147.05,33) ;
\draw [shift={(147.05,30)}, rotate = 90] [fill={rgb, 255:red, 0; green, 0; blue, 0 }  ][line width=0.08]  [draw opacity=0] (5.36,-2.57) -- (0,0) -- (5.36,2.57) -- cycle    ;
\draw [line width=1.5]    (55.68,60.77) -- (147.05,106.92) -- (223.18,230) ;
\draw [line width=1.5]    (223.18,60.77) -- (147.05,153.08) -- (147.05,183.85) -- (70.91,230) ;
\draw [line width=1.5]    (534.09,60) -- (544.09,230) ;
\draw [line width=1.5]    (376.36,60.77) -- (454.09,110) -- (454.09,170) -- (494.09,230) ;
\draw  [dash pattern={on 0.84pt off 2.51pt}]  (148.53,134.4) -- (280.53,133.73) ;

\draw (368.09,82.4) node [anchor=north west][inner sep=0.75pt]  [font=\footnotesize]  {$\xi _{1}$};
\draw (545.09,142.4) node [anchor=north west][inner sep=0.75pt]  [font=\footnotesize]  {$\xi _{2}$};
\draw (71,84.4) node [anchor=north west][inner sep=0.75pt]  [font=\footnotesize]  {$\xi _{1}$};
\draw (215.33,85.73) node [anchor=north west][inner sep=0.75pt]  [font=\footnotesize]  {$\xi _{2}$};
\draw (284,122.07) node [anchor=north west][inner sep=0.75pt]  [font=\footnotesize]  {$\tau $};
\draw (368,40.4) node [anchor=north west][inner sep=0.75pt]  [font=\footnotesize]  {$x_{1}$};
\draw (528.67,41.07) node [anchor=north west][inner sep=0.75pt]  [font=\footnotesize]  {$x_{2}$};
\draw (49.33,40.4) node [anchor=north west][inner sep=0.75pt]  [font=\footnotesize]  {$x_{1}$};
\draw (221.33,39.07) node [anchor=north west][inner sep=0.75pt]  [font=\footnotesize]  {$x_{2}$};
\draw (586,37.07) node [anchor=north west][inner sep=0.75pt]  [font=\footnotesize]  {$\omega $};
\draw (272,37.73) node [anchor=north west][inner sep=0.75pt]  [font=\footnotesize]  {$\omega $};

\end{tikzpicture}

\caption{{\scshape Case 1}: $\max \mathcal{C}_0(u,x_1) \geq \min \mathcal{C}_0(u^*,x_2)$ (right);
{\scshape Case 2}: $\max \mathcal{C}_0(u,x_1) < \min \mathcal{C}_0(u^*,x_2)$ (left) \\
}
\label{figproofidentification}
\end{figure}
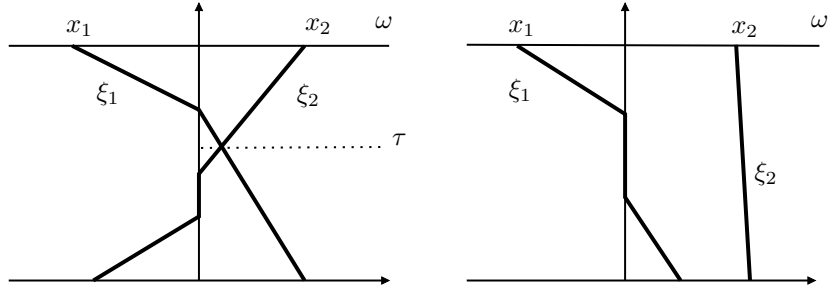

 \textbf{2.} 
Now we prove that if $u_0 \in \mathbf L^{\infty}(\mathbb R)$ satisfies \eqref{condleq-2}, \eqref{condgeq-2}, then $\sabpT u_0 = \omega^T$. 
Namely, we are going to prove that
under the conditions \eqref{condleq-2}, \eqref{condgeq-2},
there hold
\begin{equation}\label{Lebesgueeq}
\int_{x_1}^{x_2} \big(\omega^T(x) - \sabpT  u_0(x)\big) \dif x = 0, 
\quad\ \forall~x_1< x_2, 
\end{equation}
which clearly implies that $\sabpT u_0 = \omega^T$. 

Towards a proof of~\eqref{Lebesgueeq}
we will first show that 
\begin{equation}\label{Lebesguegeq}
\int_{x_1}^{x_2} \big(\omega^T(x) - \sabpT u_0(x)\big) \dif x \geq 0,
\quad\ \forall~x_1< x_2, 
\end{equation}
distinguishing two cases.
\\

\noindent {\scshape Case 1.} $\max \mathcal{C}_0(u,x_1) \geq \min \mathcal{C}_0(u^*,x_2)$ (see Figure \ref{figproofidentification}, right). Then we can choose $\zeta_1 \in \mathcal{C}(u,x_1)$ and $\zeta_2 \in \mathcal{C}(u^*,x_2)$ such that $\zeta_1(0)\geq \zeta_2(0)$. By continuity there will be a point $\tau \in [0,T[\,$ such that $\zeta_1(\tau) = \zeta_2(\tau)$,
$\zeta_1(t) < \zeta_2(t)$
for all $t\in\,]\tau, T]$. Applying Lemma \ref{Curves} to the solution $u^*$, with the curves 
$\alpha=\zeta_1$,
$\beta=\zeta_2$,
and using the first part of Lemma \ref{Flemma} 
for $\zeta_1$, we obtain
\begin{equation}
\int_{x_1}^{x_2} \omega^T(x) \dif x = \mathcal{F}_{\tau}(u^*,\zeta_1)- \mathcal{F}_{\tau}(u^*,\zeta_2) \geq \mathcal{F}_{\tau}(u,\zeta_1)- \mathcal{F}_{\tau}(u^*,\zeta_2).
\end{equation}
Next, applying again Lemma \ref{Curves} to the solution $u$, with the curves 
$\alpha=\zeta_1$,
$\beta=\zeta_2$, 
and then Lemma \ref{Flemma}
for $\zeta_2$, we obtain
\begin{equation}
\int_{x_1}^{x_2} \sabpT u_0(x) \dif x = \mathcal{F}_{\tau}(u,\zeta_1)- \mathcal{F}_{\tau}(u,\zeta_2) \leq \mathcal{F}_{\tau}(u,\zeta_1)- \mathcal{F}_{\tau}(u^*,\zeta_2).
\end{equation}
Taking the difference of the above two inequalities, we derive \eqref{Lebesguegeq}. Note that in this case we are not using
the conditions \eqref{condleq-2}, \eqref{condgeq-2} to establish~\eqref{Lebesguegeq}.
\\

\noindent {\scshape Case 2.} $\max \mathcal{C}_0(u,x_1) < \min \mathcal{C}_0(u^*,x_2)$ (see Figure \ref{figproofidentification}, left). Choose any $\zeta_1 \in \mathcal{C}_0(u,x_1)$, and set $y \stackrel{\cdot}{=} \zeta_1(0)$. Since $y <\min \mathcal{C}_0(u^*,x_2)$, invoking condition~\eqref{condleq-2}
we find that there exists 
$\zeta_2 \in\mathcal{C}(u^*,x_2)$ such that, setting $y_2=\zeta_2(0)$,
there holds
\begin{equation}\label{condleqy_2}
\int_y^{y_2} u_0(x) \dif x \leq \int_y^{y_2}u^*_0(x) \dif x\,.
\end{equation}
Note that, since $\max \mathcal{C}_0(u,x_1) < \min \mathcal{C}_0(u^*,x_2)$, we have $\zeta_1(t)<\zeta_2(t)$ for all $t\in [0,T]$.
Hence, applying Lemma \ref{Curves} to $u^*$, with the curves 
$\alpha=\zeta_1$,
$\beta=\zeta_2$, and  using the first part of Lemma \ref{Flemma} 
for $\zeta_1$, we obtain
\begin{equation}
    \begin{aligned}
        \int_{x_1}^{x_2} \omega^T(x) \dif x &= \mathcal{F}(u^*,\zeta_1)- \mathcal{F}(u^*,\zeta_2) + \int_y^{y_2}u^*_0(x) \dif x \\ 
        &\geq \mathcal{F}(u,\zeta_1)- \mathcal{F}(u^*,\zeta_2)+ \int_y^{y_2}u^*_0(x) \dif x\,.
    \end{aligned}
\end{equation}
Next, applying Lemma \ref{Curves} to $u$, with the curves 
$\alpha=\zeta_1$,
$\beta=\zeta_2$, and  using the first part of Lemma \ref{Flemma} 
for $\zeta_2$, we obtain
\begin{equation}
\begin{aligned}
    \int_{x_1}^{x_2} \sabpT u_0(x) \dif x &= \mathcal{F}(u,\zeta_1)- \mathcal{F}_{\tau}(u,\zeta_2)+\int_y^{y_2}u_0(x) \dif x \\ 
    &\leq \mathcal{F}(u,\zeta_1)- \mathcal{F}(u^*,\zeta_2)+\int_y^{y_2}u_0(x) \dif x\,.
\end{aligned}
\end{equation}
Taking the difference of the  above two inequalities, and using \eqref{condleqy_2}, we derive
\begin{equation}
\int_{x_1}^{x_2} \omega^T(x) \dif x-\int_{x_1}^{x_2} \sabpt u_0(x) \dif x \geq \int_y^{y_2}u^*_0(x) \dif x-\int_y^{y_2}u_0(x) \dif x \geq 0
\end{equation}
which proves \eqref{Lebesguegeq} also in {\scshape Case 2.}
\smallskip

The proof of the opposite inequality
of~\eqref{Lebesguegeq} is entirely symmetric and is accordingly omitted. 
Thus the proof of~\eqref{Lebesgueeq}
is completed, and this concludes the proof of the Theorem.
\end{proof}

\begin{remark}
\label{rem:continuity-hyp}
    By the proof of Theorem~\ref{initialdataid} it follows that it is sufficient to assume:
$$\begin{aligned}
    &\text{\it for every point  $\overline x\in\R$ 
of continuity  of $\omega^T$, 
there exists $\overline y \in \mathcal{C}_0(u^*,\overline x)$}\\
&\text{\it such that
there hold~\eqref{condleq-2}, \eqref{condgeq-2},}
\end{aligned}
$$
to conclude that $u_0 \in  \mc I_T^{[AB]}(\omega^T)$.
In fact, in order to show that $\omega^T=\sabpT u_0$,
it is sufficient to prove that~\eqref{Lebesgueeq}
is verified whenever $x_1, x_2$ are points of continuity for $\omega^T$, 
since they are  dense in $\R$.
\end{remark}

\section{Proof of Theorem \ref{geometrical properties}}\label{sec:secthm}

\begin{proof}
Given $\omega^T \in \mathcal{A}^{[AB]}(T)$,
let $u^*$ be the 
 $AB$-entropy solution defined by~\eqref{eq:vertexinitialset}-\eqref{eq:sol-vertexinitialset}.
We prove the Theorem point by point, in order.

\vspace{0.5cm}

 \textbf{1.} Proof of (i). 
First assume that $\left|\mathcal{C}_0(u^*,x)\right| = 1$ for every $x \in \mathbb R$. We will show that any initial data $u_0\in \mathcal{I}_T^{[AB]}(\omega^T)$ satisfies
\begin{equation}\label{eq:initidatumsingleton}
\int_{y_1}^{y_2} u_0(x) \dif x = \int_{y_1}^{y_2} u_0^* (x) \dif x, \quad \forall~y_1 < y_2, 
\end{equation}
and this uniquely identifies $u_0$ as an element of $L^{\infty}(\mathbb R)$, 
thus proving that $\mathcal{I}_T^{[AB]}(\omega^T)=\{u_0^*\}$.
Given any two points $y_1 < y_2$,  by Lemma \ref{lemma:Csurjective}
and because of Proposition~\ref{lemma:Cnotempty}-(iv), there exist $x_1<x_2$, and  
$\zeta_i \in \mathcal{C}(u^*,x_i)$,
$i=1,2$, such that $\zeta_i(0)=y_i$,
$i=1,2$.
Then, applying~\eqref{condleq-2} 
of Theorem~\ref{initialdataid} we find
\begin{equation}
\int_{y_1}^{y_2} u_0(x) \dif x \leq \int_{y_1}^{y_2} u_0^* (x) \dif x\,.
\end{equation}
Next, if we exchange the role of $y_1$ and $y_2$, applying this time~\eqref{condgeq-2}
of Theorem~\ref{initialdataid} we find the opposite inequality
\begin{equation}
\int_{y_1}^{y_2} u_0(x) \dif x \geq \int_{y_1}^{y_2} u_0^* (x) \dif x\,.
\end{equation}
Combining together the above two inequalities we obtain~\eqref{eq:initidatumsingleton}.

Conversely, assume that $\mathcal{I}_T^{[AB]}(\omega^T) = \{u_0^*\}$, and  by contradiction suppose that there is some $\widetilde{x} \in \mathbb R$ such that $\left|\mathcal{C}_0(u^*,\widetilde{x}\,)\right| \neq 1$. 
Using the characterization of Theorem \ref{initialdataid}
we will then show that there exist infinitely many initial data
$u_0 \neq u^*_0$ such that $\sabpT u_0 = \omega^T$. 
To this end, set
\begin{equation*}
    \mathrm{conv} \, \mc C_0(u^*, \widetilde x\,)\doteq 
    [\min \mc C_0(u^*, \widetilde x\,),\, \max \mc C_0(u^*, \widetilde x\,)]\,,
\end{equation*}
and let $\mathbf  L^{\infty}(\mathrm{conv} \, \mc C_0(u^*, \widetilde x\,))$ denote the space of 
$\mathbf  L^{\infty}(\R)$ function with essential support in $\mathrm{conv} \, \mc C_0(u^*, \widetilde x\,))$. Note that $\mathrm{conv} \, \mc C_0(u^*, \widetilde x\,)$ is a non trivial interval because
$\left|\mathcal{C}_0(u^*,\widetilde{x}\,)\right| \neq 1$, and hence $\mathbf  L^{\infty}(\mathrm{conv} \, \mc C_0(u^*, \widetilde x\,))$ is an infinite dimensional space. Next,
consider the infinite dimensional cone
$\mathrm{V}_0 \subset \mathbf  L^{\infty}(\mathrm{conv} \, \mc C_0(u^*, \widetilde x\,))$
consisting of all $v_0\in \mathbf  L^{\infty}(\mathrm{conv} \, \mc C_0(u^*, \widetilde x\,))$
that satisfy
\begin{equation}\label{eq:v_0props}
\begin{aligned}
     \int_{y}^{\max \mc C_0(u^*, \widetilde x\,)} v_0(x) \dif x &\leq 0, \qquad \forall \; y \in \mathrm{conv}\;  \mc C_0(u^*, \widetilde x\,),\\
\int_{\min \mc C_0(u^*, \widetilde x\,)}^y v_0(x) \dif x &\geq 0, \qquad \forall \; y \in \mathrm{conv}\;  \mc C_0(u^*, \widetilde x\,).
     \end{aligned}
\end{equation}
Note that~\eqref{eq:v_0props} in particular imply
\begin{equation}
\label{eq:V-0int-eq}
    \int_{\min \mc C_0(u^*, \widetilde x\,)}^{\max \mc C_0(u^*, \widetilde x\,)} v_0(x) \dif x = 0\,.
\end{equation}
We will show that 
\begin{equation}
\label{eq:Vcone-1}
    V \doteq u_0^* + V_0 \subset \mathcal{I}_T^{[AB]}(\omega^T).
\end{equation}
Relying on Theorem~\ref{initialdataid}
this is equivalent to prove that, for any $v_0\in V_0$, and for every $\ol x\in\R$, there exists $\ol y \in \mc C_0(u^*, \ol x)$ such that~\eqref{condleq-2}, \eqref{condgeq-2} hold for $u_0\doteq u_0^*+v_0$. We will verify only~\eqref{condleq-2}, the proof of the other inequality being entirely symmetric. 
 

Then, consider first any $\ol x \leq \widetilde{x}$,
and choose $\ol y = \min \mc C_0(u^*, \ol x)$. 
Observe that, for every $y < \min \mc C_0(u^*, \ol x)$, we have $u_0 = u_0^*$ on the interval $[y,\, \ol y]$ since $\ol x \leq \widetilde x$, together with
 Proposition~\ref{lemma:Cnotempty}-(iv), implies
$$\ol y 
\leq \min \mc C_0(u^*, \widetilde x),$$
and hence $v_0 =0$ on  $[y,\, \ol y]$,
because the essential support of $v_0$ is contained in $\mathrm{conv} \, \mc C_0(u^*, \widetilde x\,))$. 
This implies that,
 for every $y < \min \mc C_0(u^*, \ol x)$,
 we have
\begin{equation}
\label{eq:V-0int-eq-2}
    \int_y^{\ol y} u_0(x) \dif x = \int_y^{\ol y} u_0^*(x) \dif x\,,
\end{equation}
    which proves~\eqref{condleq-2} as an equality. 

Next, consider any $\ol x > \widetilde{x}$, and choose $\ol y = \max \mc C_0(u^*, \ol x)$. Then, for every $y < \min \mc C_0(u^*, \ol x)$, 
one of the following three cases occurs:

\begin{enumerate}
[leftmargin=48pt]
    \item[{\scshape Case 1.}] If $y \in \,]\max \mc C_0(u^*, \tilde x), \min \mc C_0(u^*, \bar x)[\,$, then~\eqref{condleq-2} holds again as an equality,
because $u_0$  coincides with $u_0^*$ in the interval $[y,\, \ol y]$ as in the case $\ol x \leq \widetilde{x}$ considered above,
and thus~\eqref{eq:V-0int-eq-2} is verified.
\item[{\scshape Case 2.}] If $y \in \mathrm{conv} \; \mc C_0(u^*, \tilde x)$, then by \eqref{eq:v_0props} we have
$$
\begin{aligned}
    \int_y^{\ol y} u_0(x) \dif x 
    &= \int_y^{\max \mc C_0(u^*, \widetilde x)} u_0(x) \dif x + \int _{\max \mc C_0(u^*, \widetilde x)}^{\ol y} u^*_0(x) \dif x
    \\
    &\leq \int_y^{\ol y} u_0^*(x) \dif x\,,
\end{aligned}
$$
which proves~\eqref{condleq-2}.
\item[{\scshape Case 3.}] If $y < \min \mc C_0(u^*, \tilde x)$, 
we obtain~\eqref{condleq-2} relying on~\eqref{eq:V-0int-eq}, since
$$
\begin{aligned}
    \hspace{30pt}\int_y^{\ol y}u_0(x) \dif x 
    &= \!\int_y^{\min \mc C_0(u^*, \widetilde x)} \!u_0^*(x) \dif x +\! \int_{\min \mc C_0(u^*, \widetilde x)}^{\max  \mc C_0(u^*, \widetilde x)} \!u_0(x) \dif x +\! \int_{\max  \mc C_0(u^*, \widetilde x)}^{\bar y} \!u_0^*(x) \dif x 
    \\
    &= \int_y^{\bar y}u^*_0(x) \dif x\,.
\end{aligned}
$$
\end{enumerate}
Thus, 
for all $u_0=u_0^*+v_0$, $v_0\in V_0$,
and for every $\ol x\in\R$, there exists $\ol y \in \mc C_0(u^*, \ol x)$ such that~\eqref{condleq-2},
\eqref{condgeq-2}
hold.
Hence~\eqref{eq:Vcone-1} is verified,
which contradicts the assumption $\mathcal{I}_T^{[AB]}(\omega^T) = \{u_0^*\}$, and thus
completes the proof of the first part of property (i).

Finally, observe that if $x$ is a point of discontinuity for $\omega^T$, then one can consider 
the $AB$-gics  $\vartheta_{x,-}, \vartheta_{x,+}: [0,T]\to\R$
that are 
the minimal and maximal $AB$-gics for $u^*$
reaching at time~$T$ the point $x$
 (e.g. see point {\bf 2} of
the proof of Proposition~\ref{Clemma}).
Since $\vartheta_{x,-}(0)\neq \vartheta_{x,+}(0)$
if $x\neq 0$,
and because $\{\vartheta_{x,-}(0),\,\vartheta_{x,+}(0)\}\subset\mc C_0(u^*, \tilde x)$,
this implies $\left|\mathcal{C}_0(u^*,x\,)\right| \neq 1$, thus proving by contradiction that if $\mathcal{I}_T^{[AB]}(\omega^T)$ is a singleton, then $\omega^T$ must be continuous at any point $x\neq 0$.
This concludes the proof of property (i).

\vspace{0.3cm}
 \textbf{2.} Proof of (ii). 
To prove that the set $\mathcal{I}_T^{[AB]}(\omega^T)-u_0^*$ is a linear cone,
we will show that, for every $u_0 \in \mathcal{I}_T^{[AB]}(\omega^T)$ and $\lambda \geq 0$, it holds $u_0^* + \lambda(u_0 -u_0^*) \in \mathcal{I}_T^{[AB]}(\omega^T)$. To see this, applying Theorem \ref{initialdataid}
it's sufficient to prove that, given any $\ol x\in\R$, 
there exists $\ol y \in \mc C_0(u^*, \ol x)$
such that
~\eqref{condleq-2}, \eqref{condgeq-2}  hold
with $u_0^* + \lambda(u_0 -u_0^*)$ in place of $u_0$. 
Since $u_0 \in \mathcal{I}_T^{[AB]}(\omega^T)$,
by Theorem~\ref{initialdataid} we know that
there is some $\ol y\in \mc C_0(u^*, \ol x)$ such that \eqref{condleq-2} holds. Then, for all $y < \min \mathcal{C}_0(u^*,\overline x)$,
one finds
\begin{equation*}
\int_y^{\ol y} \big(u_0^*(x)+\lambda(u_0(x)-u_0^*(x))\big) \dif x \leq \int_y^{\ol y} \big(u_0^*(x)+\lambda(u_0^*(x)-u_0^*(x))\big) \dif x = \int_y^{\ol y} u_0^*(x) \dif x\,.
\end{equation*}
This proves that~\eqref{condleq-2} is verified
with $u_0^* + \lambda(u_0 -u_0^*)$ in place of $u_0$.
The proof that also \eqref{condgeq-2} holds, 
is entirely symmetric. 

Next, we prove that $u_0^*$ is an extremal point
of $\mathcal{I}_T^{[AB]}(\omega^T)$.
Assume by contradiction that there exist
$u_{0,i} \in \mathcal{I}_T^{[AB]}(\omega^T)$, 
$u_{0,i}\neq u_0^*$, $i = 1,2$, and 
$\lambda \in \,]0,1[\,$, such that 
\begin{equation}
\label{eq:uo*convexcomb}
    u^*_0 = \lambda u_{0,1} + (1-\lambda) u_{0,2}\,.
\end{equation}
Take any $\ol x \in \R$ for which  $\mathcal{C}_0(u^*, \ol x)$ is a singleton 
(one can choose $\ol x$ as a point of continuity 
for~$\omega^T$ belonging to the set $]-\infty, \ms L[\, \cup\,]\ms R, +\infty[$, with $\ms L, \ms R$ as in~\eqref{eq:LR-def}), and call $\ol y$ the unique element of $\mathcal{C}_0(u^*, \ol x)$.  
Because of~\eqref{eq:uo*convexcomb} it holds
\begin{equation}\label{convexinteq}
\int_{\bar y}^y u_0^*(x) \dif x=
\lambda \int_{\bar y}^y u_{0,1}(x) \dif x+ (1-\lambda) \int_{\bar y}^y u_{0,2}(x) \dif x, \qquad \forall \; y \in \mathbb R\,.
\end{equation}
Then, since $u_{0,1}, u_{0,2}$ are different from $u_0^*$, there must be some $y \in \mathbb R$ such that one of the following three cases occurs:
\begin{enumerate}
[leftmargin=48pt]
\item[{\scshape Case 1.}]
    \begin{equation}
    \label{convexinteq-2}
\int_{\ol y}^y u_{0,1}(x) \dif x \neq \int_{\ol y}^y u_0^*(x) \dif x \qquad \text{and} \qquad  \int_{\ol y}^y u_{0,2}(x) \dif x = \int_{\ol y}^y u_0^*(x) \dif x.
\end{equation}
\item[{\scshape Case 2.}]
    \begin{equation}
\int_{\ol y}^y u_{0,1}(x) \dif x = \int_{\ol y}^y u_0^*(x) \dif x \qquad \text{and} \qquad  \int_{\ol y}^y u_{0,2}(x) \dif x \neq \int_{\ol y}^y u_0^*(x) \dif x.
\end{equation}
    \item[{\scshape Case 3.}]
    \begin{equation}
\int_{\ol y}^y u_{0,1}(x) \dif x \neq \int_{\ol y}^y u_0^*(x) \dif x \qquad \text{and} \qquad  \int_{\ol y}^y u_{0,2}(x) \dif x \neq \int_{\ol y}^y u_0^*(x) \dif x.
\end{equation}
    \end{enumerate}
Assume that {\scshape Case 1} holds, with $y \neq \ol y$. Then, applying conditions~\eqref{condleq-2}, \eqref{condgeq-2} of Theorem~\ref{initialdataid} to~$u_0^1$, we find that
\begin{equation}
\label{convexinteq-3}
\int_{\ol y}^y u_{0,1}(x) \dif x > \int_{\ol y}^y u_0^* \dif x.
\end{equation}
But~\eqref{convexinteq-3}, together with the equality in~\eqref{convexinteq-2}, is in contradiction with \eqref{convexinteq}. 
The analysis of the other two cases is entirely similar, thus it is omitted.
This proves that $u_0^*$ is an extremal point of $\mathcal{I}_T^{[AB]}(\omega^T)$ (and of course it is unique since $u_0^*$ is the vertex of the affine cone $\mathcal{I}_T^{[AB]}(\omega^T)$), and thus concludes the proof of property (ii).

\vspace{0.3cm}
 \textbf{3.} Proof of (iii). 
 We first show that, if condition~\eqref{eq:refinedconditionconvex-2}
 is verified than the set $\mathcal{I}_T^{[AB]}(\omega^T)$
is convex.
Given $u_{0,1}, u_{0,2} \in \mathcal{I}_T^{[AB]}(\omega^T)$, and $\lambda \in \,]0,1[\,$, let $ \ol x \in \,]\ms L,\, \ms R[\,$ be a point of continuity for~$\omega^T$. 
By Theorem~\ref{initialdataid}, and because 
$\mathcal{C}_0(u^*,\ol x)$ is a singleton $\{\ol y\}$, we
know that  
there hold
\begin{equation}\label{ybar12}
\int_y^{\ol y} u_{0,1}(x) \dif x \leq \int_y^{\ol y}u_0^*(x) \dif x , \qquad \int_y^{\ol y} u_{0,2}(x) \dif x \leq \int_y^{\ol y}u_0^*(x) \dif x, \quad \forall \; y < \ol y.
\end{equation}
Then, 
using \eqref{ybar12}, we derive
\begin{equation}
\label{eq:convexity-ineq}
\int_y^{\ol y} \big(\lambda u_{0,1}(x) + (1-\lambda)u_{0,2}(x)\big)\dif x \leq \int_y^{\ol y} u^*(x) \dif x, \qquad \forall \; y < \ol y, 
\quad\forall~\lambda \in \,]0,1[\,,
\end{equation}
so that $\lambda u_{0,1}+(1-\lambda) u_{0,2}$ satisfies \eqref{condleq-2}
for all $\lambda \in \,]0,1[\,$,
and $\ol y\in \mathcal{C}_0(u^*,\ol x)$
with $ \ol x \in \,]\ms L,\, \ms R[\,$
 of continuity for~$\omega^T$. The proof that also \eqref{condgeq-2} holds for the same $\ol y$
 is entirely similar and is accordingly omitted. Next, consider a point $ \ol x \in \,]-\infty,\,\ms L[\ \cup\ ]\ms R, +\infty[\,$
 of continuity for~$\omega^T$. 
Notice that by definition~\eqref{eq:LR-def}
the classical backward characteristics starting
from $(\ol x, T)$ never cross the interface $x=0$
at positive times. Therefore the unique $AB$-gic
reaching the point $\ol x$ at time $t=T$ is a classical genuine characteristic starting say at
$\ol y$ at time $t=0$. Hence $\mathcal{C}_0(u^*,\ol x)=\{\ol y\}$, and we can proceed as above to show 
that $\lambda u_{0,1}+(1-\lambda) u_{0,2}$ satisfies~\eqref{condleq-2}, \eqref{condgeq-2} for all $\lambda \in \,]0,1[\,$,
also when $\ol y\in \mathcal{C}_0(u^*,\ol x)$
with $ \ol x \in \,]-\infty,\,\ms L[\ \cup\ ]\ms R, +\infty[\,$
 of continuity for~$\omega^T$.
Then, by Remark~\ref{rem:continuity-hyp} we can conclude that $\lambda u_{0,1}+(1-\lambda) u_{0,2}\in \mathcal{I}_T^{[AB]}(\omega^T)$,
for all $\lambda \in \,]0,1[\,$.


Now assume that condition~\eqref{refinedconditionconvex} is verified.
By the above analysis it is clear that in order to prove 
the convexity of $\mathcal{I}_T^{[AB]}(\omega^T)$
 it is sufficient to show that, 
for any $ \ol x \in \,]\ms L,\, \ms R[\,$  of continuity for~$\omega^T$ there exists 
$\ol y \in \mathcal{C}_0(u^*, \ol x)$ such that
there holds
\begin{equation}
\label{eq:convexity-ineq-2}
    \int_y^{\ol y} \big(\lambda u_{0,1}(x) + (1-\lambda)u_{0,2}(x)\big)\dif x \leq \int_y^{\ol y} u^*(x) \dif x, \qquad \forall \; y < \min \mathcal{C}_0(u^*,\ol x)\,, 
\quad\forall~\lambda \in \,]0,1[\,.
\end{equation}
The problem in this case is the following.
Since  $\mathcal{C}_0(u^*, \ol x)$ may not be a singleton, by Theorem~\ref{initialdataid}
we know that there will be
in general 
$\ol y_i\in \mathcal{C}_0(u^*, \ol x)$, $i=1,2$, $\ol y_1\neq \ol y_2$,
such that there hold
\begin{equation}
\label{ybar12-2}
\int_y^{\ol y_1} u_{0,1}(x) \dif x \leq \int_y^{\ol y_1}u_0^*(x) \dif x , \qquad \int_y^{\ol y_2} u_{0,2}(x) \dif x \leq \int_y^{\ol y_2}u_0^*(x) \dif x, \quad \forall \; y < 
\min \mathcal{C}_0(u^*,\ol x).
\end{equation}
Here the choice of $\ol y_i$
depends on the initial datum $u_{0,i}\in \mathcal{C}_0(u^*, \ol x)$, $i=1,2$.
Hence, 
we cannot rely on~\eqref{ybar12-2} to  derive immediately the existence of $\ol y \in \mathcal{C}_0(u^*, \ol x)$
such that~\eqref{eq:convexity-ineq-2} holds.
However, thanks to the assumption~\eqref{refinedconditionconvex} we can show that, for every point $ \ol x \in \,]\ms L,\, \ms R[\,$  of continuity for~$\omega^T$, there exists 
$\ol y \in \mathcal{C}_0(u^*, \ol x)$, independent on the
initial datum $u_0$ taken in consideration, such that~\eqref{condleq-2}, \eqref{condgeq-2} are verified.
In fact, given any point 
$ \ol x \in \,]\ms L,\, \ms R[\,$  of continuity for~$\omega^T$,
let 
$\{\ol x_n\}_n$ be a sequence of points in $\mathcal{X}(\omega^T)$
such that $\ol x_n\to \ol x$.  Letting $\{\ol y_n\}=\mathcal{C}_0(u^*, \ol x_n)$,
we may assume that, up to a subsequence, $\{\ol y_n\}_n$
converges to some point $\ol y\in\R$.
Since $x\mapsto \mathcal{C}_0(u^*, x)$ has closed graph
by Proposition~\ref{lemma:Cnotempty}
it follows that $\ol y\in \mathcal{C}_0(u^*, \ol x)$.
Hence, applying Theorem~\ref{initialdataid}
we find that, for any $y < 
\min \mathcal{C}_0(u^*,\ol x)$, and for $n$ sufficiently large, there hold
\begin{equation}
\label{ybar12-3}
\int_y^{\ol y_n} u_{0,1}(x) \dif x \leq \int_y^{\ol y_n}u_0^*(x) \dif x , \qquad \int_y^{\ol y_n} u_{0,2}(x) \dif x \leq \int_y^{\ol y_n}u_0^*(x) \dif x\,.
\end{equation}
taking the limit as $n\to\infty$ in~\eqref{ybar12-3} we derive
\begin{equation*}
\int_y^{\ol y} u_{0,1}(x) \dif x \leq \int_y^{\ol y}u_0^*(x) \dif x , \qquad \int_y^{\ol y} u_{0,2}(x) \dif x \leq \int_y^{\ol y}u_0^*(x) \dif x, \quad \forall \; y < \min \mathcal{C}_0(u^*,\ol x)\,,
\end{equation*}
which yields~\eqref{eq:convexity-ineq-2}.
This completes the proof of property (iii), 
and thus concludes the proof of the theorem.

\end{proof}


\subsection{A nonconvex set of initial data}\label{exmp:nonconvex}
We provide here an example of attainable profile 
$\omega^T \in\mc A^{[AB]}(T)$ for which the set $\mathcal{I}_T^{[AB]}(\omega^T)$ is not convex. Let $f\doteq f_l=f_r = u^2/2$, and set
\begin{equation}
\label{eq:AB-def}
    A 
    \doteq 4L_0,\qquad\ \  
    B \doteq -4L_0,
\end{equation}
for some constant $L_0<0$.
By definition~\eqref{eq:bar-AB-def}
we have
\begin{equation}
\label{eq:AB-bar-def}
    \ol A= -4L_0,\qquad 
    \ol B = 4L_0.
    \quad
\end{equation}
Then, consider the profile
\begin{equation}
\label{eq:profile-nonconvex-initialset}
\omega_3(x) = \begin{cases}
\ol A \ & x \leq L_0, \\
A \ & x \in \,]L_0,0[\,, \\
p \ & x >0, \\
\end{cases}
\end{equation}
with 
\begin{equation}
\label{eq:v-def}
    p < 12 L_0\,.
\end{equation}
0bserve that, by definition~\eqref{eq:LR-def}
we have $\ms L= \ms L[\omega_3, f]=L_0$,
and $\ms R= \ms R[\omega_3, f]=0$
since $f'(p)=p< 0$.
Moreover, recall that the quantity
$\bs v\doteq \bs v[L_0,A, f]$
defined in~\cite[\S~3.2]{talamini_ancona_attset}
satisfies 
\begin{equation}
   \label{eq:ex1-constraint-3}
       \ol A>\bs v>A.
   \end{equation}
Then, one can readly verify that $\omega_3$ fulfills the conditions (i)-(ii) of~\cite[Theorem 4.7]{talamini_ancona_attset}.
To simplify the analysis we shall consider 
a time horizon $T=1$.
With this choice, by a direct computation 
one finds that $\bs v=A-2\sqrt{A\,L_0}=0$.
Following the same type of procedure of Remark~\ref{rem:nofinerpart}
we now construct explicitly the 
$AB$-entropy solution $u^*$ defined by
\begin{equation}
    \label{eq:indatum-entr-sol-omega2}
        u_0^* \doteq  \mc S^{[AB]-}_T \omega_3,\qquad\quad  u^*(\cdot,t) \doteq \mc S^{[AB]+}_t u^*_0 \qquad \forall~t \in [0,1]\,.
    \end{equation}
Observe that condition~\eqref{eq:ex1-constraint-3}
ensures the existence in $u^*$ of a shock curve 
parametrized by a map 
$\gamma : [\bs\sigma, 1]\to \,]-\infty, 0]$,  
with $\bs\sigma=-L_0/\,\ol A= 1/4$,
    such that $\gamma(\bs\sigma)=0$,
    $\gamma(1)=L_0$.
The curve $t\to (\gamma(t), t)$, $t\in [1/4, 1]$, is a shock curve for the conservation law
    $u_t+f(u)_x=0$, which connects 
    the left states 
    $(\gamma(t)-L_0)/{t}$
    with the right state~$A$.
    On the left of $\gamma(t)$ there is a rarefaction wave connecting the left state 
    $0$ with the right state $\ol A$,
    and centered at the point $(L_0, 0)$.
    \begin{figure}
\centering
\begin{tikzpicture}
\scriptsize{
\draw[->] (-5,0)--(6,0);
\draw[->] (0,0)--(0,2.5)node[above]{$t$};
\draw[very thin] (-5,2)--(6,2);
\draw (-1,2)node[above]{$L$}--(-5,0)node[below]{$5L$};

\draw (-1,2)--(-1,0);
\draw (0,2)--(4,0)node[below]{$-4L$};
\draw (0,2)--(6,0)node[below]{$-p$};

\draw (-4,1)node{$\ol A$};
\draw (5,1)node{$p$};

\draw[dashed] (-5,0)--(-1,2);
\draw[dashed] (-4.5,0)--(-1,2);
\draw [dashed](-4,0)--(-1,2);
\draw[dashed] (-3.5,0)--(-1,2);
\draw[dashed] (-3,0)--(-1,2);
\draw[dashed] (-2.5,0)--(-1,2);
\draw[dashed] (-2,0)--(-1,2);
\draw[dashed] (-1.5,0)--(-1,2);

\draw[smooth, red] plot coordinates{(-1,2)(-0.1,1)(0,0.5)};
\draw[dashed] (0,0.5)--(-1,0);
\draw[dashed] (-1,0)--(-0.8,0.8);
\draw[dashed] (-1,0)--(-0.6,0.6);
\draw[dashed] (-1,0)--(-0.4,0.5);

\draw (-0.2,0.2)node{\tiny{$\ol A$}};
\draw (-0.3,1.5)node{\tiny{$A$}};
\draw (1,0.7)node{\tiny{$\ol B$}};

\draw[dashed] (0,2)--(4.5,0);
\draw[dashed] (0,2)--(5,0);
\draw[dashed] (0,2)--(5.5,0);

}
\end{tikzpicture}
\caption{The solution produced by the initial datum $u_0^*$.}
\label{u_0*part}
\end{figure}
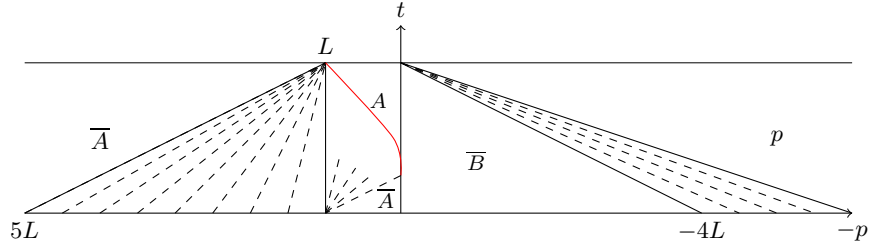
Then, $u^*$ is defined by (see Figure~\ref{u_0*part}) 
    \begin{equation}
    \label{eq:u^*-ex4}
        u^*(x,t)=
        \begin{cases}
            \ \ol A\ \ &\text{if} \quad \ x < L_0-(1-t)\cdot\ol A, \ \ \ t\in [0,1],
            \\
            \noalign{\smallskip}
            \dfrac{L_0-x}{1-t}&\text{if} \quad \ \  L_0-(1-t)\cdot\ol A\leq x \leq  L_0, \ \ \ t\in [0,1[\,,
            \\
            \noalign{\medskip}
            \dfrac{x-L_0}{t}
            \ \ &\text{if}\quad \
            \left\{
            \begin{aligned}
                &L_0<x<\gamma(t), \ \ \ t\in [1/4, 1],
            \\
            \noalign{\smallskip}
            &L_0<x<L_0+t\cdot \ol A,
            \ \ \ t\in \,]0,1/4],
            \end{aligned}
            \right.
            \\
            \noalign{\medskip}
            \ A \ \ &\text{if} \quad\ \ \gamma(t)<x<0, \ \ t\in [1/4, 1],
            \\
            \noalign{\medskip}
            \ \ol A \ \ &\text{if} \quad\  \ L_0+t\cdot \ol A<x<0, \ \ t\in [0,1/4],
            \\
            \noalign{\medskip}
            \ \ol B \ \ &\text{if} \quad\ \ 0<x<(t-1)\cdot \ol B, \ \ t\in [0,1],
            \\
            \noalign{\medskip}
            \dfrac{x}{t-1}
            \ \ &\text{if}\quad \
            (t-1)\cdot \ol B\leq x\leq (t-1)\cdot p,
             \ \ t\in [0,1[\,,
             \\
            \noalign{\medskip}
            \ p  \ \ &\text{if}\quad \ 
            x > (t-1)\cdot p,
             \ \ t\in [0,1]\,,
        \end{cases}
    \end{equation}
and the corresponding initial datum is given by
\begin{equation}
\label{eq:indastum-star-6}
u_0^*(x) = \begin{cases}
\ \ol A & \text{if} \quad x < 5L_0, \\
\noalign{\smallskip}
\ L_0-x & \text{if} \quad 5L_0 < x < L_0, \\
\noalign{\smallskip}
\ \ol A & \text{if} \quad L_0 < x < 0, \\
\noalign{\smallskip}
\ \ol B & \text{if} \quad 0 < x < -4L_0, \\
\noalign{\smallskip}
\ -x & \text{if} \quad -4L_0 < x < -p, \\
\ p & \text{if} \quad -p < x.
\end{cases}
\end{equation}
    Our  goal is to find two initial data $u_{0,1}, u_{0,2} \in \mathcal{I}_T^{[AB]}(\omega_3)$ such that for some $\lambda \in \,]0,1[\,$, 
    we have $\lambda u_{0,1}+ (1-\lambda) u_{0,2} \notin \mathcal{I}_T^{[AB]}(\omega_3)$.
    Then, consider the following two initial data
    (see Figure~\ref{u_01part} and Figure~\ref{u_02part}):
\begin{equation}
\label{eq:indastum-1-6}
u_{0,1}(x) = \begin{cases}
\ \ol A & \text{if} \quad x < L_0, \\
\ A & \text{if} \quad L_0 < x < 0, \\
\ B & \text{if} \quad 0 < x < -\lambda(B,p), \\
\ p & \text{if} \quad -\lambda(B,p) < x,
\end{cases}
\end{equation}
where $\lambda(B,p)=(B+p)/2=(-4L_0+p)/2$ denotes the Rankine-Hugoniot speed of the jump with left state $B$ and right state $p$, 
\begin{equation}
\label{eq:indastum-2-6}
u_{0,2}(x) = \begin{cases}
\ \ol A & \text{if} \quad x < 5L_0, \\
\ L_0-x & \text{if} \quad 5L_0 < x < L_0, \\
\ 2\, \ol A & \text{if} \quad L_0 < x < 0, \\
\ 2\, \ol B & \text{if} \quad 0 < x < -L_0, \\
\ \ol B & \text{if} \quad -L_0 < x < -4L_0, \\
\ -x & \text{if} \quad -4L_0 < x < -p, \\
\ p & \text{if} \quad x>-p.
\end{cases}
\end{equation} 
    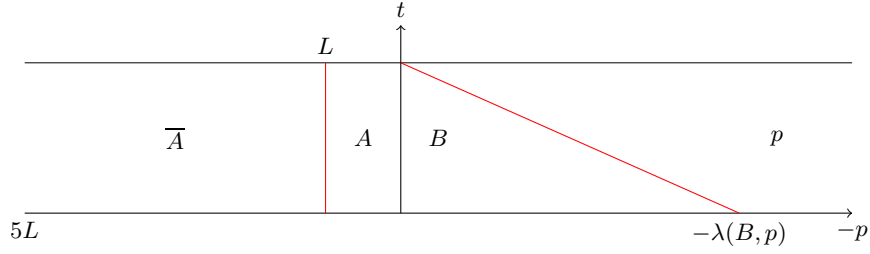
\begin{figure}
\centering
\begin{tikzpicture}
\scriptsize{
\draw[->] (-5,0)--(6,0);
\draw[->] (0,0)--(0,2.5)node[above]{$t$};
\draw[very thin] (-5,2)--(6,2);
\draw (-5,0)node[below]{$5L$};

\draw[red] (-1,2)node[above]{\textcolor{black}{$L$}}--(-1,0);
\draw (6,0)node[below]{$-p$};
\draw (5,1) node{$p$};
\draw (-3,1)node{$\ol A$};
\draw (-0.5,1)node{$A$};
\draw (0.5,1)node{$B$};

\draw[red] (0,2)--(4.5,0)node[below]{\textcolor{black}{$-\lambda(B,p)$}};

}
\end{tikzpicture}
\caption{The solution produced by the initial datum $u_0^1$.}
\label{u_01part}
\end{figure}
\begin{figure}
\centering
\begin{tikzpicture}
\scriptsize{
\draw[->] (-5,0)--(6,0);
\draw[->] (0,0)--(0,2.5)node[above]{$t$};
\draw[very thin] (-5,2)--(6,2);
\draw (-1,2)node[above]{$L$}--(-5,0)node[below]{$5L$};

\draw (-1,2)--(-1,0);
\draw (0,2)--(4,0)node[below]{$-4L$};
\draw (0,2)--(6,0)node[below]{$-p$};

\draw (-4,1)node{$\ol A$};
\draw (5,1)node{$p$};

\draw[dashed] (-5,0)--(-1,2);
\draw[dashed] (-4.5,0)--(-1,2);
\draw [dashed](-4,0)--(-1,2);
\draw[dashed] (-3.5,0)--(-1,2);
\draw[dashed] (-3,0)--(-1,2);
\draw[dashed] (-2.5,0)--(-1,2);
\draw[dashed] (-2,0)--(-1,2);
\draw[dashed] (-1.5,0)--(-1,2);

\draw[smooth, red] plot coordinates{(-1,2)(-0.1,1)(0,0.5)};
\draw (0,0.5)--(-1,0);
\draw[dashed] (-1,0)--(-0.8,0.8);
\draw[dashed] (-1,0)--(-0.6,0.6);
\draw[dashed] (-1,0)--(-0.4,0.5);

\draw (-1,0)--(0,0.25);
\draw(1,0)--(0,0.5);
\draw(1,0)--(0,0.25);

\draw (-0.5,0) node[below]{$2 \ol A$};
\draw (0.5,0) node[below]{$2 \ol B$};

\draw (-0.3,1.5)node{\tiny{$A$}};
\draw (1,0.7)node{\tiny{$\ol B$}};

\draw[dashed] (0,2)--(4.5,0);
\draw[dashed] (0,2)--(5,0);
\draw[dashed] (0,2)--(5.5,0);

}
\end{tikzpicture}
\caption{The solution produced by the initial datum $u_0^2$.}
\label{u_02part}
\end{figure}
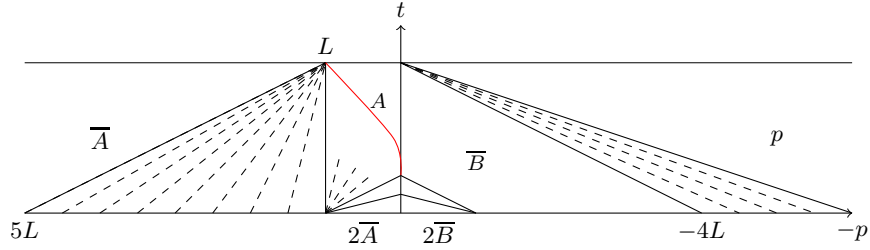

\noindent
With similar arguments as for the construction of $u^*$ above, one can easily see that the $AB$-entropy solutions 
to~\eqref{conslaw},~\eqref{discflux}, with initial
data $u_{0,1}, u_{0,2}$, reach at time $T=1$ the profile~$\omega_3$ in~\eqref{eq:profile-nonconvex-initialset}
(see Figures~\ref{u_01part}, \ref{u_02part}).
Hence, we have $u_{0,i} \in \mathcal{I}_T^{[AB]}(\omega_3)$, $i = 1,2$. 
We will now show that
\begin{equation}
\label{eq:nonconvex-initial-set}
    u^{\lambda}_0 \doteq  \lambda u_{0,1}+ (1-\lambda) u_{0,2}~\notin~\mathcal{I}_T^{[AB]}(\omega_3)
    \qquad\ \forall~\lambda \in \,]0,1[\,.
\end{equation}
Toward this end, we will first show that, if 
$u^{\lambda}_0 \in \mathcal{I}_T^{[AB]}(\omega_3)$
for some $\lambda \in \,]0,1[$,
then there exists $\ol y \in [L_0,-3L_0]$ such that
there holds
\begin{equation}
\label{eq:ineq-utheta-1}
\int_{5L_0}^{\ol y} u_0^{\lambda}(x) \dif x \leq \int_{5L_0}^{\ol y} u_0^*(x)\dif x\,.
\end{equation}
In fact, observe first that with the same analysis
in Remark~\ref{rem:nofinerpart}
we deduce that~\eqref{eq:C0-ex1-1} is verified
also for $u^*$ defined in~\eqref{eq:u^*-ex4}.
Since here we have 
$\bs v=0$, $T=1$, one thus finds that
there holds 
\begin{equation}
    \label{eq:C0-ex1-3}
        \mc C_0(u^*,x)=[\,L_0,\, (x/A-1)\cdot \ol B\,]=
        [\,L_0,\,x-4\,L_0],
        \qquad\forall~x\in\,]\,L_0, 0[\,.
\end{equation}
Then, considering a sequence of points $\ol x_n \downarrow L_0$, and applying Theorem~\ref{initialdataid}
with $u_0^\theta$ in place of $u_0$, we deduce that
for every $n$ there exists $\ol y_n\in \mc C_0(u^*,\ol x_n)=[\,L_0,\,\ol x_n-4\,L_0]$ such that there holds
\begin{equation}
\label{eq:ineq-utheta-3}
\int_{5L_0}^{\ol y_n} u_0^{\lambda}(x) \dif x \leq \int_{5L_0}^{\ol y_n} u_0^*(x)\dif x\,.
\end{equation}
We may assume that, up to a subsequence, $\{\ol y_n\}_n$
converges to some point $\ol y\in [L_0,-3L_0]$.
Then, taking the limit in~\eqref{eq:ineq-utheta-3}
as $n\to\infty$, we derive that~\eqref{eq:ineq-utheta-1} holds for such $\ol y$.

We will now show that, by definitions of $u_0^*$,
$u_{0,i}$, $i=1,2$, in~\eqref{eq:indastum-star-6}, \eqref{eq:indastum-1-6}, \eqref{eq:indastum-2-6},
it follows
\begin{equation}
\label{eq:ineq-utheta-4}
\int_{5L_0}^{\ol y} \big(u_0^{\lambda}(x)- u_0^*(x)\big) \dif x >0, \quad \forall~\ol y \in [L_0,-3L_0]\,,
\end{equation}
which is in contrast with~\eqref{eq:ineq-utheta-1},
thus proving~\eqref{eq:nonconvex-initial-set}
by contradiction.
We distinguish three cases:
%
\smallskip

\noindent {\scshape Case 1.} $\ol y \in [L_0,0]$.
By direct computations we find that
\begin{equation}
\int_{5L_0}^{\ol y} u_{0,1}(x) \dif x = 12L_0^2+4L_0\,\ol y, \qquad \int_{5L_0}^{\ol y} u_{0,2}(x) \dif x  = 16L_0^2-8L_0\, \ol y,
\end{equation}
and 
\begin{equation}
\int_{5L_0}^{\ol y} u_0^*(x) \dif x = 12L_0^2-4L_0\,\ol y\,.
\end{equation}
Thus, for every $\ol y \in [L_0,0]$, we derive
\begin{equation}
\int_{5L_0}^{\ol y} \big(u_0^{\lambda}(x)-u_0^*(x)\big) \dif x = 8\,\lambda L_0\, \ol y+ 8(1-\lambda) L_0(L_0-\ol y) > 0,\qquad\forall~\lambda\in [0,1].
\end{equation}
\smallskip

\noindent {\scshape Case 2.} $\ol y \in [0,-L_0]$.
Observe that, because of~\eqref{eq:v-def}, we have
$\lambda(B,v)>-4L_0$, which implies that $u_{0,1}(x)=B$ for all $x\in\,]0, \ol y]$. Then, by  computations 
as in previous case, for $\ol y \in [0,-L_0]$ we find
\begin{equation}
\int_{5L_0}^{\ol y} u_{0,1}(x) \dif x = 12L_0^2-4L_0\,\ol y, \qquad \int_{5L_0}^{\ol y} u_{0,2}(x) \dif x  = 16L_0^2+8L_0\, \ol y,
\end{equation}
and 
\begin{equation}
\int_{5L_0}^{\ol y} u_0^*(x) \dif x = 12L_0^2 +4L_0\,\ol y\,.
\end{equation}
Thus, for every $\ol y \in [0,-L_0]$, we derive
\begin{equation}
\int_{5L_0}^{\ol y} \big(u_0^{\lambda}(x)-u_0^*(x)\big) \dif x = -8\,\lambda L_0\, \ol y+ 4(1-\lambda) L_0(L_0+\ol y) > 0,\qquad\forall~\lambda\in [0,1].
\end{equation}
\smallskip

\noindent {\scshape Case 3.} $\ol y \in [-L_0,-3L_0]$.
Note that, as in Case 2, we have $u_{0,1}(x)=B$ for all $x\in\,]0, \ol y]$. Then, by  computations 
as in previous cases, for $\ol y \in [-L_0,-3L_0]$ we find
\begin{equation}
\int_{5L_0}^{\ol y} u_{0,1}(x) \dif x = 12L_0^2-4L_0\,\ol y, \qquad \int_{5L_0}^{\ol y} u_{0,2}(x) \dif x  = 12L_0^2+4L_0\, \ol y,
\end{equation}
and 
\begin{equation}
\int_{5L_0}^{\ol y} u_0^*(x) \dif x = 12L_0^2 +4L_0\,\ol y\,.
\end{equation}
Hence, for every $\ol y \in [-L_0, -3L_0]$, we derive
\begin{equation}
\int_{5L_0}^{\ol y} \big(u_0^{\lambda}(x)-u_0^*(x)\big) \dif x = -8\,\lambda L_0\, \ol y
> 0,\qquad\forall~\lambda\in [0,1].
\end{equation}
The analysis of all three cases shows that~\eqref{eq:ineq-utheta-4} is verified, and thus concludes the proof that $\mathcal{I}_T^{[AB]}(\omega_3)$ is not convex since~\eqref{eq:nonconvex-initial-set} holds.

\bibliographystyle{plain}
\bibliography{references.bib}

\end{document}